\pdfoutput=1
\RequirePackage{ifpdf}
\ifpdf 
\documentclass[pdftex]{sigma}
\else
\documentclass{sigma}
\fi

\usepackage{mathtools}

\newtheorem{thm}{Theorem}[section]
\newtheorem{cor}[thm]{Corollary}
\newtheorem{prop}[thm]{Proposition}
\newtheorem{Lemma}[thm]{Lemma}
\newtheorem{Claim}[thm]{Claim}

\theoremstyle{definition}
\newtheorem{ex}[thm]{Example}
\newtheorem{Definition}[thm]{Definition}
\newtheorem{Remark}[thm]{Remark}

\numberwithin{equation}{section}

\newcommand{\C}{\mathbb{C}}

\newcommand{\R}{\mathbb{R}}

\newcommand{\Z}{\mathbb{Z}}

\newcommand{\EE}{\mathbb{E}}

\newcommand{\E}[1]{\mathbb{E}\left [ #1 \right ]}

\begin{document}

\allowdisplaybreaks

\newcommand{\arXivNumber}{2303.13812}

\renewcommand{\PaperNumber}{053}

\FirstPageHeading

\ShortArticleName{Rectangular Matrix Additions in Low and High Temperatures}

\ArticleName{Rectangular Matrix Additions\\ in Low and High Temperatures}

\Author{Jiaming XU}

\AuthorNameForHeading{J.~Xu}

\Address{Columbus, USA}
\Email{\mail{jxu0800@gmail.com}}

\ArticleDates{Received April 09, 2025, in final form May 06, 2026; Published online May 27, 2026}

\Abstract{We study the addition of two independent random $N\times M$ rectangular matrices with invariant distributions in two limiting regimes, where the parameter $\beta$ (inverse temperature) tends to infinity and~$0$. In the low temperature regime the random singular values of the sum concentrate at deterministic points, while in the high temperature regime, we obtain a law of large numbers for the empirical measures. As a consequence, we obtain a~duality between low and high temperatures. Our proof uses the type BC Bessel function as characteristic function of rectangular matrices, and through the analysis of this function we introduce a new family of cumulants, that linearize the addition in the high temperature limit, and degenerate to the classical and free cumulants in special cases.}


\Classification{05E05; 60B10; 60B20}

\section{Introduction}
\subsection{Overview}
Addition is one of the most natural operations on matrices. For deterministic matrices, a~classical question posed by Weyl~\cite{W} in 1912 concerns the eigenvalues $c_{1}\le \cdots \le c_{N}$ of $C = A + B$, where~$A$,~$B$ are two arbitrary self-adjoint $N\times N$ matrices with fixed real eigenvalues~$a_{1}\le \cdots \le a_{N}$ and $b_{1}\le \cdots \le b_{N}$. The problem is to describe all possible values of $c_{1}\le \cdots \le c_{N}$. This question was solved by the end of the 20th century through the combined efforts of Horn, Klyachko, Knutson--Tao, and others; see, for example, \cite{Ho, Kl,KT}.

In random matrix theory, one usually assumes that the summands $A$ and $B$ are random, independent, and satisfy certain symmetry conditions. The study of such questions has significant connections with free probability theory.

A well-known classical result connecting random matrix addition and free probability is due to Voiculescu~\cite{Vo}, who considered the sum of two independent real/complex/quaternionic self-adjoint matrices and related its asymptotic eigenvalue distribution, as the matrix size tends to infinity, to the notion of \emph{free convolution}. There is also a classical result of a similar flavor in the rectangular setting, stated as follows. Let $\{A_{N}\}_{N=1}^{\infty}$ and $\{B_{N}\}_{N=1}^{\infty}$ be two independent sequences of $N\times M$ matrices ($M\ge N$) with real/complex/quaternionic entries, each uniformly chosen from the set of rectangular matrices with prescribed singular values $a_{N,1}\ge\cdots\ge a_{N,N}\ge 0$ and~${b_{N,1}\ge\cdots\ge b_{N,N}\ge 0}$, respectively. Denote $C_{N}=A_{N}+B_{N}$, and let $c_{N,1}\ge\cdots\ge c_{N,N}\ge 0$ be its (random) singular values.

\begin{Definition}\label{def:empiricalmeas}
 For an $N\times M$ ($M\ge N$) matrix $A$ with singular values $a_{1},\dots,a_{N}\ge 0$, define its (symmetric) empirical measure to be
 $\mu_{A}:=\frac{1}{2N}\sum_{i=1}^{N}(\delta_{a_{i}}+\delta_{-a_{i}})$.
\end{Definition}

\begin{thm}[{\cite[Proposition~2.1]{B1}}]\label{thm:recfreeconvolution}
Define $\{A_{N}\}_{N=1}^{\infty}$, $\{B_{N}\}_{N=1}^{\infty}$ as above. Assume that $N,M\rightarrow \infty$ in a way that $M(N)/N\rightarrow q$ for some constant\footnote{$q=\lambda^{-1}$ in the notation of~\cite{B1} and the other references on the rectangular free convolution. } $q\ge 1$, and there exist deterministic probability measures $\mu_{A}$, $\mu_{B}$ on $\R$, such that
\[
\lim_{N\rightarrow \infty}\mu_{A_{N}}=\mu_{A},\qquad \lim_{N\rightarrow \infty}\mu_{B_{N}}=\mu_{B}.
\]
Then the random empirical measure of $C_{N}=A_{N}+B_{N}$, $\mu_{C_{N}}=\frac{1}{2N}\sum_{i=1}^{N}(\delta_{c_{N,i}}+\delta_{-c_{N,i}})$, converges weakly in probability to some deterministic probability measure $\mu_{C}$ on $\R$.

$\mu_{C}=\mu_{A}\boxplus_{q} \mu_{B}$ is called the \emph{rectangular free convolution} of $\mu_{A}$ and $\mu_{B}$.
\end{thm}

The rectangular free convolution is a deterministic binary operation of symmetric\footnote{$\mu(A)=\mu(-A)$ for all Borel subsets $A$ of $\R$.} probability measures on $\R$, which itself does not rely on any random matrix structure, and it has been well-studied in free probability theory from different aspects. In particular, for each measure~$\mu$ with finite moments, there exists a collection of \emph{rectangular free cumulants} $\{c^{q}_{l}\}_{l=1}^{\infty}$ (see \cite[Section~3.1]{B1}) that are polynomials of moments with explicit expressions, and these quantities linearize free convolution, i.e., $c^{q}_{l}(\mu_{A}\boxplus \mu_{B})=c^{q}_{l}(\mu_{A})+c^{q}_{l}(\mu_{B})$ for all $l$. It turns out that the existence of such cumulants is a common feature of the various versions of convolutions in free probability theory, and each convolution is characterized by its corresponding cumulants.

There have also been many papers studying additions of $\beta$-random matrices that generalize the above theory in different parameter regimes. The parameter $\beta>0$ is interpreted in physics as inverse temperature, and the cases $\beta=1,2,4$ correspond to matrices with real/complex/real quaternionic entries. There are two classes of matrix ensembles, the $N\times N$ self-adjoint matrices and the $N\times M$ rectangular matrices, and the most classical examples are Gaussian ensembles and Laguerre ensembles, respectively. For the first class, \cite{GM}~studied the limit behavior of eigenvalues of $C=A+B$ when $N$ is fixed and $\beta\rightarrow\infty$. In the physics literature, \cite{MeP}~considered the limit behavior of the additions
when $\beta\to0$, and shortly thereafter~\cite{BCG} proved a law of large numbers similar to Theorem~\ref{thm:recfreeconvolution} in the regime $N\rightarrow\infty$, $\beta\rightarrow0$, $N\beta/2\rightarrow \gamma>0$. In another direction, \cite{AP,GM,MSS} developed the theory of convolution and cumulants for additions of finite dimensional $\beta$-random matrices, which is known as \emph{finite free probability}. The second class is relatively less understood. Papers \cite{B1, B2} studied the rectangular matrix additions for $\beta=1,2$, $N,M\rightarrow\infty$ and $M/N\rightarrow q\ge 1$, and~\cite{Gri,GrM} studied the finite free convolution and cumulants for rectangular matrix additions for $\beta=1,2$. After the first version of this paper was released, the law of large numbers for $\beta$-matrix additions with fixed $\beta>0$ was proved independently in~\cite{Y} and~\cite{CX}; the former treated both the self-adjoint and rectangular settings.

The matrix ensemble considered in this paper belongs to the second class, and we study the limiting behavior of singular values of $C=A+B$ in both low and high temperature regimes, more precisely, when $N$, $M$ are fixed, $\beta\rightarrow\infty$, and when $ N,M\rightarrow\infty$, $\beta\rightarrow0$, $N\beta/2\rightarrow \gamma>0$, $M\beta/2\rightarrow q\gamma$ for some $q\ge 1$. Note that even defining the operation $C=A+B$ for $\beta\ne 1,2,4$ is nontrivial, and this is one of our tasks.

Our approach is based on distributions of rectangular matrices in a version of characteristic function. The symmetry of self-adjoint/rectangular matrices with fixed eigenvalues/singular values is given by invariance under actions of classical Lie groups ${\rm O}(N)/{\rm U}(N)/{\rm Sp}(N)$, and when $\beta=1,2,4$, the matrix characteristic functions have matrix integral representations with representation-theoretic background. Such functions admit an analytic continuation to all ${(\beta>0)}$, and can be identified as eigenfunctions of certain differential operators. See the following papers~\cite{BCG, BoG, GM, GS} for applications of this idea in random matrices, and also~\cite{BG1,BG2,H} for the study of more general $N$-particle system using symmetric characteristic functions of similar flavor. While the above works deal with self-adjoint matrices, or more generally $N$-particle systems that correspond to a root system of type A with a single root multiplicity~${\theta=\beta/2}$, rectangular matrices correspond to a root system of type BC, that has two distinct root multiplicities parameterized by $\beta$ in a more involved way. For more connections of type BC Lie-theoretic objects with probability, see, e.g., \cite{KVW, V, VW}.

In this paper, the randomness of an $N$-tuple of nonnegative real numbers (to be thought of as the singular values of some $N\times M$ random matrices) is encoded by a multivariate symmetric function, known in the special functions literature as the type BC Bessel function. It is also a special case of the symmetric Dunkl kernel, which generalizes the usual Fourier kernel to nontrivial root multiplicities; see~\cite{A} for a review. Motivated by the asymptotic behavior in the high-temperature regime, we adopt and further develop the philosophy that the limits of the partial derivatives at $0$ of the logarithm of our characteristic function yield a collection of cumulants, and that the existence of such cumulants is equivalent to the existence of limiting moments. This, in turn, implies that the empirical measures of the random singular values satisfy a~law of large numbers. These new $q$-$\gamma$ cumulants are designed to linearize rectangular addition in the regime $N\beta/2\to \gamma$, $M\beta/2\to q\gamma$, and the corresponding limiting operation is called $q$-$\gamma$ convolution. Similar to classical and free cumulants, they enjoy pleasant combinatorial relations with moments. Finally, we point out a surprising identification of $q$-$\gamma$ theory with rectangular finite free probability, which was developed in~\cite{Gri, GrM,MSS} in the study of finite rectangular matrix additions.

\subsection{Rectangular matrix addition}
Throughout the paper, we always take\footnote{The parameter $\beta$ is standard in the random matrix context, while $\theta$ is more commonly used in the special function literature. We follow this convention and mostly use $\theta$, except for the random matrix terminologies like ``$\beta$-ensembles".} $\beta=2\theta>0$, and $\beta=1, 2, 4 \ \big(\theta=\frac{1}{2},1,2$\big) correspond to the real, complex and quaternionic (skew) fields (whose real dimensions are given by $\beta$). For~${M\ge N}$, given two $N\times M$ independent random matrices $A$ and $B$, we study the randomness of the sum~${C=A+B}$.

Inspired by the classical theory of adding independent random variables $X+Y$, namely, that the characteristic function satisfies
$\phi_{X+Y}(t)=\phi_{X}(t)\cdot\phi_{Y}(t)$,
where $t\in \R$ is the variable, we have the following.

\begin{prop}\label{prop:characteristic1}
For $\theta=\frac{1}{2},1,2$, let $A$ and $B$ be $N\times M$ independent random rectangular matrices, let $X$ be $M\times N$ an arbitrary deterministic matrix with real/complex/real quaternionic entries, and let $C=A+B$. We have
\begin{equation}\label{eq_matrixfouriertransform}
\EE [\exp(\operatorname{Re}(\operatorname{Tr}(CX)))]=\EE[\exp(\operatorname{Re}(\operatorname{Tr}(AX)))]\cdot\EE[\exp(\operatorname{Re}(\operatorname{Tr}(BX)))].
\end{equation}
\end{prop}
\begin{proof}
 $\operatorname{Re}(\operatorname{Tr}(CX))=\operatorname{Re}(\operatorname{Tr}(AX))+\operatorname{Re}(\operatorname{Tr}(BX)),$ and since $A$, $B$ are independent, the expectation of the exponential function factors.
\end{proof}

Let us now rewrite~\eqref{eq_matrixfouriertransform} in terms of singular values of $A$, $B$ and $C$, and for simplicity first take $\theta =1$, i.e., we deal with complex matrices. In this paper, we are considering the summands~$A$ and $B$ that have distributions invariant under left and right unitary actions, i.e.,
\begin{equation}\label{eq_invariance}
A\stackrel{d}{=} UAV, \qquad B\stackrel{d}{=} UBV,
\end{equation}
where $U\in {\rm U}(N)$, $V\in {\rm U}(M)$ are arbitrary unitary matrices. One example is a real/complex/real quaternionic $N\times M$ matrix with i.i.d.\ mean 0 Gaussian entries.

 Note that if $A$, $B$ satisfy~\eqref{eq_invariance}, so does $C$. For simplicity, in the following discussion we usually focus on the matrix $A$. By singular value decomposition, it is useful to write $A$ as $U\Lambda V$, where
 \begin{gather*}
 \Lambda=\begin{bmatrix}
a_{1} & & & &&0&\dots &0\\
 & a_{2} & & &&0&\dots & 0\\
 & &\dots & &&\\
 &&&\dots &\\
 &&&&a_{N} &0&\dots &0
\end{bmatrix}_{N\times M},
\end{gather*}
$\Vec{a}=(a_{1},\dots,a_{N})\in \R_{\ge 0}$, and $U\in {\rm U}(N)$, $V\in {\rm U}(M)$ are random elements under Haar measures on the corresponding unitary groups. For now, assume that the singular values of $A$ are deterministic. One can consider rectangular matrices $A$ with real/real quaternionic entries and define invariant distribution in exactly the same way, while replacing the Haar distributed~$U$,~$V$ by elements in orthogonal/unitary symplectic groups ${\rm O}(N)/{\rm Sp}(N)$, ${\rm O}(M)/{\rm Sp}(M)$. Similarly for~$B$.

The eigenvectors of $AA^{*}$, $A^{*}A$, $BB^{*}$, $B^{*}B$ are distributed uniformly, and so are the eigenvectors of $CC^{*}$, $C^{*}C$. Therefore, the nontrivial randomness of $C$ is about its singular values. Also, by the singular value decomposition in~\eqref{eq_invariance}, it suffices to take the parameter matrix $X$ of the form \[
 X=\begin{bmatrix}
x_{1} & & & &\\
 & x_{2} & & &\\
 & &\dots &\\
 &&&\dots &\\
 &&&&x_{N} &\\
 0&&\dots &&0\\
 && \dots &&\\
 0&&\dots &&0
\end{bmatrix}_{M\times N},\]
 where $x_{1},\dots,x_{N}\in \R$. Therefore, we can rewrite the matrix Fourier transform of $A$ in~\eqref{eq_matrixfouriertransform} as a function $B(\vec{a},x_{1},\dots,x_{N};\theta,N,M)$ for $\theta=1$, such that
\begin{equation}\label{eq_matrixint}
 B(\vec{a},x_1,x_2,\dots,x_N;1, N,M)=\int {\rm d}U\int {\rm d}V \exp\left(\frac{1}{2}\operatorname{Tr}(U\Lambda VX+X^{*}V^{*}\Lambda^{*}U^{*})\right),
\end{equation}
where $U\in {\rm U}(N)$, $V\in {\rm U}(M)$ are integrated under Haar measures respectively. A similar argument defines $B(\vec{a},x_1,x_2,\dots,x_N;\theta, N,M)$ for $\theta=\frac{1}{2}$ and $2$, where one replaces the unitary group by the orthogonal/unitary symplectic group, respectively.

The function $B(\vec{a},x_{1},\dots,x_{N};\theta,N,M)$ is known as multivariate type BC Bessel function in a general theoretical framework of special functions. The note~\cite{Xu} gives a brief outline of the framework, while in this paper, we
give a more concrete definition of $B(\vec{a},x_{1},\dots,x_{N};\theta,N,M)$, and state its properties needed later in
Sections~\ref{sec:typebc} and \ref{sec:dunkl}.

We note that because of the symmetry of Haar measure, $B(\vec{a},x_{1},\dots,x_{N};\theta,N,M)$ is symmetric in both $(a_{1},\dots,a_{N})$ and $(x_{1},\dots,x_{N})$, and without loss of generality we can always take~${a_{1}\ge a_{2}\ge \dots \ge a_{N}}$. The same holds for matrix $B$ and $C$.
Then Proposition~\ref{prop:characteristic1} is rewritten as the following result.
\begin{prop}\label{prop:characteristic2}
For $\theta=\frac{1}{2},1,2$, fix $a_{1}\ge\dots \ge a_{N}\ge 0$, $b_{1}\ge\dots \ge b_{N}\ge 0$, let
$A_{N\times M}$ and $B_{N\times M}$ be real/complex/real quaternionic rectangular matrices with deterministic singular values $\{a_{i}\}_{i=1}^{N}$, $\{b_{i}\}_{i=1}^{N}$ and invariant distribution, as in {\rm\eqref{eq_invariance}}. Let the singular values of~${C=A+B}$ be $\Vec{c}=(c_{1}\ge \dots \ge c_{N}\ge 0)$,
then as a function on $\R^{N}$, the expected value of~${B(\Vec{c};x_{1},\dots,x_{N};\theta,N,M)}$ is given by
\begin{gather*}
\E{B(\Vec{c};x_{1},\dots,x_{N};\theta,N,M)}=
B(\vec{a},x_{1},\dots,x_{N};\theta,N,M)\cdot B\big(\Vec{b};x_{1},\dots,x_{N};\theta,N,M\big).
\end{gather*}
\end{prop}

For general $\beta>0$, there is no (skew) field of real dimension $\beta$, and therefore no concrete $\beta$-rectangular matrices. Motivated by Proposition~\ref{prop:characteristic2}, we first identify an invariant $N\times M$ matrix with uniform ``singular vectors'' and deterministic singular values $a_{1},\dots,a_{N}$ with the $N$-tuple $\vec{a}$. Moreover, it is known (see, e.g., \cite[Section~13.4.3]{Forrester} and~\cite{GT}) that multivariate Bessel functions admit a natural extrapolation from $\theta=\tfrac12,1,2$ to arbitrary real $\theta>0$, and we continue to denote the resulting function by
$
B(\vec{a},x_{1},\dots,x_{N};\theta,N,M)$.
In this way, $N\times M$ random matrix addition is extended to all $\theta>0$ by generalizing Proposition~\ref{prop:characteristic2}.
\begin{Definition}\label{def:deterministicaddition}
Fix $\theta>0$, $M\ge N$, $\Vec{a}=(a_{1}\ge\dots \ge a_{N}\ge 0)$, $\Vec{b}=(b_{1}\ge\dots \ge b_{N}\ge 0)$. Let $\Vec{c}$ be a symmetric random vector in $\R_{\ge 0}^{N}$, such that as a function on $\R^{N}$,
\begin{gather}\label{eq_additioindef}
\EE[B(\Vec{c};x_{1},\dots,x_{N};\theta,N,M)]=
B(\vec{a},x_{1},\dots,x_{N};\theta,N,M)\cdot B\big(\Vec{b};x_{1},\dots,x_{N};\theta,N,M\big).
\end{gather}
We write \smash{$
\Vec{c}=\Vec{a}\boxplus_{N,M}^{\theta}\Vec{b}$}.
\end{Definition}

From the probabilistic point of view, $\vec{c}$ is identified with the singular values of the “virtual” random $N\times M$ matrix $C = A\boxplus^{\theta}_{N,M} B$, and $B(\vec{c};x_{1},\dots,x_{N};\theta,N,M)$ serves as the characteristic function of $C$. From the analytic point of view, the operation in~\eqref{eq_additioindef} has been studied previously in the context of the Dunkl kernel and Dunkl translation; see \cite[Section~3.6]{A}.

The expectation symbol on the left-hand side of~\eqref{eq_additioindef} is understood in the following sense: there exists a unique generalized function\footnote{Throughout this paper, we use the term “generalized function” instead of “distribution” to denote a linear functional on smooth functions, in order to avoid confusion with probability distributions.} $\mathfrak m$ on \smash{$\mathbb R_{\ge 0}^{N}$}, depending on $\vec{a}$ and $\vec{b}$, such that for any~${(x_{1},\dots,x_{N})\in\mathbb R^{N}}$, testing on $B(\cdot;x_{1},\dots,x_{N};\theta,N,M)$ yields the right-hand side of~\eqref{eq_additioindef}, and in particular, taking $x_{1}=\cdots=x_{N}=0$ gives $\mathfrak m(1)=1$.

Note that $\mathfrak m$ is symmetric in the sense that, for any suitable test function $f$ and any permutation $\sigma$, $ \langle \mathfrak m, f(c_{1},\dots,c_{N})\rangle
=\langle \mathfrak m, f(c_{\sigma(1)},\dots,c_{\sigma(N)})\rangle$,
where $\langle \mathfrak m, f\rangle$ denotes the value of the functional $\mathfrak m$ on $f$. Moreover, by \cite[Lemma~3.23]{A}, the generalized function $\mathfrak m$ is compactly supported.\looseness=-1

The rectangular addition \smash{$\Vec{a}\boxplus_{N,M}^{\theta}\Vec{b}$} can also be naturally generalized to independent random $N$-tuples $\Vec{a}$, $\Vec{b}$, by first conditioning on the value of $\vec{a}$ and $ \vec{b}$, then applying Definition~\ref{def:deterministicaddition}. Formally, for random $N$-tuple $\Vec{a}$ we replace the type BC Bessel function by
\begin{equation}\label{eq_bgf0}
G_{N;\theta}(x_{1},\dots,x_{N}):=\EE [B(\Vec{a},x_{1},\dots,x_{N};\theta,N,M)],
\end{equation}
the type BC Bessel generating function of $\Vec{a}$, and we assume the randomness of $\Vec{a}$ to be reasonable, in the sense that the right side of~\eqref{eq_bgf0} is finite and well-behaved as an analytic function of~${(x_{1},\dots,x_{N})\in \R_{N}}$. See Section~\ref{sec:bgf} for more details.

\subsection{Low and high temperature behavior}
By viewing \smash{$\Vec{c}=\Vec{a}\boxplus_{N,M}^{\theta}\Vec{b}$} as the random $N$-tuple of singular values of some $N\times M$ virtual rectangular matrix with invariant distribution, it is then natural to study the behavior of $\Vec{c}$ from a random matrix point of view. The distribution of $\Vec{c}$ depends on summands $\Vec{a}$, $\Vec{b}$ and parameters~${M\ge N}$, $\theta>0$. This paper answers the following two questions:
\begin{itemize}\itemsep=0pt
\item[(1)] What is the ``low temperature'' behavior of $\Vec{c}$, i.e., when taking $N$,~$M$ to be fixed, and $\theta\rightarrow\infty$?
\item[(2)] What is the ``high temperature'' behavior of $\Vec{c}$, i.e., when taking $\theta\rightarrow 0$, and $N,M\rightarrow \infty$, growing at potentially different speeds?
\end{itemize}
The solutions of these two questions can be found in Sections~\ref{sec:lowtemp} and~\ref{sec:lln}, \ref{sec:momentcumulant}, respectively. We remark that for $\theta\ne \frac{1}{2},1,2$, it is not yet verified that the generalized function under $\Vec{c}$ is a probability measure, see, e.g., \cite[Section~3.5]{A}.
We do not rely on the validity of this statement, and instead analyze moments of the distribution of $\Vec{c}$, which can be defined no matter the positivity conjecture holds or not. See Proposition~\ref{prop:polyexpectation} for the precise statement.

In the low temperature regime, we observe that the random $N$-tuple are becoming ``frozen'' at some deterministic positions. More precisely we have the following statement.
Let $1\le N\le M$, $z$ be a formal variable, $(a_{1},\dots,a_{N}), (b_{1},\dots,b_{N})\in \R^{N}_{\ge 0}$, we define a polynomial $P_{N,M}(z)$ by
\begin{align}
P_{N,M}(z)={}&\sum_{l=0}^{N}(-1)^{l}\Biggl(\sum_{i\ge 0,\, j\ge 0,\, i+j=l}\frac{(N-i)!(N-j)!}{N!(N-l)!}\nonumber\\
&\hphantom{\sum_{l=0}^{N}(-1)^{l}\Biggl(}{} \times\frac{(M-i)!(M-j)!}{M!(M-l)!}
e_{i}\big(a_{1}^{2},\dots,a_{N}^{2}\big)e_{j}\big(b_{1}^{2},\dots,b_{N}^{2}\big)\Biggr)z^{N-l}.\label{eq_charpoly}
\end{align}

\begin{thm}\label{thm:lln}
Fix $M\ge N$, given $\Vec{a}$ and $\Vec{b}$, let \smash{$\Vec{c}=\Vec{a}\boxplus_{N,M}^{\theta}\Vec{b}$}. Then as
 $\theta\rightarrow \infty$, the distribution of \smash{$\Vec{c^{2}}=\bigl(c_{1}^{2},\dots,c_{N}^{2}\bigr)$} converges on polynomial test functions to $\delta$-measures on roots of $P_{N,M}(z)$.
\end{thm}
\begin{Remark}
 The polynomial $P_{N,M}(z)$ has previously appeared in~\cite{GrM}, and it is shown in \cite[Theorem 2.3]{GrM}, using the theory of stable polynomials, that all roots of $P_{N,M}(z)$ are real and nonnegative.
\end{Remark}

In the high temperature regime, when taking $\theta\rightarrow 0$, $M\rightarrow\infty$ and fixing $N$, the number of singular values, the type BC multivariate Bessel function becomes a simple symmetric combination of exponential functions
\begin{equation*}
 B(\Vec{a},M\theta x_{1},\dots,M\theta x_{N};\theta, N,M)\longrightarrow \frac{1}{N!}\sum_{\sigma\in S_{N}}\prod_{i=1}^{N}{\rm e}^{a_{i}^{2}x_{\sigma(i)}^{2}},
\end{equation*}
where $S_{N}$ denotes the symmetric group of degree $N$. See the appendix for more details. The limit expression has a clear probabilistic interpretation. Given deterministic $N$-tuple $\Vec{a}(N)$ and~$\Vec{b}(N)$ as before, $\Vec{c}(N)=(c_{1},\dots,c_{N}\ge 0)$ is obtained by choosing an uniformly random element $\sigma$ in~$S_{N}$, and taking
\smash{$
 \bigl(c_{1}^{2},\dots,c_{N}^{2}\bigr)=\bigl(a_{1}^{2}+b_{\sigma(1)}^{2},\dots,a_{N}^{2}+b_{\sigma(N)}^{2}\bigr)$}.
 Taking $N\rightarrow \infty$, and assume that the empirical measures of $\{\Vec{a}(N)\}_{N=1}^{\infty}$ and $\bigl\{\Vec{b}(N)\bigr\}_{N=1}^{\infty}$
 \[\frac{1}{N}\sum_{i=1}^{N}\delta_{x_{i}^{2}}\qquad (x_{i}=a_{N,i} \text{ or } b_{N,i})\]
 converge weakly to some probability measure $\mu_{a}$, $\mu_{b}$ on $\R_{\ge 0}$, then so does $\{\Vec{c}(N)\}_{N=1}^{\infty}$. Moreover, the limiting measure
 $\mu_{c}=\mu_{a}*\mu_{b}$,
 where $*$ denotes the usual convolution of measures on $\R$.

We observe two distinct limiting behaviors of $\vec{a}\boxplus_{N,M}^{\theta}\vec{b}$ as $N\to\infty$: for $\theta=0$ we obtain the usual convolution, whereas for $\theta>0$ we obtain the rectangular free convolution. This motivates us to consider an intermediate regime in which
$
N\to\infty$, $ M\to\infty$, $ \theta\to 0$, $ N\theta\to\gamma>0$, $ M\theta\to q\gamma$
for some $q\ge 1$. We then study the sequence of (random) virtual singular values
\[
{\vec{c}(N)=(c_{N,1}\ge\cdots\ge c_{N,N}\ge 0)}_{N=1}^{\infty}
\]
and the limiting behavior of the associated symmetric empirical measures.

\begin{Remark}
 The same intermediate regime was considered in~\cite{ABMV}, where the authors studied the limiting behavior of the Laguerre ensemble and proved that its empirical measure converges in this regime to a deterministic probability measure $\mu_{q,\gamma}$, which interpolates between the Marchenko--Pastur law (as $\gamma\to\infty$) and a certain Gamma distribution (as $\gamma\to 0$). As an application of the theory developed in this paper, we rederive this result in Section~\ref{sec:laguerre}. Analogous results for the Gaussian ensemble can be found in~\cite{ABG}.
\end{Remark}

\begin{Definition}\label{def:llnsatisfaction0}
 Let $\{\vec{a}(N)\}_{N=1}^{\infty}$ be a sequence of random $N$-tuple such that $\Vec{a}(N)=(a_{N,1}\ge \dots \ge a_{N,N}\ge 0)$. Denote
 \[
 p^{N}_{k}=\frac{1}{2N}\sum_{i=1}^{N}\bigl[a_{N,i}^{k}+(-a_{N,i})^{k}\bigr].
 \]
 We say $\{\Vec{a}(N)\}$ converges in moments, if there exist deterministic nonnegative real numbers $\{m_{k}\}_{k=1}^{\infty}$ such that for any $s=1,2,\dots$ and any $k_{1},\dots,k_{s}\in \Z_{\ge 1}$, we have
 \[
 \lim_{N\rightarrow \infty}\E{\prod_{i=1}^{s}p_{k_{i}}^{N}}=\prod_{i=1}^{s}m_{k_{i}}.
 \]
 We write
 \[
 \Vec{a}(N)\xrightarrow[N\rightarrow \infty]{m} \{m_{k}\}_{k=1}^{\infty}.
 \]
\end{Definition}
\begin{Remark}
 Definition~\ref{def:llnsatisfaction0} can be interpreted as saying that the empirical measure of~$(a_{N,1},\allowbreak\dots,a_{N,N})$ converges weakly to a deterministic probability measure with moments $\{m_{k}\}_{k=1}^{\infty}$, provided that the moment problem for $\{m_{k}\}_{k=1}^{\infty}$ has a unique solution. By definition, $p^{N}_{k}=0$ for all odd $k$, and hence $m_{k}=0$ for all odd $k$ as well. We work with symmetric empirical measures because there is no canonical choice of sign for the singular values.
\end{Remark}
\begin{Remark}
 The convergence is well-posed as long as the randomness of the sequence $\Vec{a}(N)$ are given by compactly supported generalized function, where the expectation $\EE$ is understood as testing the generalized function against the
polynomial function $p^{N}_{k}$ of $\Vec{a}(N)$.

\end{Remark}

We prove a law of large numbers for the symmetric empirical measure of $\vec{c}(N)$, interpreted as the empirical measure of the $N\times M$ matrix $C$ with singular values $c_{N,1},\dots,c_{N,N}$. We assume that the distribution of each $\vec{a}(N)$ and \smash{$\vec{b}(N)$} is given either by a real-valued, compactly supported generalized function or by an exponentially decaying measure. For the precise meaning of the latter notion and further details on this technical point, see Section~\ref{sec:bgf}.
\begin{thm}\label{thm:hightemplln}
 Fix $\gamma>0, q\ge 1$. For $M=1,2,\dots $, let $M(N)\ge N$, $\theta(N)>0$ be two sequences satisfying $N\rightarrow\infty$, $\theta\rightarrow 0$, $N\theta\rightarrow \gamma$, $M\theta\rightarrow q\gamma$ as $N\rightarrow \infty$. Suppose for two sequences of random tuple $\{\vec{a}(N)\}_{N=1}^{\infty}$, $\bigl\{\Vec{b}(N)\bigr\}_{N=1}^{\infty}$,
 \[\Vec{a}(N)\xrightarrow[N\rightarrow \infty]{m}\{m^{a}_{k}\}_{k=1}^{\infty},\qquad \Vec{b}(N)\xrightarrow[N\rightarrow \infty]{m}\bigl\{m^{b}_{k}\bigr\}_{k=1}^{\infty}.\]
 Then
 \[\Vec{a}(N)\boxplus^{\theta}_{N,M}\Vec{b}(N)\xrightarrow[N\rightarrow\infty]{m}\{m^{c}_{k}\}_{k=1}^{\infty},\]
 where \smash{$\{m^{c}_{k}\}_{k=1}^{\infty}$} is a sequence of nonnegative deterministic real numbers.

 We say $\{m^{c}_{k}\}_{k=1}^{\infty}$ is the $q$-$\gamma$ convolution of $\{m^{a}_{k}\}_{k=1}^{\infty}$ and \smash{$\bigl\{m^{b}_{k}\bigr\}_{k=1}^{\infty}$}, written as
 \[\{m^{c}_{k}\}_{k=1}^{\infty}=\{m^{a}_{k}\}_{k=1}^{\infty}\boxplus_{q,\gamma}\bigl\{m^{b}_{k}\bigr\}_{k=1}^{\infty}.\]
\end{thm}

The $q$-$\gamma$ convolution appears as a new member of the family of convolutions, each being a~binary operation on moment sequences $\{m_{k}\}_{k=1}^{\infty}$. Beyond introducing it in the context of high-temperature $\beta$-additions, we provide an explicit characterization in terms of $q$-$\gamma$ cumulants and establish its connections to classical convolutions.

\begin{thm}\label{thm:momentcumulantsummary}
There exists an invertible map $\mathrm{T}^{q,\gamma}_{m\rightarrow \kappa}\colon \R^{\infty}\rightarrow \R^{\infty}$, that corresponds each $\{m_{2k}\}_{k=1}^{\infty}$ with a collection of \emph{$q$-$\gamma$ cumulants} $\{\kappa_{l}\}_{l=1}^{\infty}$, i.e., $\{\kappa_{l}\}_{l=1}^{\infty}=\mathrm{T}^{q,\gamma}_{m\rightarrow \kappa}(\{m_{2k}\}_{k=1}^{\infty})$. The $q$-$\gamma$ cumulants linearize $q$-$\gamma$ convolution: for $l=1,2,\dots $,
\[\kappa_{l}(\{m^{c}_{2k}\}_{k=1}^{\infty})=\kappa_{l}(\{m^{a}_{2k}\}_{k=1}^{\infty})+\kappa_{l}\bigl(\bigl\{m^{b}_{2k}\bigr\}_{k=1}^{\infty}\bigr).\]
$\kappa_{l}=0$ for all odd $l$. Also $m_{k}^{a}$, $m_{k}^{b}$, $m_{k}^{c}$ are 0 for all odd $k$.

Treating each $\kappa_{l}$ as a variable of degree $l$, then each $m_{2k}$ is a homogeneous polynomial in the~$\kappa_{l}$ of degree $2k$, whose coefficients are polynomials of $q$, $\gamma$ with explicit combinatorial description.
Conversely, treating $m_{2k}$ as a variable of degree $2k$, each even $q$-$\gamma$ cumulant $\kappa_{2l}$ is a homogeneous polynomial in the $m_{2k}$ of degree $2l$.
\end{thm}
\begin{thm}\label{thm:degenerationsummary}
When $\gamma\rightarrow 0, q\gamma\rightarrow \infty$, the $q$-$\gamma$ convolution of $\{m_{k}^{a}\}_{k=1}^{\infty}$ and $\bigl\{m_{k}^{b}\bigr\}_{k=1}^{\infty}$ converges to the usual convolution of the two corresponding independent random variables, and the $q$-$\gamma$ cumulants of $\{m_{2k}^{a}\}_{k=1}^{\infty}$, $\bigl\{m_{2k}^{b}\bigr\}_{k=1}^{\infty}$ converge to the usual cumulants after proper rescaling.
Similarly, when $q$ is fixed, $\gamma\rightarrow\infty$, the $q$-$\gamma$ convolution converges to the rectangular free convolution, and the $q$-$\gamma$ cumulants converge to rectangular free cumulants after proper rescaling.

\end{thm}
Theorem~\ref{thm:hightemplln} is proved in Section~\ref{sec:lln}. Theorem~\ref{thm:momentcumulantsummary} summarizes the results of Sections~\ref{sec:cumulanttomoment} and~\ref{sec:momenttocumulant}: the combinatorial moment–cumulant formula is stated in Theorem~\ref{thm:cumulanttomomentcomb}, and the relation between the moment generating function and the cumulant generating function is given in Theorem~\ref{thm:momenttocumulant}. Theorem~\ref{thm:degenerationsummary} summarizes the connections between $q$-$\gamma$ convolution and classical and free convolution, proved in Theorems~\ref{thm:usualcumulant} and~\ref{thm:connectiontofreecumulant}, respectively. We also provide, in Theorem~\ref{thm:gammacumulant}, a limit transition from our $q$-$\gamma$ convolution to the $\gamma$-convolution defined in~\cite{BCG}, which is related to the asymptotic behavior of self-adjoint matrix additions in the high-temperature regime.

\subsection{A quantitative match between low and high temperatures}
It has been observed in $\beta$-random matrix theory that there is a duality between the parameters $\beta$ and $\tfrac{4}{\beta}$ \big(or, equivalently, between $\theta$ and $\tfrac{1}{\theta}$\big). For example, in~\cite{De} the author establishes an identity between averaged products of characteristic polynomials of Gaussian and chiral $\beta$-ensembles at parameters $\beta$ and $\tfrac{4}{\beta}$. Similarly, in~\cite{F2} it is shown that, for Gaussian, Laguerre, and Jacobi $\beta$-ensembles, the one-point and higher-point correlation functions describing linear statistics of eigenvalues at low and high temperature can be identified with each other.

The phenomenon is not yet fully understood. An analogue appears in the theory of symmetric polynomials, where there is an automorphism sending a Jack polynomial to its dual by simultaneously transposing its labeling Young diagram and inverting the parameter $\theta$; see \cite[Section~3]{S} or \cite[equation~(10.17)]{M} for a precise statement. Since in this paper we consider both low and high temperature regimes, this duality suggests a connection between them. When~$N$,~$M$ are fixed and $\theta\to\infty$, the vector \smash{$\vec{c}=\vec{a}\boxplus_{N,M}^{\theta}\vec{b}$} concentrates at the roots of $P_{N,M}(z)$, which are identified with the rectangular finite free convolution of $\vec{a}$ and \smash{$\vec{b}$} as defined in~\cite{GrM,MSS}. When $N,M\to\infty$, $\theta\to 0$, and $N\theta\to\gamma$, $M\theta\to q\gamma$, the vector $\vec{c}$ converges in moments to the $q$-$\gamma$ convolution of~$\vec{a}$ and~\smash{$\vec{b}$}. We find that the $(N,M)$-rectangular finite free convolution and the $q$-$\gamma$ convolution coincide under a suitable identification of parameters. More precisely, \cite{Gri} introduces a degree~$N$ polynomial, the so-called rectangular $R$-transform, which linearizes the rectangular finite free convolution. We interpret the coefficients of this rectangular $R$-transform as rectangular finite cumulants, and show that, upon identifying~$N$ in the rectangular finite convolution with $-\gamma$ in the $q$-$\gamma$ convolution, the corresponding moment–cumulant relations match exactly. Moreover, since both $N$ and $\gamma$ are positive, these two operations are analytic continuations of each other, jointly extending the moment–cumulant relation to $\gamma\in\R_{\ge 0}\cup\Z_{\le -1}$. See Section~\ref{sec:duality} for more details.

We also note that a similar identification between low and high temperature regimes holds for self-adjoint matrix additions. In~\cite{BCG}, the authors study the addition of two $N\times N$ self-adjoint matrices in the high temperature regime $N\to\infty$, $\theta\to 0$, $N\theta\to\gamma>0$, and introduce the so-called $\gamma$-convolution and $\gamma$-cumulants. On the other hand, \cite{AP} introduces a family of $d\times d$ free cumulants arising from finite self-adjoint matrix additions in the low temperature setting, and the authors of these two papers observed that their moment-cumulant relations also match upon identifying $d$ with $-\gamma$. We believe that this matching, appearing in both self-adjoint and rectangular matrix additions, should not be a mere coincidence.

\subsection{Techniques and difficulties}
Unlike many other classes of $\beta$-random matrices, we do not have a density function for our object $\Vec{c}=\Vec{a}\boxplus_{N,M}^{\theta}\Vec{b}$, and, due to the openness of the positivity conjecture, we cannot even guarantee that such a density exists. Consequently, the proofs of the main results in the low and high temperature regimes, Theorems~\ref{thm:lln} and~\ref{thm:hightemplln}, both rely heavily on moment calculations. We characterize the distribution of $\Vec{c}$ using the type BC Bessel generating function $G_{N;\theta}(x_{1},\dots,x_{N};\Vec{c})$, which is a new object in the random matrix literature, and we apply two different approaches in the low and high temperature regimes, respectively, to extract the moment information of $\Vec{c}$.

In order to apply such approaches, it is necessary to figure out the correct notion of Bessel function $B(\vec{c};x_{1},\dots,x_{N};\theta,N,M)$ for rectangular matrices. On one hand, we start from the case $\theta=\frac{1}{2},1,2$ and define $B(\Vec{c};x_{1},\dots,x_{N};\theta,N,M)$ as the matrix integral in~\eqref{eq_matrixint}, based on the probabilistic intuition of rectangular random matrices. On the other hand, for arbitrary $\theta>0$, we define our type BC Bessel function to be a symmetric Dunkl kernel, that is known as the joint eigenfunction of the corresponding type BC Dunkl operators, with eigenvalues given by the symmetric moments of $\Vec{c}$. While there are infinite versions of Dunkl kernels, we choose the root multiplicities $m_{\pm e_{i}}$, $m_{\pm e_{i}\pm e_{j}}$ in a unique way that
\begin{itemize}\itemsep=0pt
\item[(1)] For $\theta=\frac{1}{2},1,2$, it coincides with~\eqref{eq_matrixint}.

\item[(2)] For general $\theta>0$, it has nice explicit power series expansion that naturally extrapolates from $\theta=\frac{1}{2},1,2$.
\end{itemize}

We find such root multiplicities and verify the analytic and combinatorial properties of $B(\cdot;x_{1},\dots,x_{N};\theta,N,M)$ in Section~\ref{sec:pre}, by applying the general theory of special functions and symmetric spaces under random matrix motivations.

In the low temperature regime, we use the explicit expansion of the Bessel generating function to compute the limiting distribution of $\Vec{c}$. In the high temperature regime, we study the asymptotic behavior of the action of Dunkl operators on $G_{N;\theta}(x_{1},\dots,x_{N};\Vec{c})$, which extracts moment information. More precisely, in Theorem~\ref{thm:hightemperaturemainthm} we establish an equivalence between the following two conditions for a sequence of random $N$-tuple $\Vec{c}_{N}=(c_{N,1},\dots,c_{N,N})\in\R_{\ge 0}^{N}$, $N=1,2,\dots$, in the regime $N,M\to\infty$, $\theta\to 0$, $N\theta\to\gamma$, $M\theta\to q\gamma$:
\begin{itemize}\itemsep=0pt
\item[(1)] $\{\Vec{c}_{N}\}_{N=1}^{\infty}$ converges in moments as in Definition~\ref{def:llnsatisfaction0}.

\item[(2)] The $l^{\rm th}$ order partial derivative in $x_{1}$ of $\ln(G_{N;\theta}(x_{1},\dots,x_{N};\Vec{c}))$ at 0 converges to some real number for all $l=1,2,\dots $, and the partial derivatives in more than one variables among~${x_{1},\dots,x_{N}}$ at 0 all converge to 0.
\end{itemize}

The nontrivial limit of the $l^{\text{th}}$ order derivative in condition 2 yields, up to a constant factor, the $q$-$\gamma$ cumulant $\kappa_{l}$ of the sequence $\{\Vec{c}_{N}\}_{N=1}^{\infty}$. Note that this equivalence is itself independent of the addition operation and can be applied to a single sequence of (virtual) rectangular matrices; see Section~\ref{sec:laguerre} for an example.

\begin{Remark}
There exists an analog of Theorem~\ref{thm:hightemperaturemainthm} in the fixed temperature regime (for $\beta=1,2$), given in~\cite{B2} using a different approach that relies on the concrete matrix structure.
\end{Remark}

Compared to previous studies of rectangular additions, which mostly treat real or complex matrices, this paper defines and analyzes general $\beta$-additions that do not rely on a concrete matrix realization. Relative to the self-adjoint case, several additional technical difficulties arise. First, there are two size parameters, $N$ and $M$, and we allow them to grow at different rates. More importantly, due to the more involved root multiplicities, the type BC Bessel generating functions and type BC Dunkl operators have more complicated expressions, which in turn makes the combinatorics in the asymptotic analysis considerably more intricate. Because of these two issues, and because rectangular matrices are relatively less studied in the literature, it requires more effort to properly define the rectangular versions of empirical measures, moments, cumulants, and so on, and to identify the limiting regimes in which nontrivial behavior and connections to known objects arise. The reader will also see that the moment-cumulant relation for our $q$-$\gamma$ convolution is substantially more complicated, yet it degenerates to the usual, free, rectangular free, and $\gamma$-convolutions, which characterize several other random matrix addition models.

\subsection{Summary of the paper}

The paper is organized as follows. In Section~\ref{sec:pre}, we introduce the type BC Bessel function and the Bessel generating function, which play the role of characteristic functions for rectangular matrices. In Section~\ref{sec:lowtemp}, we study the low temperature behavior. In Section~\ref{sec:lln}, we prove the main theorem in the high-temperature regime and introduce the $q$-$\gamma$ cumulants in an analytic way. We then study the moment–cumulant relations for $q$-$\gamma$ convolution in more detail, provide an explicit combinatorial description, and point out its connection with classical free probability theory in Section~\ref{sec:momentcumulant}. In Section~\ref{sec:duality}, we investigate the quantitative connection between the low and high temperature regimes. Finally, the appendix gives a brief calculation about the high temperature behavior of the type BC Bessel functions.

\section{Bessel functions and Dunkl operators} \label{sec:pre}
\subsection{Symmetric polynomials}\label{sec:sympoly}
Symmetric polynomials are polynomials in $N$ variables for some $N\ge 1$, and are invariant under any permutation of $\vec{x}:=(x_1,\dots,x_{N})$. They play a crucial role in combinatorics, representation theory, and random matrices. This section fixes some related notation that we will use in this paper. For a detailed introduction of classical results of symmetric polynomials, see, e.g., \cite{M}.

\begin{Definition}\label{def:YD}
 A \emph{partition} $\lambda$ is an $N$-tuple of nonnegative integers $(\lambda_{1}\ge \lambda_{2}\ge \dots \ge \lambda_{N}\ge 0)$. We identify $(\lambda_{1},\dots,\lambda_{N})$ with $(\lambda_{1},\dots,\lambda_{N},0,\dots,0)$, and denote the \emph{length} of $\lambda$ by $l(\lambda)\in \Z_{\ge 1}$, which is the number of strictly positive $\lambda_{i}$. We say a partition is even, if $\lambda_{1},\dots,\lambda_{l(\lambda)}$ are all even.

 Let
 $|\lambda|=\sum_{i=1}^{l(\lambda)}\lambda_{i}$. A \emph{Young diagram} is a graphical representation of a partition. Given a~partition $\lambda$, view it as a collection of $|\lambda|$ boxes, such that there are $\lambda_{i}$ boxes in the $i^{\rm th}$ row. In this paper, we do not distinguish a partition and its corresponding Young diagram. Let $\lambda_{j}'$ be the number of boxes on the \smash{$j^{\rm th}$} column of $\lambda$, and \smash{$\lambda'=(\lambda_{1}',\dots,\lambda_{\lambda_{1}}')$} be the transpose of $\lambda$.

 For two partitions $\lambda$, $\mu$ such that $|\lambda|=|\mu|$, there is a lexicographical order between them, that is, $\lambda>\mu$ if and only if for some $j\in \Z_{\ge 1}$, $\lambda_{1}=\mu_{1},\dots,\lambda_{j-1}=\mu_{j-1}$ and $\lambda_{j}>\mu_{j}$.
\end{Definition}

A central object in symmetric polynomials that we use in this paper is the Jack polynomials, see \cite[Chapter VI.10]{M} for a detailed exposition. For completeness, we give a brief definition.\footnote{We will not directly apply this definition later though.} Let~$X$ be a formal auxiliary variable. Let $\partial_{i}$, $i=1,2,\dots,N$, be the partial derivative operator in $x_{i}$, and \smash{$V(\vec{x})=\prod_{1\le i<j\le N}(x_{i}-x_{j})$} be the Vandermonde determinant.

\begin{Definition}[{\cite[Chapter VI]{M}}]\label{def:jack}
 Let $D_{N}(X;\theta)$ be a differential operator of the form
 \[
 D_{N}(X;\theta)=V(\vec{x})^{-1}\det\Bigg[x_{i}^{N-j}\left(x_{i}\frac{\partial}{\partial x_{i}}+(N-j)\theta+X\right)\Bigg]_{1\le i,j\le N}.
 \]
 $D_{N}(X;\theta)$ is a generating function (with variable $X$) of linear differential operators $D_{N}^{1},\dots, D_{N}^{N}$ acting on the symmetric polynomials in $N$ variables, such that
 \[D_{N}(X;\theta)=\sum_{r=0}^{N}D^{r}_{N}X^{N-r}.\]

 The Jack polynomials in $N$ variables are a collection of elements $P_{\lambda}(\vec{x};\theta)$, indexed by the partitions $\lambda$ such that $l(\lambda)\le N$. Each $P_{\lambda}(\vec{x};\theta)$ is uniquely determined by the following two properties:
 \begin{equation}\label{eq_jack1}
 P_{\lambda}(\vec{x};\theta)=m_{\lambda}(\vec{x})+\sum_{\mu<\lambda}u_{\mu}^{\lambda}(\theta)m_{\mu}(\vec{x}),
 \end{equation}
 where $m_{\lambda}$ is the monomial symmetric polynomial indexed by the partition $\lambda$, $u_{\mu}^{\lambda}(\theta)\in \R$ are parameterized by $\theta$; and
\[
 D_{N}(X;\theta)P_{\lambda}(\vec{x};\theta)=c_{\lambda}^{\lambda}(\theta)P_{\lambda}(\vec{x};\theta),
 \]
 where
 \[
 c_{\lambda}^{\lambda}(\theta)=\prod_{i=1}^{N}\big(X+\theta^{-1}\lambda_{i}+N-i\big).\]
\end{Definition}

Furthermore, let \begin{equation}\label{eq_jack2}
Q_{\lambda}(\vec{x};\theta)=b_{\lambda}(\theta)\cdot P_{\lambda}(\vec{x};\theta)\end{equation}
be the dual of the Jack polynomial,
 where \[b_{\lambda}(\theta)=\prod_{(i,j)\in \lambda}\frac{\lambda_{i}-j+\theta (\lambda_{j}'-i)+\theta}{\lambda_{i}-j+\theta (\lambda_{j}'-i)+1}.\]

Given two Jack polynomials $P_{v}(\vec{x};\theta)$ and $P_{\mu}(\vec{x};\theta)$, their product $P_{v}(\vec{x};\theta)\cdot P_{\mu}(\vec{x};\theta)$ is again a symmetric polynomial, and hence can be written as a unique linear combination of Jack polynomials. Namely, we have the following equality, where $C^{v,\mu}_{\lambda}(\theta)$ is the coefficient of $P_{\lambda}(\vec{x};\theta)$ in the expansion:
\begin{equation}\label{eq_coefficient1}
P_{v}(\vec{x};\theta)P_{\mu}(\vec{x};\theta)=\sum_{\lambda}C^{v,\mu}_{\lambda}(\theta)P_{\lambda}(\vec{x};\theta)
.\end{equation}

It turns out that the coefficients $u_{\mu}^{\lambda}(\theta)$ are positive rational functions of $\theta$ independent of the number of variables, see \cite[Chapter VI.10]{M} for the explicit expressions. As a consequence, $C^{v,\mu}_{\lambda}(\theta)$ are also independent of $N$.

\subsection{Type BC Bessel functions}\label{sec:typebc}
For positive integers $M\ge N$, $\theta\in \R_{>0}$, take an $N$-tuple of nonnegative real numbers $\Vec{a}=(a_{1}\ge a_{2}\ge\dots \ge a_{N})$ as the given data. The type BC Bessel function $B(\Vec{a},x_{1},\dots,x_{N};\theta,N,M)$ is a version of a multivariate symmetric Fourier kernel, with certain nontrivial root multiplicities given by parameter $\theta>0$. In the special functions literature, this is a special case of the so-called symmetric Dunkl kernel, see Section~\ref{sec:dunkl}.

The following definition gives a matrix integral expression for $B(\Vec{a},x_{1},\dots,x_{N};\theta,N,M)$, when~${\theta=\frac{1}{2},1,2}$.

\begin{Definition}\label{def:matrixintegral}
When $\theta=1$,
\[B(\vec{a},x_1,x_2,\dots,x_N;\theta, N,M)=\int{\rm d}U\int {\rm d}V \exp \left(\frac{1}{2}\operatorname{Tr}(U\Lambda VX+X^{*}V^{*}\Lambda^{*}U^{*})\right),\]
where \begin{gather*}
 \Lambda=\begin{bmatrix}
a_{1} & & & &&0&\dots &0\\
 & a_{2} & & &&0&\dots & 0\\
 & &\dots & &&\\
 &&&\dots &\\
 &&&&a_{N} &0&\dots &0
\end{bmatrix}_{N\times M},\\
 X=\begin{bmatrix}
x_{1} & & & &\\
 & x_{2} & & &\\
 & &\dots &\\
 &&&\dots &\\
 &&&&x_{N} &\\
 0&&\dots &&0\\
 && \dots &&\\
 0&&\dots &&0
\end{bmatrix}_{M\times N},\end{gather*}
$U\in {\rm U}(N)$, $V\in {\rm U}(M)$ are integrated under Haar measures.

When $\theta=\frac{1}{2},2$, $B(\Vec{a},x_{1},\dots,x_{N};\theta,N,M)$ has a similar integral expression, with $({\rm U}(N),\allowbreak {\rm U}(M))$ replaced by $({\rm O}(N),{\rm O}(M))$ and $({\rm Sp}(N), {\rm Sp}(M))$, respectively.
\end{Definition}

Definition~\ref{def:matrixintegral} provides an explicit connection with rectangular matrices, where the integral is of the form as a ``Fourier transform''/characteristic function of $A=U\Lambda V$. However, since there is no (skew) field with real dimension $\beta$ for general $\beta>0$, one needs to define the Bessel functions in an alternate way that does not rely on explicit matrix structure. For this purpose, we first introduce some notations.

\begin{Definition}
For a partition $\mu$, let $s=(i,j)\in \mu$ be the coordinate of the box on the $j^{\rm th}$ column and the $i^{\rm th}$ row in $\mu$. Fix $t\in \R$, $\theta>0$. We denote
\begin{gather*}
H(\mu)=\prod_{s\in \mu}[\mu_{i}-j+1+\theta (\mu_{j}'-i)],
\qquad
H'(\mu)=\prod_{s\in \mu}[\mu_{i}-j+\theta+\theta (\mu_{j}'-i)],
\end{gather*}
and
\begin{gather*}
(t)_{\mu}=\prod_{s\in \mu}[t+j-1-\theta (i-1)].
\end{gather*}
\end{Definition}

\begin{Definition}\label{def:bessel}
Take $\theta>0$, $N\le M$, the type BC multivariate Bessel function labeled by~${\Vec{a}=(a_{1}\ge a_{2}\ge\dots \ge a_{N})}$ is a multivariate analytic function in both $\Vec{a}$ and $(x_{1},\dots,x_{N})$, defined by
\begin{gather}
B(\Vec{a},x_{1},\dots,x_{N};\theta,N,M)\nonumber\\
\qquad=\sum_{\mu}\prod_{i=1}^{N}\frac{\Gamma(\theta M -\theta (i-1))}{\Gamma(\theta M-\theta(i-1)+\mu_{i})}\frac{1}{H(\mu)}2^{-2|\mu|}\frac{P_{\mu}\bigl(a_{1}^{2},\dots,a_{N}^{2};\theta\bigr)P_{\mu}\bigl(x_{1}^{2},\dots,x_{N}^{2};\theta\bigr)}{P_{\mu}\bigl(1^{N};\theta\bigr)}\nonumber\\
\qquad=\sum_{\mu}\prod_{i=1}^{N}\frac{\Gamma(\theta N -\theta (i-1))}{\Gamma(\theta N-\theta(i-1)+\mu_{i})}\frac{\Gamma(\theta M -\theta (i-1))}{\Gamma(\theta M-\theta(i-1)+\mu_{i})}\frac{H'(\mu)}{H(\mu)}\nonumber\\
\phantom{\qquad=}{}\times 2^{-2|\mu|}P_{\mu}\bigl(a_{1}^{2},\dots,a_{N}^{2};\theta\bigr)P_{\mu}\bigl(x_{1}^{2},\dots,x_{N}^{2};\theta\bigr),\label{eq_expansion}
\end{gather}
where $\mu$ is summed over all partitions of length at most $N$.
\end{Definition}

\begin{Remark}
The last equality above holds since by \cite[Chapter~VI.10, equation~(10.20)]{M},
\[P_{\mu}\bigl(1^{N};\theta\bigr)=\frac{(N\theta)_{\mu}}{H'(\mu)}.\]
\end{Remark}

The following example gives a connection of $B(\cdot,x_{1},\dots,x_{N};\theta, N,M)$ with the usual single variable Bessel function.
\begin{ex}\label{ex:usualbessel}
 When $N=1$,
 \[B(a,{\rm i}x;\theta, 1,M)=\Gamma(M\theta)\cdot\left(\frac{ax}{2}\right)^{-(M\theta-1)} B_{M\theta-1}(ax),\]
 where $B_{\alpha}$ on the right hand side is the Bessel function of the first kind.
\end{ex}

Definition~\ref{def:bessel} generalizes the notion of the type BC Bessel function to any $\theta>0$. In particular, when $\theta=\frac{1}{2}, 1, 2$, Definition~\ref{def:bessel} provides an explicit power series expansion of the matrix integral in Definition~\ref{def:matrixintegral}.

\begin{thm}[{\cite[Section 13.4.3]{Forrester}}]\label{thm:fouriertransform}
For $\theta=\frac{1}{2}, 1, 2$,
\begin{align*}
 \int {\rm d}U\int {\rm d}V \,\exp\left( \frac{1}{2}\operatorname{Tr}(U\Lambda VX+X^{*}V^{*}\Lambda^{*}U^{*})\right)
=B(\Vec{a},x_{1},\dots,x_{N};\theta,N,M),
\end{align*}
where the matrix integral on the left is defined in the same way as in Definition~{\rm~\ref{def:matrixintegral}}.
\end{thm}

\begin{Remark}
 There are more than one way to show the equivalence of these two expressions, see, e.g., \cite{Xu} for more details.

 Let us mention that the integral expression on the left can be identified (up to a factor $2{\rm i}$ in the exponent) with the spherical function of the Euclidean symmetric spaces corresponding to~${{\rm O}(N+M)/{\rm O}(N)\times {\rm O}(M)}$, ${\rm U}(N+M)/{\rm U}(N)\times {\rm U}(M)$, ${\rm Sp}(N+M)/{\rm Sp}(N)\times {\rm Sp}(M)$, respectively, for $\theta=\frac{1}{2},1,2$. This is related to the fact that Bessel functions are eigenfunctions of Dunkl operators (presented in next section), since spherical function can also be defined as the (unique up to constant) eigenfunction of certain differential operators acting on its corresponding symmetric space. See, e.g., \cite{Hel2} for more details.
\end{Remark}

\subsection{Type BC Dunkl operators}\label{sec:dunkl}
As a special class of differential operators, Dunkl operators were introduced in~\cite{D}, and can be thought of as a generalization of the usual partial derivatives on multivariate analytic functions, which take Fourier kernels as eigenfunctions. For the purpose of this paper, we specify to a~special class of rational Dunkl operators under root system of type BC, which is parametrized by a single variable $\theta>0$ and plays a central role in Section~\ref{sec:lowtemp}. For more detailed exposition of Dunkl theory, see~\cite{A,Ro}.

\begin{Definition}\label{def:dunkl}
 For $M\ge N\ge 2$, $\theta>0$,
 let $D_{i}$ be a differential operator acting on analytic functions on $\C^{N}$ with variables $x_{1},\dots,x_{N}$, that
 \[
 D_{i}=\partial_{i}+\left[\theta(M-N+1)-\frac{1}{2}\right]\frac{1-\sigma_{i}}{x_{i}}+\theta\sum_{j\ne i}\left[\frac{1-\sigma_{ij}}{x_{i}-x_{j}}+\frac{1-\tau_{ij}}{x_{i}+x_{j}}\right],
 \]
 where $\sigma_{i}$ interchanges $x_{i}$ and $-x_{i}$, $\sigma_{ij}$ interchanges $x_{i}$ and $x_{j}$, and $\tau_{ij}$ interchanges $x_{i}$ and $-x_{j}$.
\end{Definition}

\begin{Remark}
There are Dunkl operators of general type, with a corresponding root system of type A-D and a complex-valued root multiplicity function. The
 $D_{i}$ in this paper specify \cite[Definition 2.7 and Example 2.15\,(3)]{Ro} by setting the root multiplicity function to be
 \[
 k_{e_{i}}=\theta(M-N+1)-\frac{1}{2},\qquad k_{e_{i}+e_{j}}=\theta.
 \]
\end{Remark}

\begin{prop}[\cite{D}]\label{prop:commutativity}
 The Dunkl operators of the same root multiplicities commute, i.e.,
$D_{i}D_{j}=D_{j}D_{i}$
 for any $1\le i$, $j\le N$.
\end{prop}
The following result provides connection of type BC multivariate Bessel functions and Dunkl operators, namely, the former are eigenfunctions of the latter.

\begin{Definition}
 Fix $N\ge 1$. For $k=1,2,\dots $, denote $\mathrm{P}_{k}=D_{1}^{k}+\dots+D_{N}^{k}$.
\end{Definition}
\begin{thm}[{\cite[Proposition 4.5]{Ro3}}]\label{thm:dunklonbessel}
Given $\Vec{a}=(a_{1}\ge \dots \ge a_{N})$ for each $k=1,2,\dots $,
 \[
 \mathrm{P}_{2k}B(\Vec{a}, x_{1},\dots,x_{N};\theta,N,M)=\left(\sum_{i=1}^{N}(a_{i})^{2k}\right)\cdot B(\Vec{a},x_{1},\dots,x_{N};\theta,N,M).
 \]
\end{thm}

\begin{Remark}
 From Definition~\ref{def:bessel}, one can see that $B(\Vec{a},x_{1},\dots,x_{N};\theta, N,M)$ is symmetric under the actions of the Weyl group of the root system $BC_{N}$, namely, invariant by interchanging~$x_{i}$ with $x_{j}$ and replacing $x_{i}$ by $-x_{i}$. Similarly, it is necessary to take symmetric power sum of the $D_{i}$ with even power, which satisfies the same symmetry.
\end{Remark}

\subsection{Matrix addition and moments}\label{sec:addition}

For $\Vec{c}=\Vec{a}\boxplus_{N,M}^{\theta}\Vec{b}$, we assume in this section that $\Vec{a}$, $\Vec{b}$ are deterministic, and recall from Definition~\ref{def:deterministicaddition} that the distribution $\mathfrak{m}$ of $\Vec{c}$ is given by acting on type BC Bessel function.
Note that polynomials are bounded and smooth on compact sets, and therefore are legitimate test functions of $\mathfrak{m}$. Moreover, using the expansion in Definition~\ref{def:bessel}, we can view the type BC Bessel function as a generating function of symmetric polynomials of $N$ variables $c_{1},\dots,c_{N}$. More precisely, by expanding Bessel functions on both sides of~\eqref{eq_additioindef} using~\eqref{eq_expansion}, we have the following.

\begin{prop}\label{prop:polyexpectation}
For each partition $\lambda$ with $l(\lambda)\le N$, let \smash{$\Vec{c}=\Vec{a}\boxplus_{N,M}^{\theta}\Vec{b}$}, then
\begin{gather}
 \E{P_{\lambda}\bigl(c_{1}^{2},\dots,c_{N}^{2};\theta\bigr)}\nonumber\\
\qquad
=\sum_{|v|+|\mu|=|\lambda|}\frac{H(\lambda)}{H(v)H(\mu)}\frac{\prod_{i=1}^{N}\Gamma(\theta (M-i+1))\Gamma(\theta (M-i+1)+\lambda_{i})}{\prod_{i=1}^{N}\Gamma(\theta (M-i+1)+v_{i})\Gamma(\theta (M-i+1)+\mu_{i})}\nonumber\\
\phantom{\qquad
=}{} \times\frac{P_{\lambda}\bigl(1^{N};\theta\bigr)}{P_{v}\bigl(1^{N};\theta\bigr)P_{\mu}\bigl(1^{N};\theta\bigr)}C^{v,\mu}_{\lambda} (\theta)P_{v}\bigl(a_{1}^{2},\dots,a_{N}^{2};\theta\bigr)P_{\mu}\bigl(b_{1}^{2},\dots,b_{N}^{2};\theta\bigr),
\label{eq_polyexpectation}
\end{gather}
where $v$, $\mu$ are two partitions of length at most $N$.
\end{prop}

Proposition~\ref{prop:polyexpectation} provides explicit data of the distribution of random singular values $\Vec{a}\boxplus_{N,M}^{\theta}\Vec{b}$ in terms of moments. It is believed, but not yet proved that $\mathfrak m$ (with such moments) is indeed a~(symmetric) positive probability measure on $\R^{N}$. For $\beta=1,2,4$, this holds automatically because the probability measure is constructed explicitly by the matrix structure, (and \cite[Corollary~4.8]{Ro3} provides an explicit expression of this measure for $M\ge 2N$), while for general~${\beta>0}$, the randomness of $\Vec{a}\boxplus_{N,M}^{\theta}\Vec{b}$ holds and is studied in this paper in the weaker sense given by~\eqref{eq_polyexpectation}.

\subsection{Type BC Bessel generating functions}\label{sec:bgf}
For $\Vec{a}=(a_{1}\ge \dots \ge a_{N}\ge 0)$, we assume that $\Vec{a}$ is random, and its distribution is given by a~symmetric generalized function $\mathfrak{m}$ acting on smooth functions, and in particular polynomials, on $\R_{\ge 0}^{N}$.

\begin{Definition}\label{def:bgf}
 Fix $M\ge N$, $\theta>0$. Given a compactly supported symmetric generalized function $\mathfrak{m}$ on $\R_{\ge 0}^{N}$ defined as above, let the Bessel generating function of $\mathfrak{m}$ be a function of~${x_{1},\dots,x_{N}}$ given by
 \[
 G_{N,\theta}(x_{1},\dots,x_{N};\mathfrak{m}):=\langle \mathfrak{m}, B(\Vec{a},x_{1},\dots,x_{N};\theta,N,M)\rangle,
 \]
 where the bracket denotes testing $\mathfrak{m}$ by $B(\Vec{a},x_{1},\dots,x_{N};\theta, N,M)$, in which $\Vec{a}$ are the variables and $x_{1},\dots,x_{N}$ are parameters.
\end{Definition}

We also define the Bessel generating function for a class of fast decaying probability measures, for potential applications of our theory (see, e.g., Section~\ref{sec:laguerre}). As preparation, we state a uniform upper bound of multivariate Bessel functions.
\begin{prop}\label{prop:besselbound}
 For any $\theta>0$, $M\ge N$, $\Vec{a}=(a_{1}\ge\dots \ge a_{N})\in \R_{\ge 0}^{N}$, $x=(x_{1},\dots,x_{N})\in \R^{N}$, we have
 \begin{equation}\label{eq_besselbound1}
 0\le B(\Vec{a},x_{1},\dots,x_{N};\theta,N,M)\le \left[\frac{1}{\theta}+\frac{1}{\theta}\left(\frac{a_{1}|x|}{2}\right)^{2}{\rm e}^{\frac{a_{1}|x|}{2}}\right]^{N},
 \end{equation}
 and for any $k_{1},\dots,k_{s}\in \Z_{\ge 1}$,
 \begin{gather}
 \left|\left(\prod_{i=1}^{s}P_{2k_{i}}\right)B(\Vec{a},x_{1},\dots,x_{N};\theta,N,M)\right|\nonumber\\
 \qquad\le \prod_{i=1}^{s}\left(\sum_{j=1}^{N}a_{i}^{2k_{i}}\right)\left[\frac{1}{\theta}+\frac{1}{\theta}\left(\frac{a_{1}|x|}{2}\right)^{2}{\rm e}^{\frac{a_{1}|x|}{2}}\right]^{N}.\label{eq_besselbound2}
 \end{gather}
 \end{prop}
\begin{proof}
 Without loss of generality, assume that $|x_{1}|\ge |x_{2}|\ge \cdots \ge |x_{N}|$. From Definition~\ref{def:bessel} and~\eqref{eq_jack1}, it is clear that $B(\Vec{a},x_{1},\dots,x_{N};\theta,N,M)\ge 0$, and since \smash{$P_{\mu}\bigl(1^{N};\theta\bigr)=\frac{(N\theta)_{\mu}}{H'(\mu)}$},
 \begin{gather*}
 B(\Vec{a},x_{1},\dots,x_{N};\theta,N,M)\\
 \qquad=\sum_{\mu}\prod_{i=1}^{N}\frac{\Gamma(\theta M -\theta (i-1))}{\Gamma(\theta M-\theta(i-1)+\mu_{i})}\frac{(N\theta)_{\mu}}{H(\mu)H'(\mu)}2^{-2|\mu|}\frac{P_{\mu}\bigl(a_{1}^{2},\dots,a_{N}^{2};\theta\bigr)P_{\mu}\bigl(x_{1}^{2},\dots,x_{N}^{2};\theta\bigr)}{P_{\mu}\bigl(1^{N};\theta\bigr)^{2}}\\
 \qquad\le\sum_{\mu}\prod_{i=1}^{N}\frac{\Gamma(\theta M -\theta (i-1))}{\Gamma(\theta M-\theta(i-1)+\mu_{i})}\frac{(N\theta)_{\mu}}{H(\mu)H'(\mu)}2^{-2|\mu|}a_{1}^{2|\mu|}x_{1}^{2|\mu|}\\
 \qquad\le\sum_{\mu}\prod_{i=1}^{N}\left[\frac{\Gamma(\theta M -\theta (i-1))}{\Gamma(\theta M-\theta(i-1)+\mu_{i})}\frac{\Gamma(\theta N-\theta(i-1)+\mu_{i})}{\Gamma(\theta N-\theta(i-1))}\right]\\
\phantom{\qquad\le}{}\times\frac{1}{\prod_{i=1}^{N}\mu_{i}!}\frac{1}{\prod_{i=1}^{N}\prod_{j=0}^{\mu_{i}-1}(\theta+j)}2^{-2|\mu|}a_{1}^{2|\mu|}x_{1}^{2|\mu|},
 \end{gather*}
 where we rewrite $(N\theta)_{\mu}$ using Gamma functions, upper bound $\frac{1}{H(\mu)H'(\mu)}$ by the fractions in the last line above, and simply bound
 \[
 \frac{P_{\mu}\bigl(a_{1}^{2},\dots,a_{N}^{2};\theta\bigr)P_{\mu}\bigl(x_{1}^{2},\dots,x_{N}^{2};\theta\bigr)}
 {P_{\mu}(1^{N};\theta)^{2}}
 \]
 by \smash{$a_{1}^{2|\mu|}x_{1}^{2|\mu|}$}. Since $M\ge N$, the fractions in the bracket above is bounded by 1, and the whole expression above can be further bounded by
 \begin{gather*}
 \sum_{\mu_{1}\ge\dots \ge\mu_{N}\ge 0}\frac{1}{\prod_{i=1}^{N}\mu_{i}!}\frac{1}{\prod_{i=1}^{N}\prod_{j=0}^{\mu_{i}-1}(\theta+j)}\left(\frac{a_{1}x_{1}}{2}\right)^{2|\mu|}\\
\qquad\le\prod_{i=1}^{N}\left(\sum_{\mu_{i}=0}^{\infty}\frac{1}{\mu_{i}!\prod_{j=0}^{\mu_{i}-1}(\theta+j)}\left(\frac{a_{1}x_{1}}{2}\right)^{2\mu_{i}}\right)\\
 \qquad
 \le\prod_{i=1}^{N}\left(\frac{1}{\theta}+\sum_{\mu_{i}=1}^{\infty}\frac{1}{\theta}\frac{1}{[(\mu_{i}-1)!]^{2}}\left(\frac{a_{1}|x|}{2}\right)^{2\mu_{i}}\right)\\
 \qquad\le \prod_{i=1}^{N}\left(\frac{1}{\theta}+\frac{1}{\theta}\left(\frac{a_{1}|x|}{2}\right)^{2}{\rm e}^{\frac{a_{1}|x|}{2}}\right)
 =\left[\frac{1}{\theta}+\frac{1}{\theta}\left(\frac{a_{1}|x|}{2}\right)^{2}{\rm e}^{\frac{a_{1}|x|}{2}}\right]^{N},
 \end{gather*}
 where we bound
 \[\frac{1}{\prod_{i=1}^{N}\prod_{j=0}^{\mu_{i}-1}(\theta+j)}
 \]
 by $\prod_{i=1}^{N}\bigl(\frac{1}{\theta}\frac{1}{(\mu_{i}-1)!}\bigr)$. This verifies~\eqref{eq_besselbound1}. \eqref{eq_besselbound2} follows from~\eqref{eq_besselbound1} and Theorem~\ref{thm:dunklonbessel}.
\end{proof}

\begin{Definition}\label{def:expdecaying}
 We say a measure $\mathfrak{m}$ on the $N$-tuple $a_{1}\ge\dots \ge a_{N}\ge 0$ is \emph{exponentially decaying} with exponent $R>0$, if
$\int {\rm e}^{NRa_{1}} \mu({\rm d}a_{1},\dots,{\rm d}a_{N})<\infty$.
\end{Definition}

 By Proposition~\ref{prop:besselbound} and Definition~\ref{def:expdecaying}, the Bessel generating function of $\mathfrak{m}$, where $\mathfrak{m}$ is a~compactly supported generalized function or an exponentially decaying measure, is well-defined on a domain near 0. Moreover,
 we will take $\mathfrak{m}$ to be of total mass $1$, which means $\langle \mathfrak{m}, 1\rangle=1$, where $1$ is the constant function~1. So we have
 $G_{N,\theta}(0,\dots,0;\mathfrak{m})=1$.

Now we generalize the addition to random vectors $\Vec{a}$ and \smash{$\Vec{b}$} following Definition~\ref{def:deterministicaddition}.

\begin{Definition}\label{def:additioningeneral}
Given $\theta>0$, $M\ge N$, let $\Vec{a}=(a_{1}\ge\dots \ge a_{N}\ge 0)$, $\Vec{b}=(b_{1}\ge\dots \ge b_{N}\ge 0)$ be two random $N$-tuples whose distributions are given by generalized functions $\mathfrak{m}_{a}$ and $\mathfrak{m}_{b}$ on~$\R_{\ge 0}^{N}$. Let $\Vec{c}$ be a symmetric random vector in $\R_{\ge 0}^{N}$ whose distribution is given by generalized function~$\mathfrak{m}_{c}$, such that
\[
G_{N,\theta}(x_{1},\dots,x_{N};\mathfrak{m}_{c})=
G_{N,\theta}(x_{1},\dots,x_{N};\mathfrak{m}_{a})\cdot G_{N,\theta}(x_{1},\dots,x_{N};\mathfrak{m}_{b}).
\]
We write
\smash{$\Vec{c}=\Vec{a}\boxplus_{N,M}^{\theta}\Vec{b}$}.
\end{Definition}

Since $B(x_{1},\dots,x_{N};\theta,N,M)$ behaves nicely enough in the analytic sense, one can interchange the differentiation over $x_{1},\dots,x_{N}$ and the pairing with $\mathfrak{m}$, and therefore Theorem
\ref{thm:dunklonbessel} generalizes to the following.
\begin{thm}\label{thm:dunklonbgf}
 Let $\mathfrak{m}$ be a symmetric compactly supported generalized function on $\R^{N}$, or an exponentially decaying measure as in Definition~{\rm\ref{def:expdecaying}} with exponent $R$. Let $k_{1},\dots,k_{s}\in \Z_{\ge 1}$. Then $G_{N,\theta}( x_{1},\dots,x_{N};\mathfrak{m})$ is analytic as a function of $(x_{1},\dots,x_{N})$ {\rm(}in the domain $\{x\in \R^{N}\mid |x|<R\}$ in the second case$)$. Moreover,
 \[
 \left(\prod_{i=1}^{s}\mathrm{P}_{2k_{i}}\right)G_{N,\theta}( x_{1},\dots,x_{N};\mathfrak{m})\bigr|_{x_{1}=\dots =x_{N}=0}=\left\langle\mathfrak{m}, \prod_{i=1}^{s}\left(\sum_{j=1}^{N}(a_{j})^{2k_{i}}\right)\right\rangle.
 \]
 The above properties also hold for \[G_{N,\theta}(x_{1},\dots,x_{N};\mathfrak{m}_{c})=
G_{N,\theta}(x_{1},\dots,x_{N};\mathfrak{m}_{a})\cdot G_{N,\theta}(x_{1},\dots,x_{N};\mathfrak{m}_{b}),
\]
where $\mathfrak{m}_{a}$, $\mathfrak{m}_{b}$ are of the above two types.
\end{thm}
\begin{proof}
 This follows from dominated convergence theorem, where the uniform upper bounds of~${B(\cdot,x_{1},\dots,x_{N};\theta, N,M)}$ and its derivatives are given by Proposition~\ref{prop:besselbound}.
\end{proof}

\section{Concentration in low temperature}\label{sec:lowtemp}
In this section, we fix the sizes of the matrices $N$, $M$ and take the inputs $\Vec{a}$, $\Vec{b}$ to be deterministic, and study the behavior of $\Vec{c}=\Vec{a}\boxplus_{N,M}^{\theta}\Vec{b}$ as $\theta\rightarrow \infty$. According to the statistical physics interpretation, when $\theta\rightarrow \infty$ the temperature is going down to 0, and hence the random vector~$\Vec{c}$ will freeze at some deterministic $N$-tuple.
\subsection{Finite law of large numbers}
Before taking the limit, we consider the expected characteristic polynomial of $CC^{*}$ for each $\theta<\infty$. It turns out that the expression does not really depend on $\theta$. The following lemma will be used later in the proof. Let $C^{v,\mu}_{\lambda}(\theta)$ be the coefficient defined in~\eqref{eq_coefficient1}.

\begin{Lemma}\label{lem:automorphism}
When $\lambda=1^{l}$, $C^{v,\mu}_{\lambda}(\theta)\ne 0$ only when $v=1^{i}$, $\mu=1^{j}$, and $i+j=l$. Moreover,
\[
C^{1^{i},1^{j}}_{1^{l}}(\theta)=\frac{\prod_{m=1}^{l}\bigl(\frac{l\theta-m\theta+\theta}{l\theta-m\theta+1}\bigr)}{\prod_{m=1}^{i}\bigl(\frac{i\theta-m\theta+\theta}{i\theta-m\theta+1}\bigr)
\prod_{m=1}^{j}\bigl(\frac{j\theta-m\theta+\theta}{j\theta-m\theta+1}\bigr)}.
\]
\end{Lemma}
\begin{proof}
This is studied in~\cite{GM}, and for the convenience of the readers we reproduce the proof. Applying the automorphism $\omega_{\theta}$ of the algebra of symmetric functions (see \cite[Chapter VI.10]{M}), which acts on Jack polynomials in the following way
$
\omega_{\theta}(P_{\lambda}(\vec{x};\theta))=Q_{\lambda'}\bigl(\vec{x};\theta^{-1}\bigr)$,
\eqref{eq_coefficient1} becomes
\[
Q_{(i,0,\dots )}\bigl(\vec{x};\theta^{-1}\bigr)\cdot Q_{(j,0,\dots )}\bigl(\vec{x};\theta^{-1}\bigr)=\sum_{\mu}C^{1^{i},1^{j}}_{\mu}(\theta)\cdot Q_{\mu'}\bigl(\vec{x};\theta^{-1}\bigr).
\]
Recall from~\eqref{eq_jack2} that $Q_{\lambda}(\vec{x};\theta)=b_{\lambda}(\theta)P_{\lambda}(\vec{x};\theta)=\frac{H(\lambda)}{H'(\lambda)}P_{\lambda}(\vec{x};\theta)$. By comparing the coefficient of the leading monomial \smash{$x_{1}^{l}$}, we have
\[
C^{1^{i},1^{j}}_{1^{l}}(\theta)=\frac{b_{1^{i}}\bigl(\theta^{-1}\bigr)b_{1^{j}}\bigl(\theta^{-1}\bigr)}{b_{1^{l}}\bigl(\theta^{-1}\bigr)},
\]
and $C^{v,\mu}_{1^{l}}=0$ if $v$ or $\mu$ has more than one column.
\end{proof}

\begin{thm}\label{thm:characteristicpoly}
Fix $M\ge N$, given $\Vec{a}$ and $\Vec{b}$, let $\Vec{c}=\Vec{a}\boxplus_{N,M}^{\theta}\Vec{b}$. Take $z$ as a formal variable, and let
\begin{equation}\label{eq_characteristicpoly}
P_{N,M}^{\theta}(z)=\E{\prod_{i=1}^{N}\bigl(z-c_{i}^{2}\bigr)}.
\end{equation}
Then the explicit expression of $P_{N,M}^{\theta}(z)$ is $\theta$-independent, and
$
P_{N,M}^{\theta}(z)
=P_{N,M}(z)
$
for all $\theta>0$, where $P_{N,M}(z)$ is defined in {\rm\eqref{eq_charpoly}}.
\end{thm}

\begin{proof}
Rewrite the product on the right side of~\eqref{eq_characteristicpoly} as
\[\prod_{i=1}^{N}\bigl(z-c_{i}^{2}\bigr)=\sum_{l=0}^{N}(-1)^{l}e_{l}\bigl(c_{1}^{2},\dots,c_{N}^{2}\bigr)z^{N-l},\]
it turns out that $P_{N,M}^{\theta}(z)$ is given by the moments of $\bigl\{c_{i}^{2}\bigr\}_{i=1}^{N}$ only in terms of elementary symmetric polynomials.

Taking the partition $\lambda=\bigl(1^{j},0^{M-j}\bigr)$, $P_{\lambda}(\vec{x};\theta)=e_{j}(\vec{x})$ for any $\theta>0$. We use Proposition~\ref{prop:polyexpectation} and it remains to specify the coefficients.
From Lemma~\ref{lem:automorphism}, we get
\[\frac{H\bigl(1^{l}\bigr)}{H\bigl(1^{i}\bigr)H\bigl(1^{j}\bigr)}C^{1^{i},1^{j}}_{1^{l}}(\theta)=\frac{H'\bigl(1^{l}\bigr)}{H'\bigl(1^{i}\bigr)H'\bigl(1^{j}\bigr)}=\frac{l!}{i!j!}.\]
Moreover, direct calculation yields
\[\frac{e_{l}\bigl(1^{N}\bigr)}{e_{i}\bigl(1^{N}\bigr)e_{j}\bigl(1^{N}\bigr)}=\frac{i!(N-i)!j!(N-j)!}{N!l!(N-l)!},\]
and when $\lambda=1^{l}$, $v=1^{i}$, $\mu=1^{j}$,
\[\frac{\prod_{i=1}^{N}\Gamma(\theta (M-i+1))\Gamma(\theta (M-i+1)+\lambda_{i})}{\prod_{i=1}^{N}\Gamma(\theta (M-i+1)+v_{i})\Gamma(\theta (M-i+1)+\mu_{i})}=\frac{(M-i)!}{M!}\frac{(M-j)!}{(M-l)!}.\] Combine all these together finishes the proof.
\end{proof}

We highlight the connection of our result with the so-called finite free probability, which was initiated in recent years by Marcus, Spielman and Srivastava and studies convolution of polynomials. Given two polynomials $p(z)=\sum_{i=0}^{N}z^{N-i}a_{i}, q(z)=\sum_{i=0}^{N}z^{N-i}b_{i}$ with degree at most $N$, \cite{MSS} defines the rectangular additive convolution for two $N\times N$ matrices, and~\cite{GrM} generalizes it to arbitrary rectangular matrices, such that the $(N,M)^{\rm th}$ rectangular additive convolution of $p(z)$ and $q(z)$ is defined as
\[p(z)\boxplus\boxplus_{N}^{M-N}q(z)=\sum_{l=0}^{N}z^{N-l}(-1)^{l}\sum_{i+j=l}\Bigg(\frac{(N-i)!(N-j)!}{N!(N-l)!}\frac{(M-i)!(M-j)!}{M!(M-l)!}a_{i}b_{j}\Bigg).\]

Let $\chi_{z}(\cdot)$ be the characteristic polynomial $\det(zI-\cdot)$. Take $p(z)=\chi_{z}(AA^{*})$, $q(z)=\chi_{z}(BB^{*})$, where $A$ and $B$ are two $N\times M$ real/complex matrices, and let $U_{N\times N}$, $V_{M\times M}$ are independent Haar orthogonal/unitary. In~\cite{GrM}, it is shown that
\[
p(z)\boxplus\boxplus_{N}^{M-N}q(z)=\E{\chi_{z}((A+UBV)(A+UBV)^{*})}.
\]
 Theorem~\ref{thm:characteristicpoly} generalizes this operation from $\theta=\frac{1}{2},1$ to arbitrary $\theta>0$, with a different approach not relying on the concrete matrix structure. In particular, it shows that the rectangular additive convolution is $\theta-$independent.

Our next result is the law of large numbers of $\Vec{c}=\Vec{a}\boxplus_{N,M}^{\theta}\Vec{b}$ in the regime $\theta\rightarrow \infty$. As preparation we state a combinatorial result. For partitions $v$, $\mu$, $\lambda$ such that $\max(v_{1}, \mu_{1})\le \lambda_{1}$, $l(v)$, and $ \max(l(v),l(\mu), l(\lambda))\le N$,
let $\{(k,l)\}$ be an index set that $l=1,2,\dots,\lambda_{k}-\lambda_{k+1}, k=1,2,\dots,N$. Furthermore, let $\{i(k,l)\}$, $\{j(k,l)\}$ be two collections of nonnegative integers. We do not distinguish \smash{$\{i(k,l)\}_{l=1}^{\lambda_{k}-\lambda_{k+1}}$} with \smash{$\{i(k,\sigma(l))\}_{l=1}^{\lambda_{k}-\lambda_{k+1}}$}, where $\sigma\in S_{\lambda_{k}-\lambda_{k+1}}$ is an arbitrary permutation, and same for \smash{$\{j(k,l)\}_{l=1}^{\lambda_{k}-\lambda_{k+1}}$}.

In the remainder of the text, let $\mathbf{1}_{E}$ denote the indicator function of the set $E$.

\begin{prop}\label{prop:combcounting}
Let $C^{v,\mu}_{\lambda}$ be the coefficient of $m_{\lambda}(\vec{x})$ in the expansion
\[m_{v}(\vec{x})\cdot m_{\mu}(\vec{x})=\sum_{\lambda}C^{v,\mu}_{\lambda}m_{\lambda}(\vec{x}).\]
Then
\smash{$
C^{v',\mu'}_{\lambda'}=$}\# ways to choose $\{i(k,l)\}$, $\{j(k,l)\}$ such that for $n=1,2,\dots,N$,
\begin{align}\label{eq_combcounting}
 v_{n}=\sum_{k=1}^{N}\sum_{l=1}^{\lambda_{k}-\lambda_{k+1}}\mathbf{1}_{i(k,l)\ge n}, \qquad
 \mu_{n}=\sum_{k=1}^{N}\sum_{l=1}^{\lambda_{k}-\lambda_{k+1}}\mathbf{1}_{j(k,l)\ge n}.
\end{align}
\end{prop}
\begin{proof}
We choose $\{i(k,l)\}$, $\{j(k,l)\}$ in an explicit way. By definition of $C^{v,\mu}_{\lambda}$, we are combining column $v'_{l_{1}}$ with column $\mu'_{l_{2}}$ to get a column $\lambda'_{l_{3}}$, where $l_{1}$, $l_{2}$, $l_{3}$ are chosen among $1,2,\dots,\lambda_{1}$, and $v'_{l_{1}}$, $\mu'_{l_{2}}$ might be of length 0. Inspired by this, let \smash{$\{i(k,l)\}_{l=1}^{\lambda_{k}-\lambda_{k+1}}$} be the length of (distinct) columns of $v$, \smash{$\{j(k,l)\}_{l=1}^{\lambda_{k}-\lambda_{k+1}}$} be the length of (distinct) columns of $\mu$, which are chosen to contribute to \smash{$\lambda'_{\lambda_{k+1}+1},\dots,\lambda'_{\lambda_{k}}$}. Then $i(k,l)+j(k,l)=k$ for each $l$.

We immediately see that the above way to choose $\{i(k,l)\}$, $\{j(k,l)\}$ satisfy~\eqref{eq_combcounting}, whose total number is equal to \smash{$C^{v',\mu'}_{\lambda'}$}. It remains to check that each way of choosing $\{i(k,l)\}$, $\{j(k,l)\}$ can be interpreted in this way. Given a sequence of nonnegative integers $\{i(k,l)\}$, $\{j(k,l)\}$ satisfying~\eqref{eq_combcounting}, we have for $n=1,2,\dots,N$,
\begin{equation}\label{eq_combcounting2}
 v_{n}-v_{n+1}=\sum_{k=1}^{N}\sum_{l=1}^{\lambda_{k}-\lambda_{k+1}}\mathbf{1}_{i(k,l)=n}, \qquad
 \mu_{n}-\mu_{n+1}=\sum_{k=1}^{N}\sum_{l=1}^{\lambda_{k}-\lambda_{k+1}}\mathbf{1}_{i(k,l)= n}.
\end{equation}
Then one can split \smash{$\{i(k,l)\}_{l=1}^{\lambda_{k}-\lambda_{k+1}}$}, $k=1,2,\dots,N$, into disjoint groups, such that the number of elements in group $n$ is exactly $v_{n}-v_{n+1}$, which is equal to the number of length $n$ columns in $v$. Vice versa for \smash{$\{j(k,l)\}_{l=1}^{\lambda_{k}-\lambda_{k+1}}$}, $k=1,2,\dots,N$.
\end{proof}

\begin{proof}[Proof of Theorem~\ref{thm:lln}:]
The weak convergence to a delta function on polynomial test functions is equivalent to the statement that, given any arbitrary collection of polynomials $f_{1},\dots,f_{n}$ of N variables, we have
\begin{equation}\label{eq_lln1}
\lim_{\theta\rightarrow \infty}\E{\prod_{i=1
}^{n}f_{i}\bigl(\Vec{c^{2}}\bigr)}=\lim_{\theta\rightarrow\infty}\prod_{i=1}^{n}\bigl[\mathbb{E}\big[f_{i}\bigl(\Vec{c^{2}}\bigr)\big]\bigr].
\end{equation}

Since $\Vec{c}$ is symmetric in distribution, it suffices to consider symmetric polynomials in $\Lambda_{N}$, which can be generated (in the sense of algebra) by elementary symmetric functions $e_1,\dots,e_{N}$. Since~\eqref{eq_lln1} is multilinear in $f_{i}$, we reduce to showing for any positive integers $k_1,\dots,k_{N}$,
\begin{equation}\label{eq_lln2}
\lim_{\theta\rightarrow \infty}\E{\prod_{i=1
}^{N}e_{i}\bigl(\Vec{c^{2}}\bigr)^{k_{i}}}\stackrel{?}{=}\lim_{\theta\rightarrow\infty}\prod_{i=1}^{N}\bigl[\mathbb{E}\big[e_{i}\bigl(\Vec{c^{2}}\bigr)\big]\bigr]^{k_{i}}.
\end{equation}
Once we show this, the deterministic limit of $\Vec{c}$ will be an $N$-tuple $\Vec{\lambda}$,
such that $\EE\bigl[e_{i}(\Vec{c})\bigr]=e_{i}\bigl(\Vec{\lambda}\bigr)$ for all $i=1,2,\dots,N$. Then Theorem~\ref{thm:characteristicpoly} identifies \smash{$\Vec{\lambda}$} with roots of $P^{\theta}_{N,M}(z)$.

We connect the left side of~\eqref{eq_lln2} with Jack polynomials, using the following result \cite[Proposition~7.6]{St}:
\begin{equation}\label{eq_lln3}
\lim_{\theta\rightarrow \infty}P_{\lambda}(x_{1},\dots,x_{N};\theta)=\prod_{i=1}^{N}[e_{i}(x_{1},\dots,x_{N})]^{\lambda_{i}-\lambda_{i+1}}, \end{equation}
for any partition $\lambda$.
Then, let $\lambda_{i}=k_{i}+\dots+k_{N}$, the left side of~\eqref{eq_lln2} becomes the limit of~\smash{$\mathbb{E}\bigl[P_{\lambda}\bigl(\Vec{c^{2}};\theta\bigr)\bigr]$}, since each product of the \smash{$e_{i}\bigl(\Vec{c^{2}}\bigr)$} has bounded expectation for $\theta>0$ due to the fact that $\Vec{c}$ has bounded support.
Again by Proposition~\ref{prop:polyexpectation},
\begin{gather}
\E{P_{\lambda}\bigl(c_{1}^{2},\dots,c_{N}^{2};\theta\bigr)}\nonumber
\\
\qquad=\sum_{|v|+|\mu|=|\lambda|}\frac{H(\lambda)}{H(v)H(\mu)}\frac{\prod_{i=1}^{N}\Gamma(\theta (M-i+1))\Gamma(\theta (M-i+1)+\lambda_{i})}{\prod_{i=1}^{N}\Gamma(\theta (M-i+1)+v_{i})\Gamma(\theta (M-i+1)+\mu_{i})}\nonumber\\
\phantom{\qquad=}{}\times\frac{P_{\lambda}\bigl(1^{N};\theta\bigr)}{P_{v}\bigl(1^{N};\theta\bigr)P_{\mu}\bigl(1^{N};\theta\bigr)}C^{v,\mu}_{\lambda}(\theta)P_{v}\bigl(a_{1}^{2},
\dots,a_{N}^{2};\theta\bigr)P_{\mu}\bigl(b_{1}^{2},\dots,b_{N}^{2};\theta\bigr).\label{eq_lln6}
\end{gather}
Taking $\theta\rightarrow\infty$,
\[\frac{\prod_{i=1}^{N}\Gamma(\theta (M-i+1))\Gamma(\theta (M-i+1)+\lambda_{i})}{\prod_{i=1}^{N}\Gamma(\theta (M-i+1)+v_{i})\Gamma(\theta (M-i+1)+\mu_{i})}\longrightarrow \prod_{n=1}^{N}(M-n+1)^{\lambda_{n}-v_{n}-\mu_{n}},\]
and since by definition
\smash{$P_{v}(\vec{x};\theta)P_{\mu}(\vec{x};\theta)=\sum_{\lambda}C^{v,\mu}_{\lambda}(\theta)P_{\lambda}(\vec{x};\theta)$},
applying $\omega_{\theta}$ on both sides (c.f.\ the proof of Lemma~\ref{lem:automorphism}), and use the fact that (see \cite[Proposition~7.6]{St})
\begin{equation}\label{eq_lln4}
\lim_{\theta\rightarrow 0}P_{\lambda}(x_{1},\dots,x_{N};\theta)=m_{\lambda}(x_{1},\dots,x_{N}),
\end{equation}
we have
\begin{gather*}
C^{v,\mu}_{\lambda}(\theta)\frac{H(\lambda)}{H(v)H(\mu)}\longrightarrow C^{v',\mu'}_{\lambda'}\cdot \frac{\lim_{\theta\rightarrow \infty} H'(\lambda)}{\lim_{\theta\rightarrow \infty} H'(v)\cdot \lim_{\theta\rightarrow \infty} H'(\mu)}\\
\qquad=C^{v',\mu'}_{\lambda'}\cdot \frac{\prod_{s\in \lambda}(\lambda_{j}'-i+1)}{\prod_{s\in v}(\lambda_{j}'-i+1)\cdot \prod_{s\in \mu}(\lambda_{j}'-i+1)}.
\end{gather*}
Moreover, applying~\eqref{eq_lln3} again on \smash{$\frac{P_{\lambda}(1^{N};\theta)}{P_{v}(1^{N};\theta)P_{\mu}(1^{N};\theta)}P_{v}\bigl(a_{1}^{2},\dots,a_{N}^{2};\theta\bigr)P_{\mu}\bigl(b_{1}^{2},\dots,b_{N}^{2};\theta\bigr)$}, the right side of~\eqref{eq_lln6} goes to
\begin{gather}
\sum_{|v|+|\mu|=|\lambda|}C^{v',\mu'}_{\lambda'}\frac{\prod_{s\in \lambda}(\lambda_{j}'-i+1)}{\prod_{s\in v}(\lambda_{j}'-i+1)\cdot \prod_{s\in \mu}(\lambda_{j}'-i+1)}\frac{\prod_{i=1}^{N}\binom{N}{i}^{\lambda_{i}-\lambda_{i+1}}}{\prod_{i=1}^{N}\binom{N}{i}^{v_{i}-v_{i+1}}\prod_{i=1}^{N}\binom{N}{i}^{\mu_{i}-\mu_{i+1}}}\nonumber\\
\qquad\times\prod_{n=1}^{N}(M-n+1)^{\lambda_{i}-v_{i}-\mu_{i}}\prod_{i=1}^{N}\bigl[e_{i}\bigl(a_{1}^{2},\dots,a_{N}^{2}\bigr)\bigr]^{v_{i}-v_{i+1}}\prod_{i=1}^{N}
\bigl[e_{i}\bigl(b_{1}^{2},\dots,b_{N}^{2}\bigr)\bigr]^{\mu_{i}-\mu_{i+1}}.\label{eq_lln7}
\end{gather}

On the other hand, by Theorem~\ref{thm:characteristicpoly}, the right side of~\eqref{eq_lln2} is equal to
\begin{align}
\prod_{k=1}^{N}\bigl[\EE\bigl[e_{i}\bigl(\Vec{c^{2}}\bigr)\bigr]\bigr]^{\lambda_{k}-\lambda_{k+1}}
={}&\prod_{k=1}^{N}\left[\sum_{i+j=k}\frac{(N-i)!(N-j)!}{N!(N-k)!}\frac{(M-i)!(M-j)!}{M!(M-k)!}\right.\nonumber\\
&\left.{}\times e_{i}\bigl(a_{1}^{2},\dots,a_{N}^{2}\bigr)e_{j}\bigl(b_{1}^{2},\dots,b_{N}^{2}\bigr)\right]^{\lambda_{k}-\lambda_{k+1}}.\label{eq_lln8}
\end{align}
It remains to check that~\eqref{eq_lln7} is equal to~\eqref{eq_lln8}.

We open the bracket in~\eqref{eq_lln8}, and index each term in the inner sum with the same indices $(i(k,l),j(k,l))\colon l=1,2,\dots,\lambda_{k}-\lambda_{k+1}$, $k=1,2,\dots,N$ introduced previously. Then we have $i(k,l)+j(k,l)=k$ for each $l$. The product in~\eqref{eq_lln8} then becomes a finite sum, where each summand is a multiple of the form \[\prod_{n=1}^{N}\bigl[e_{n}\bigl(a_{1}^{2},\dots,a_{N}^{2}\bigr)\bigr]^{v_{n}-v_{n+1}}\prod_{n=1}^{N}\bigl[e_{n}\bigl(b_{1}^{2},\dots,b_{N}^{2}\bigr)\bigr]^{\mu_{n}-\mu_{n+1}},\] and for $n=1,2,\dots,N$,
\begin{gather*}
v_{n}-v_{n+1}=\sum_{k=1}^{N}\sum_{l=1}^{\lambda_{k}-\lambda_{k+1}}\mathbf{1}_{i(k,l)=n},\qquad
\mu_{n}-\mu_{n+1}=\sum_{k=1}^{N}\sum_{l=1}^{\lambda_{k}-\lambda_{k+1}}\mathbf{1}_{i(k,l)=n},\\
\lambda_{n}-\lambda_{n+1}=\sum_{k=1}^{N}\sum_{l=1}^{\lambda_{k}-\lambda_{k+1}}\mathbf{1}_{k=n},
\end{gather*}
which matches~\eqref{eq_combcounting2}. Hence it remains to match the coefficients, i.e., to show that
\begin{gather}
\sum_{i(k,l)+j(k,l)=k}\prod_{k=1}^{N}\prod_{l=1}^{\lambda_{k}-\lambda_{k+1}}\left[\frac{(N-i(k,l))!(N-j(k,l))!}{N!(N-k)!}\frac{(M-i(k,l))!(M-j(k,l))!}{M!(M-k)!}\right]\nonumber\\
\qquad\stackrel{?}{=}C^{v',\mu'}_{\lambda'}\frac{\prod_{s\in \lambda}[\lambda_{j}'-i+1]}{\prod_{s\in v}[\lambda_{j}'-i+1]\cdot \prod_{s\in \mu}[\lambda_{j}'-i+1]}\nonumber\\
\phantom{\qquad\stackrel{?}{=}}{}\times\frac{\prod_{i=1}^{N}\binom{N}{i}^{\lambda_{i}-\lambda_{i+1}}}{\prod_{i=1}^{N}\binom{N}{i}^{v_{i}-v_{i+1}}\prod_{i=1}^{N}\binom{N}{i}^{\mu_{i}-\mu_{i+1}}}
\prod_{n=1}^{N}(M-n+1)^{\lambda_{n}-v_{n}-\mu_{n}}.\label{eq_lln14}
\end{gather}

We first rewrite the left side
\[
\frac{(N-i(k,l))!(N-j(k,l))!}{(N-k)!N!}
=\prod_{n=1}^{N}(N-n+1)^{\mathbf{1}_{n\le k}-\mathbf{1}_{n\le i(k,l)}-\mathbf{1}_{n\le j(k,l)}},
\]
hence
\begin{gather*}
\prod_{k=1}^{N}\prod_{l=1}^{\lambda_{k}-\lambda_{k+1}}\left[\frac{(N-i(k,l))!(N-j(k,l))!}{(N-k)!N!}\frac{(M-i(k,l))!(M-j(k,l))!}{(M-k)!M!}\right]\\
\qquad=\prod_{n=1}^{N}(N-n+1)^{\sum_{k=1}^{N}\sum_{l=1}^{\lambda_{k}-\lambda_{k+1}}[\mathbf{1}_{n\le k}-\mathbf{1}_{n\le i(k,l)}-\mathbf{1}_{n\le j(k,l)}]}\\
\phantom{\qquad=}{}\times\prod_{n=1}^{N}(M-n+1)^{\sum_{k=1}^{N}\sum_{l=1}^{\lambda_{k}-\lambda_{k+1}}[\mathbf{1}_{n\le k}-\mathbf{1}_{n\le i(k,l)}-\mathbf{1}_{n\le j(k,l)}]}.
\end{gather*}
On the right side,
\begin{equation}
 \prod_{n=1}^{N}(M-n+1)^{\lambda_{n}-v_{n}-\mu_{n}}
 =\prod_{n=1}^{N}(M-n+1)^{\sum_{k=1}^{N}\sum_{l=1}^{\lambda_{k}-\lambda_{k+1}}(\mathbf{1}_{n\le k}-\mathbf{1}_{n\le i(k,l)}-\mathbf{1}_{n\le j(k,l)})} \label{eq_lln11}
\end{equation}
and
\[
 \binom{N}{i}=\prod_{n=1}^{N}(N-n+1)^{1-\mathbf{1}_{n\ge i+1}-\mathbf{1}_{n\ge N-i+1}}=\prod_{n=1}^{N}(N-n+1)^{\mathbf{1}_{n\le i}-\mathbf{1}_{n\ge N-i+1}},
\]
hence
\begin{gather}
\frac{\prod_{i=1}^{N}\binom{N}{i}^{\lambda_{i}-\lambda_{i+1}}}{\prod_{i=1}^{N}\binom{N}{i}^{v_{i}-v_{i+1}}\prod_{i=1}^{N}\binom{N}{i}^{\mu_{i}-\mu_{i+1}}}\nonumber\\
\qquad=\prod_{i=1}^{N}\binom{N}{i}^{\sum_{k=1}^{N}\sum_{l=1}^{\lambda_{k}-\lambda_{k+1}}[\mathbf{1}_{k=i}-\mathbf{1}_{i(k,l)=i}-\mathbf{1}_{i(k,l)=i}]}\nonumber\\
\qquad=\prod_{n=1}^{N}(N-n+1)^{\sum_{i=1}^{N}(\mathbf{1}_{n\le i}-\mathbf{1}_{n\ge N-i+1})(\sum_{k=1}^{N}\sum_{l=1}^{\lambda_{k}-\lambda_{k+1}}[\mathbf{1}_{k=i}-\mathbf{1}_{i(k,l)=i}
-\mathbf{1}_{i(k,l)=i}])}\nonumber\\
\qquad=\prod_{n=1}^{N}(N-n+1)^{\sum_{k=1}^{N}\sum_{l=1}^{\lambda_{k}-\lambda_{k+1}}
\text{\scriptsize$\begin{array}{@{}l@{}} \bigl([\mathbf{1}_{n\le k}-\mathbf{1}_{n\le i(k,l)}-\mathbf{1}_{n\le j(k,l)}]\\
{}+[\mathbf{1}_{n\ge N-i(k,l)+1}+\mathbf{1}_{n\ge N-j(k,l)+1}-\mathbf{1}_{n\ge N-k+1}]\bigr)\end{array}$}},\label{eq_lln12}
\end{gather}
where the last equality holds since
\begin{gather*}
\sum_{i=1}^{N}(\mathbf{1}_{n\le i}-\mathbf{1}_{n\ge N-i+1})(\mathbf{1}_{k=i}-\mathbf{1}_{i(k,l)=i}-\mathbf{1}_{i(k,l)=i})\\
\qquad=[\mathbf{1}_{n\le i}-\mathbf{1}_{n\le i(k,l)}-\mathbf{1}_{n\le j(k,l)}]+[\mathbf{1}_{n\ge N-i(k,l)+1}+\mathbf{1}_{n\ge N-j(k,l)+1}-\mathbf{1}_{n\ge N-k+1}].\end{gather*}

Finally,
\begin{gather}
 \prod_{n=1}^{N}(N-n+1)^{\sum_{k=1}^{N}\sum_{l=1}^{\lambda_{k}-\lambda_{k+1}}[-(\mathbf{1}_{n\ge N-i(k,l)+1}+\mathbf{1}_{n\ge N-j(k,l)+1}-\mathbf{1}_{n\ge N-k+1})]}\nonumber\\
 \qquad=\prod_{n=1}^{N}(N-n+1)^{\sum_{k=1}^{N}\sum_{l=1}^{\lambda_{k}-\lambda_{k+1}}[-(\mathbf{1}_{i(k,l)\ge N-n+1}-\mathbf{1}_{i(k,l)\ge N-n+1}+\mathbf{1}_{k\ge N-n+1})]}\nonumber\\
 \qquad=\prod_{n=1}^{N}n^{\sum_{k=1}^{N}\sum_{l=1}^{\lambda_{k}-\lambda_{k+1}}[\mathbf{1}_{n\le k}-\mathbf{1}_{n\le i(k,l)}-\mathbf{1}_{n\le j(k,l)}]}\nonumber\\
 \qquad=\frac{\prod_{j=1}^{\lambda_{1}}\lambda_{j}'!}{\prod_{j=1}^{\lambda_{1}}v_{j}'!\prod_{j=1}^{\lambda_{1}}\mu_{j}'!}
 =\frac{\prod_{s\in \lambda}[\lambda_{j}'-i+1]}{\prod_{s\in v}[\lambda_{j}'-i+1]\prod_{s\in \mu}[\lambda_{j}'-i+1]}.\label{eq_lln13}
\end{gather}
Formula \eqref{eq_lln14} follows from~\eqref{eq_lln11}, \eqref{eq_lln12}, \eqref{eq_lln13} and Proposition~\ref{prop:combcounting}.
\end{proof}

\subsection[Gaussian fluctuation for 1times M matrix]{Gaussian fluctuation for $\boldsymbol{1\times M}$ matrix}
Take $N=1$, so that $A$ and $B$ are two $1\times M$ matrices with singular values $a_{1}, b_{1}\ge 0$, and let $c_{1}=a_{1}\boxplus_{1,M}^{\theta}b_{1}$. When taking $\theta\rightarrow\infty$, Theorem~\ref{thm:lln} shows
\smash{$c_{1}^{2}\longrightarrow \lambda_{1}^{2}$}
in moments, where
\[
 \E{c_{1}^{2}}=\E{e_{1}\bigl(c^{2}\bigr)}=a_{1}^{2}+b_{1}^{2}=\lambda_{1}^{2}.
\]

Based on this result, we consider further the fluctuation of $c_{1}$ around $\lambda_{1}$ in the $\theta\rightarrow\infty$ regime, which turns out to be a Gaussian random variable under proper rescaling.
\begin{thm}\label{thm:fluctuation}
 For $a_{1},b_{1}\ge 0$, let $\lambda_{1}^{2}=a_{1}^{2}+b_{1}^{2}$, and $c_{1}=a_{1}\boxplus_{1,M}^{\theta}b_{1}$. As $\theta\rightarrow\infty$, we have
\begin{equation}\label{eq_fluctuation2}
\sqrt{\theta}\bigl(c_{1}^{2}-\lambda_{1}^{2}\bigr)\stackrel{d}{\longrightarrow} Z,
\end{equation}
where $Z\sim\mathcal{N}\bigl(0, \frac{2}{M}a_{1}^{2}b_{1}^{2}\bigr)$.
\end{thm}
\begin{Remark}
 We expect the Gaussian fluctuation behavior of $\Vec{c}=\Vec{a}\boxplus_{N,M}^{\theta}\Vec{b}$ when $\theta\rightarrow\infty$, for general $N>1$, and we leave the generalization of Theorem~\ref{thm:fluctuation} as an open problem.
\end{Remark}
\begin{proof}
 We first show that the convergence holds in the sense of moments. By Proposition~\ref{prop:polyexpectation}, for all $l=1,2,\dots $
\[
\mathbb{E}\bigl[c_{1}^{2l}\bigr]=\sum_{k_{1}+k_{2}=l}\frac{l!}{k_{1}!k_{2}!}\frac{\Gamma(\theta M)\Gamma(\theta M+l)}{\Gamma(\theta M+k_{1})\Gamma(\theta M+k_{2})}a_{1}^{2k_{1}}b_{1}^{2k_{2}},
\]
since in~\eqref{eq_polyexpectation} $\lambda=(l,0,\dots )$, $v=(k_{1},0,\dots )$, $\mu=(k_{2},0,\dots )$ and $C^{v,\mu}_{\lambda}(\theta)\equiv 1$.

Then for $n\in \Z_{\ge 0}$,
\begin{gather*}
 \E{\bigl(c_{1}^{2}-\lambda_{1}^{2}\bigr)^{n}}\\
 \qquad=\sum_{k_{1}+k_{2}+k=n}\frac{(-1)^{k}n!}{(k_{1}+k_{2})!k!}\frac{(k_{1}+k_{2})!}{k_{1}!k_{2}!}\frac{\Gamma(\theta M)\Gamma(\theta M+k_{1}+k_{2})}{\Gamma(\theta M+k_{1})\Gamma(\theta M+k_{2})}a_{1}^{2k_{1}}b_{1}^{2k_{2}}\bigl(a_{1}^{2}+b_{1}^{2}\bigr)^{k}\\
 \qquad=\sum_{k_{1}+k_{2}+k_{3}+k_{4}=n}(-1)^{k}\frac{n!}{k_{1}!k_{2}!k_{3}!k_{4}!}\frac{\Gamma(\theta M)\Gamma(\theta M+k_{1}+k_{2})}{\Gamma(\theta M+k_{1})\Gamma(\theta M+k_{2})}a_{1}^{2k_{1}+2k_{3}}b_{1}^{2k_{2}+2k_{4}}.
\end{gather*}
This implies for fixed $l_{1}+l_{2}=n\ge 0$, coefficient of monomial \smash{$a_{1}^{2l_{1}}b_{1}^{2l_{2}}$} in $\mathbb{E}\bigl[\bigl[\sqrt{\theta}\bigl(c_{1}^{2}-\lambda_{1}^{2}\bigr)\bigr]^{n}\bigr]$ is
\begin{gather*}
 \sqrt{\theta}^{l_{1}+l_{2}}n!\sum_{k_{3}=0}^{l_{1}}\sum_{k_{4}=0}^{l_{2}}\frac{(-1)^{k_{3}}}{(l_{1}-k_{3})!k_{3}!}\frac{(-1)^{k_{4}}}{(l_{2}-k_{4})!k_{4}!}\frac{\Gamma(\theta M)\Gamma(\theta M+l_{1}+l_{2}-k_{3}-k_{4})}{\Gamma(\theta M+l_{1}-k_{3})\Gamma(\theta M+l_{2}-k_{4})}\nonumber\\
 \qquad=\sqrt{\theta}^{l_{1}+l_{2}}n!\sum_{k_{3}=0}^{l_{1}}\sum_{k_{4}=0}^{l_{2}}\frac{(-1)^{l_{1}-k_{3}}}{(l_{1}-k_{3})!k_{3}!}\frac{(-1)^{l_{2}-k_{4}}}{(l_{2}-k_{4})!k_{4}!}\frac{\Gamma(\theta M)\Gamma(\theta M+k_{3}+k_{4})}{\Gamma(\theta M+k_{3})\Gamma(\theta M+k_{4})}.\label{eq_fluctuation4}
\end{gather*}

It remains to match the above expression with moments of $Z$. Without loss of generality, assume that $l_{1}\ge l_{2}$ in~\eqref{eq_fluctuation4}, and we rewrite~\eqref{eq_fluctuation4} as
\begin{gather*}
\sqrt{\theta}^{l_{1}+l_{2}}n!\sum_{k_{4}=0}^{l_{2}}\frac{(-1)^{l_{2}-k_{4}}}{(l_{2}-k_{4})!k_{4}!}\frac{\Gamma(\theta M)}{\Gamma(\theta M+k_{4})}\\
\qquad\times\Bigg[\sum_{k_{3}=0}^{l_{1}}\frac{(-1)^{l_{1}-k_{3}}}{(l_{1}-k_{3})!k_{3}!}(\theta M+k_{3})(\theta M+k_{3}+1)\cdots (\theta M+k_{3}+k_{4}-1)\Bigg].
\end{gather*}
\begin{Claim}
Let $z$ be a formal variable. For $l\in \Z_{\ge 1}$, $q=0,1,2,\dots,l$, we have
\begin{equation}\label{eq_fluctuationlemma1}
 \sum_{p=0}^{l}\frac{(-1)^{(l-p)}}{(l-p)!p!}(z+p)(z+p+1)\cdots (z+p+q-1)=\mathbf{1}_{q=l}.
 \end{equation}
 \end{Claim}

\begin{proof} Expand the polynomial of $z$. For each coefficient of $z^{0}, z^{1}, z^{2},\dots, z^{q}$, it can be written as a polynomial of $p$ with degree at most $q$ with integer coefficients, and hence an integral linear combination of $p, p(p-1), p(p-1)(p-2),\dots, p(p-1)\cdots (p-q+1)$, so the left side of~\eqref{eq_fluctuationlemma1} becomes an integral linear combination of binomial sums of
\[
\sum_{p=q}^{l}\frac{(-1)^{l-p}}{(l-p)!(p-q)!}=(1+(-1))^{l-q}=\mathbf{1}_{q=l}.\tag*{\qed}
\] \renewcommand{\qed}{}
\end{proof}

By the above claim, the sum in the bracket is nonzero only when $k_{4}=l_{1}$, which implies $l_{1}=l_{2}$, $n=l_{1}+l_{2}=2l_{1}$, and~\eqref{eq_fluctuation4} becomes
\[\sqrt{\theta}^{2l_{1}}\frac{(2l_{1})!}{l_{1}!}\frac{1}{\theta M(\theta M+1)\cdots (\theta M+l_{1}-1)}.\]
Therefore, the odd moments of $\sqrt{\theta}\bigl(c_{1}^{2}-\lambda_{1}^{2}\bigr)$ are all zero, and the $2k^{\rm th}$ moment of $\sqrt{\theta}\bigl(c_{1}^{2}-\lambda_{1}^{2}\bigr)$ is equal to
\[\sqrt{\theta}^{2k}\frac{(2k)!}{k!}\frac{1}{\theta M(\theta M+1)\cdots (\theta M+k-1)}a_{1}^{2k}b_{1}^{2k},\]
which converges to
\[\frac{(2k)!}{k!}\frac{1}{M^{k}}a_{1}^{2k}b_{1}^{2k}=(2k-1)!!\left(\frac{2}{M}\right)^{k}a_{1}^{2k}b_{1}^{2k}\]
as $\theta\rightarrow\infty$.
This coincides with the moments of $Z\sim \mathcal{N}\bigl(0,\frac{2}{M}a_{1}^{2}b_{1}^{2}\bigr)$. By Example~\ref{ex:usualbessel} and \cite[Theorem 2]{MP}, which states that products of two usual Bessel functions can be written as a convex combination of Bessel functions, we have that for each $\theta>0$, $c_{1}^{2}$ is supported by a legitimate probability measure $\mu_{\theta}$. The convergence of second moment when $\theta\rightarrow\infty$ implies that $\{\mu_{\theta}\}_{\theta>0}$ are tight, hence~\eqref{eq_fluctuation2} follows from the moment convergence.
\end{proof}

\section{Law of large numbers in high temperature}\label{sec:lln}
In this section, fix two parameters $\gamma>0, q\ge 1$. We explore the behavior of empirical measures of an $N\times M$ random matrix $C$, in the regime that taking $N,M\rightarrow\infty$, $\theta\rightarrow 0$, $N\theta\rightarrow\gamma$, $M\theta\rightarrow q\gamma$. To simplify the notation, sometimes we only write $N\theta\rightarrow\gamma$, $M\theta\rightarrow q\gamma$ to denote the same regime.

\subsection{Main results}

Consider an $N$-tuple of real numbers $\Vec{c}=(c_{1}\ge \dots \ge c_{N}\ge 0)$, which should be thought of as singular values of some (virtual) rectangular matrix. Suppose that there is a sequence of random $N$-tuples $\{\Vec{c}_{N}\}_{N=1}^{\infty}$ where $\Vec{c}_{N}=(c_{N,1}\ge\dots \ge c_{N,N}\ge 0)$, and the distribution of $\Vec{c}_{N}$ is given in the sense as in Theorem~\ref{thm:dunklonbgf}. Denote its empirical measure by \smash{$\mu_{N}=\frac{1}{N}\sum_{i=1}^{N}(\delta_{c_{N,i}}+\delta_{-c_{N,i}})$}.

We set up a condition in terms of moments, that under some mild technical assumption, is equivalent to the weak convergence, in probability, of the random empirical measures $\{\mu_{N}\}$ to some limiting probability measure $\mu$ when $N\rightarrow \infty$. The moments of $\mu$ are all finite and denoted as $m_{k}$.

\begin{Definition}[LLN]\label{def:llnsatisfaction}
 Let $\{\Vec{c}_{N}\}_{N=1}^{\infty}$ be a sequence of random $N$-tuples defined as above. For $k=1,2,\dots$, denote
 \[p^{N}_{k}=\frac{1}{2N}\sum_{i=1}^{N}\bigl[c_{N,i}^{k}+(-c_{N,i})^{k}\bigr].\]
 We say $\{\Vec{c}_{N}\}$ \emph{satisfies a law of large numbers}, if there exists deterministic real numbers $\{m_{k}\}_{k=1}^{\infty}$ such that for any $s=1,2,\dots$ and any $k_{1},\dots,k_{s}\in \Z_{\ge 1}$, we have
 \begin{equation}\label{eq_rectangularlln}
 \lim_{N\rightarrow \infty}\E{\prod_{i=1}^{s}p_{k_{i}}^{N}}=\prod_{i=1}^{s}m_{k_{i}}.
 \end{equation}
\end{Definition}

Denote the Bessel generating function of $\Vec{c}_{N}$ by
\[G_{N,M;\theta}(x_{1},\dots,x_{N}):=G_{N,\theta}(x_{1},\dots,x_{N};\mathfrak{m}_{\Vec{c}_{N}}).\]
Recall from Section~\ref{sec:bgf} that \smash{$G_{N,M;\theta}(0,\dots,0)=1$}, and \smash{$G_{N,M;\theta}(x_{1},\dots,x_{N})$} is analytic on a~domain near $(0,\dots,0)$. Under these conditions, $\ln(G_{N,M;\theta}(x_{1},\dots,x_{N}))$ is analytic near ${(0,\dots,0)}$, and $\ln(G_{N,M;\theta}(0,\dots,0))=0$.

Next, we introduce a condition of the partial derivatives of ln$(G_{N,M;\theta}(x_{1},\dots,x_{N}))$ at 0, as~${N\rightarrow\infty}$.

\begin{Definition}[$q$-$\gamma$-LLN-appropriateness]\label{def:llnappropriateness}
Given the sequence $\{\Vec{c}_{N}\}_{N=1}^{\infty}$, if for a sequence of real numbers $\{\kappa_{l}\}_{l=1}^{\infty}$, the following limits hold:
\begin{itemize}\itemsep=0pt
\item[(a)] \smash{$ \lim_{N\theta\rightarrow\gamma, M\theta\rightarrow q\gamma}\frac{\partial^{l}}{\partial x_{i}^{l}}\ln(G_{N,M;\theta})\bigr|_{x_{1}=\dots =x_{N}=0}=(l-1)!\cdot \kappa_{l}$}, for all $l\in\Z_{\ge 1}$ and $i\in \{1,2,\dots,\allowbreak N\}$.
\item[(b)] \smash{$ \lim_{N\theta\rightarrow\gamma, M\theta\rightarrow q\gamma}\frac{\partial}{\partial x_{i_{1}}}\cdots \frac{\partial}{\partial x_{i_{r}}}\ln(G_{N,M;\theta})\bigr|_{x_{1}=\dots =x_{N}=0}=0$}, for all $r\ge 2$, and $i_{1},\dots,i_{r}\in \{1,2,\dots,N\}$ such that the set $\{i_{1},\dots,i_{r}\}$ is of cardinality at least two.
\end{itemize}
 We say $\{\kappa_{l}\}_{l=1}^{\infty}$ are the limiting $q$-$\gamma$ cumulants of $\{\Vec{c}_{N}\}$.
\end{Definition}
\begin{Remark}
 Recall from Definition~\ref{def:bessel} that $B(\vec{a},x_{1},\dots,x_{N};\theta,N,M)$ is an infinite sum of even-degree polynomials, then so is $G_{\theta,N,M}$ and $\ln(G_{\theta,N,M})$. Therefore, it is immediate that $\kappa_{l}$ are always 0 for all odd $l$.
\end{Remark}
\begin{Remark}
 Writing
 \[g^{\theta,N,M}(z)=\frac{\partial}{\partial z}\ln(G_{N,M;\theta}(z,0,\dots,0))=\sum_{l=1}^{\infty}k^{\theta,N,M}_{l}z^{l-1},\]
 we have \smash{$k^{\theta,N,M}_{l}\longrightarrow \kappa_{l}$}
as $N\theta\rightarrow\gamma$, $M\theta\rightarrow q\gamma$.
\end{Remark}

Our main theorem connects Definitions~\ref{def:llnsatisfaction}, \ref{def:llnappropriateness} and gives a quantitative relation between moments and $q$-$\gamma$ cumulants of the limiting empirical measure of $\{\mu_{N}\}_{N=1}^{\infty}$, which is stated using generating function. Consider $\R[[z]]$, the space of all formal power series of variable $z$ with real coefficients.
\begin{Definition}
 Let $a(z)$ be an element in $\R[[z]]$. We define four linear operators acting on~$\R[[z]]$ to itself, such that for any $n=0,1,2,\dots $,
 \begin{itemize}\itemsep=0pt
 \item[(1)] $\partial(z^{n}):=n\cdot z^{n-1}$,
 \item[(2)] $d(z^{n}):=\begin{cases}
 0, & n=0, \\
 z^{n-1}, & n\ge 1,\\
\end{cases}$
 \item[(3)] $d'(z^{n}):=\begin{cases}
 0 ,& \text{$n$ is even}, \\
 2z^{n-1}, & \text{$n$ is odd},\\
\end{cases}$
 \item[(4)] $*_{a}(z^{n}):=a(z)\cdot z^{n}$.
 \end{itemize}
\end{Definition}

\begin{Definition}\label{def:operators}
 Let $\mathrm{T_{\kappa\rightarrow m}^{q,\gamma}}\colon \R^{\infty}\rightarrow \R^{\infty}$
be an operation sending a countable sequence $\{\kappa_{l}\}_{l=1}^{\infty}$ to another sequence $\{m_{2k}\}_{k=1}^{\infty}$ such that for each $k=1,2,\dots $,
 \begin{equation}\label{eq_cumulanttomoment}
 m_{2k}=\bigl[z^{0}\bigr]\left(\partial+2\gamma d+\left((q-1)\gamma-\frac{1}{2}\right)d'+*_{g}\right)^{2k-1}g(z),
 \end{equation}
where $\bigl[z^{0}\bigr]$ takes the constant term of the formal power series in $\R[[z]]$, and
\[g(z)=\sum_{l=1}^{\infty}\kappa_{l}z^{l-1}.\]
\end{Definition}
\begin{Remark}
 Note by a simple induction on $k=1,2,\dots $ that~\eqref{eq_cumulanttomoment} implies each $m_{2k}$ is given~by a positive constant times $\kappa_{2k}+ $ a polynomial of $\kappa_{2},\kappa_{4},\dots,\kappa_{2k-2}$.
 Hence, $\mathrm{T_{\kappa\rightarrow m}^{q,\gamma}}$ is an invertible map, such that given a sequence of real numbers $\{m_{2k}\}_{k=1}^{\infty}$, there exists a unique real sequence~$\{\kappa_{l}\}_{l=1}^{\infty}$ with $\kappa_{l}=0$ for all odd $l$, and $\mathrm{T}_{\kappa\rightarrow m}^{q,\gamma}\left(\{\kappa_{l}\}_{l=1}^{\infty}\right)=\{m_{2k}\}_{k=1}^{\infty}$. More precisely, $\{m_{2j}\}_{j=1}^{l}$ are corresponding to $\{\kappa_{j}\}_{j=1}^{2l}$. We denote the inverse map by
$\{\kappa_{l}\}=\mathrm{T_{m\rightarrow \kappa}^{q,\gamma}}(\{m_{2k}\})$.
 In Section~\ref{sec:momentcumulant}, we provide various points of views on the maps $\mathrm{T}_{\kappa\rightarrow m}^{q,\gamma}$ and $\mathrm{T}_{m\rightarrow \kappa}^{q,\gamma}$.
\end{Remark}
We are ready to present the main result now.
\begin{thm}[convergence of empirical measure in high temperature]\label{thm:hightemperaturemainthm}
 The sequence of random $N$-tuples $\{\Vec{c}_{N}\}_{N=1}^{\infty}$ satisfies LLN, if and only if it is $q$-$\gamma$-LLN-appropriate.

 If this occurs, we have
 $
 \{m_{2k}\}_{k=1}^{\infty}=\mathrm{T}_{\kappa\rightarrow m}^{q,\gamma}(\{\kappa_{l}\}_{l=1}^{\infty})$,
 where $\{\kappa_{l}\}_{l=1}^{\infty}$ are the $q$-$\gamma$ cumulants corresponding to $\{m_{2k}\}_{k=1}^{\infty}$.
\end{thm}
\subsection{Asymptotic expression under Dunkl actions}
The proof of Theorem~\ref{thm:hightemperaturemainthm} is relying on the actions of Dunkl operator introduced in Section~\ref{sec:dunkl} on Bessel generating functions. Before proceeding to the proof, we first study the explicit expression of this action in detail.

Consider a symmetric function $F(x_{1},\dots,x_{N})$ which is analytic on a complex domain near 0. Then the Taylor expansion of $F$ of $2k^{\rm th}$ order is
\begin{equation}\label{eq_Taylorpolynomial}
 F(x_{1},\dots,x_{N})=\sum_{\lambda\colon |\lambda|\le 2k,\, l(\lambda)\le N}c^{\lambda}_{F}\cdot m_{\lambda}(\vec{x})+O\bigl(||x||^{2k+1}\bigr),
\end{equation}
where $m_{\lambda}(\vec{x})$ is the monomial symmetric polynomial indexed by $\lambda$. If we further assume $F$ to be a symmetric function in $x_{1}^{2},\dots,x_{N}^{2}$, then
\begin{equation}\label{eq_doublecolumn}
 c^{\lambda}_{F} \text{ is\ nonzero\ only\ if}\ \lambda\ \text{is\ even.}
\end{equation}

Fix $N\ge 1$. Recall we denote
$P_{k}=D_{1}^{k}+\dots+D_{N}^{k}$,
where $D_{i}$ is defined in Section~\ref{sec:dunkl}.

The following theorem is a technical result on the explicit expansion of $\exp(F(x_{1},\dots,x_{N}))$ under the action of $P_{k}$, and it serves as a stepping stone to the proof of Theorem~\ref{thm:hightemperaturemainthm}.
\begin{thm}\label{thm:technicalequivalence}
 Fix $k=1,2,\dots $ and an even partition $\lambda$ and $|\lambda|=2k$. Let $F(x_{1},\dots,x_{N})$ be a symmetric function on $\R^{N}$ satisfying equation~\eqref{eq_doublecolumn}, analytic on a domain near $(0,\dots,0)$ and~${F(0,\dots,0)=0}$. Then
 \begin{gather}
 N^{-l(\lambda)}\left[\prod_{i=1}^{l(\lambda)}P_{\lambda_{i}}\right]\exp(F(x_{1},\dots,x_{N}))\bigr|_{x_{1}=\dots =x_{N}=0}\nonumber\\
 \qquad=b^{\lambda}_{\lambda}\cdot c^{\lambda}_{F}+\sum_{\mu\colon |\mu|=2k,\,l(\mu)>l(\lambda)}b^{\lambda}_{\mu}\cdot c^{\mu}_{F}+L\bigl(c_{F}^{(i)},1\le i\le 2k-1\bigr)\nonumber\\
 \qquad\phantom{=}{}+R_{1}\bigl(c^{v}_{F},|v|<2k\bigr)+N^{-1}R_{2}\bigl(c^{v}_{F},|v|\le 2k\bigr),\label{eq_dunklaction}
 \end{gather}
 where $b^{\lambda}_{\mu}$ are coefficients that are uniformly bounded in the limit regime $N\rightarrow\infty$, $M\rightarrow\infty$, $\theta\rightarrow0$, $N\theta\rightarrow\gamma$, $M\theta\rightarrow q\gamma$, and the notation $(i)$ denotes the partition $(i,0,\dots,0)$.
 In particular,
 \begin{align*}
 \lim_{N\theta\rightarrow \gamma,M\theta\rightarrow q\gamma}b^{\lambda}_{\lambda}={}&\prod_{i=1}^{l(\lambda)}[\lambda_{i}(\lambda_{i}-2+2q\gamma)(\lambda_{i}-2+2\gamma)(\lambda_{i}-4-2q\gamma)\\
 &\times(\lambda_{i}-4+2\gamma)\cdots (2+2q\gamma)
(2+2\gamma)2q\gamma ],
 \end{align*}
 and
 \begin{align}
 L\big(c^{(i)}_{F},1\le i\le 2k-1\big)={}&\prod_{i=1}^{l(\lambda)}\left(\bigl[z^{0}\bigr]\left(\partial+2\gamma d+\left((q-1)\gamma-\frac{1}{2}\right)d'+*_{g}\right)^{\lambda_{i}-1}g(z)\right)\nonumber\\
 &-2k(2k-2+2q\gamma)(2k-2+2\gamma)(2k-4-2q\gamma)\nonumber\\
 &\phantom{-}{}\times(2k-4+2\gamma)\cdots (2+2q\gamma)(2+2\gamma)2q\gamma\cdot c^{(2k)}_{F}\cdot \mathbf{1}_{l(\lambda)=1}.\label{eq_L}
 \end{align}
 The operator $\partial$, $d$, $d'$ and $*_{g}$ are defined in the same way as in Definition~{\rm\ref{def:operators}}, and
 \[g(z):=\sum_{n=1}^{\infty}nc^{(n)}_{F}z^{n-1}.\]
 Here, $L$, $R_{1}$ and $R_{2}$ are all polynomials of the coefficients $c^{v}_{F}$, whose corresponding variables are
 given in the parenthesis, and the coefficient of each monomial is uniformly bounded in the limit regime. Moreover, each monomial in $R_{1}$ contains at least one $c^{v}_{F}$, where $l(v)\ge 2$.
 If we assign~$c^{v}_{F}$ with degree~$|v|$, each summand on the right of~\eqref{eq_dunklaction} is homogeneous of degree~$2k$.
\end{thm}
We postpone the proof of Theorem~\ref{thm:technicalequivalence} to next section, and using its result, we are able to prove Theorem~\ref{thm:hightemperaturemainthm}.

\begin{proof}[Proof of Theorem~\ref{thm:hightemperaturemainthm}]
We first assume the sequence $\{\Vec{c}_{N}\}_{N=1}^{\infty}$ is $q$-$\gamma$-appropriate, with limiting $q$-$\gamma$-cumulants $\{\kappa_{l}\}_{l=1}^{\infty}$. We need to show $\{\Vec{c}_{N}\}_{N=1}^{\infty}$ is satisfying LLN with moments $\{m_{2k}\}_{k=1}^{\infty}=\mathrm{T}_{\kappa\rightarrow m}^{q,\gamma}(\{\kappa_{l}\}_{l=1}^{\infty})$.

Denote the type BC Bessel generating function of $\Vec{c}_{N}$ by $G_{N,M;\theta}(x_{1},\dots,x_{N})$. By Theorem~\ref{thm:dunklonbgf}, the left side of~\eqref{eq_rectangularlln} before taking the limit is given by
\begin{equation}\label{eq_symmetricdunklonbessel}
 N^{-s}\Bigg(\prod_{i=1}^{s}P_{2k_{i}}\Bigg)G_{N,M;\theta}(x_{1},\dots,x_{N})\bigr|_{x_{1}=\dots =x_{N}=0}.
\end{equation}
For each $N=1,2,\dots $, without loss of generality, assume $k_{1}\ge k_{2}\ge\dots \ge k_{s}$ and identify $(2k_{1},\dots,2k_{s})$ with a partition $\lambda$. Also since $G_{N,M;\theta}$ is analytic on a domain near 0 and~${G_{N,M;\theta}(0,\dots,0)=1}$, there is a function $F_{N,M;\theta}(x_{1},\dots,x_{N})$ analytic near 0 and
\[
\exp(F_{N,M;\theta}(x_{1},\dots,x_{N}))=G_{N,M;\theta}(x_{1},\dots,x_{N}), \qquad F_{N,M;\theta}(0,\dots,0)=0.
\]
 We write $F_{N,M;\theta}$ in terms of its $2k^{\rm th}$ order Taylor polynomial
\[F_{N,M;\theta}(x_{1},\dots,x_{N})=\sum_{\mu:|\mu|\le 2k, l(\mu)\le N}c^{\lambda}_{F_{N,M;\theta}}\cdot m_{\mu}(\vec{x})+O\bigl(\big|\big|x^{2k+1}\big|\big|\bigr).\]

After the above identifications~\eqref{eq_symmetricdunklonbessel} satisfies the condition of Theorem~\ref{thm:technicalequivalence}. Then we turn it into the expression on the right of~\eqref{eq_dunklaction}, and take the limit $N\theta\rightarrow\gamma$, $M\theta\rightarrow q\gamma$. By $q$-$\gamma$-appropriateness,
\begin{equation*}
 \lim_{N\theta\rightarrow\gamma, M\theta\rightarrow q\gamma}c^{(n)}_{F_{N,M;\theta}}=\frac{\kappa_{n}}{n},\qquad \lim_{N\theta\rightarrow\gamma, M\theta\rightarrow q\gamma}c^{v}_{F_{N,M;\theta}}=0\qquad \text{if} \ l(v)>1.
\end{equation*}
Hence \smash{$\sum_{v\colon|v|=2k,\, l(v)>l(\lambda)}b^{\lambda}_{v}\cdot c^{v}_{F_{N,M;\theta}}$} turns to 0, since each summand contains some term converging to 0, and
\[
b^{\lambda}_{\lambda}\cdot c^{\lambda}_{F_{N,M;\theta}}\longrightarrow \begin{cases}
 0&
 \text{if}\ s>1,\\
 (2k_{1}-2+2q\gamma)(2k_{1}-2+2\gamma)(2k_{1}-4-2q\gamma)\\
 \qquad\times(2k_{1}-4+2\gamma)\cdots (2+2q\gamma)(2+2\gamma)2q\gamma\cdot c^{(2k_{1})}_{F}\kappa_{2k_{1}}&
 \text{if}\ s=1.
\end{cases}
\]

The polynomial $L$ converges to
\begin{gather*}
\prod_{i=1}^{s}\left(\bigl[z^{0}\bigr]\left(\partial+2\gamma d+\left((q-1)\gamma-\frac{1}{2}\right)d'+*_{g}\right)^{2k_{i}-1}g(z)\right)-(2k_{1}-2+2q\gamma)\\
\qquad\times(2k_{1}-2+2\gamma)(2k_{1}-4-2q\gamma)(2k_{1}-4+2\gamma)
 \cdots (2+2q\gamma)(2+2\gamma)2q\gamma\cdot \kappa_{2k_{1}}\cdot \mathbf{1}_{s=1},
\end{gather*}
where $g(z)=\sum_{n=1}\kappa_{n}z^{n-1}$, since \smash{$\sum_{n=1}^{\infty}nc^{(n)}_{F_{N,M;\theta}}z^{n-1}$} converges coefficient-wise to $g(z)$.

The polynomial $R_{1}$ converges to 0, also because each summand contains some factor \smash{$c^{v}_{F_{N,M;\theta}}$} with $l(v)>1$, that vanishes in the limit regime. The polynomial $N^{-1}R_{2}$ vanishes as well in the limit since all its coefficients converge to 0. Combining all the results above gives
\[
\lim_{N\rightarrow \infty}\E{\prod_{i=1}^{s}p_{2k_{i}}^{N}}=\prod_{i=1}^{s}\left(\bigl[z^{0}\bigr]\left(\partial+2\gamma d+\left((q-1)\gamma-\frac{1}{2}\right)d'+*_{g}\right)^{2k_{i}-1}g(z)\right),\]
which is equal to $\prod_{i=1}^{s}m_{2k_{i}}$, where $\{m_{2k}\}_{k=1}^{\infty}=\mathrm{T}_{\kappa\rightarrow m}^{q,\gamma}(\{\kappa_{l}\}_{l=1}^{\infty})$. Hence the LLN condition of~${\{\Vec{c}_{N}\}_{N=1}^{\infty}}$ is proved.

Now we go in the opposite direction, that assuming $\{\Vec{c}_{N}\}_{N=1}^{\infty}$ satisfies LLN for some $\{m_{k}\}_{k=1}^{\infty}$, i.e., for all even partition $\lambda$,
\[
N^{-l(\lambda)}\left[\prod_{i=1}^{l(\lambda)}P_{\lambda_{i}}\right]\exp(F_{N,M;\theta}(x_1,\dots,x_{N}))|_{x_{1}=\dots =x_{N}=0}\xrightarrow[M\theta\rightarrow q \gamma]{N\theta \rightarrow \gamma}\prod_{i=1}^{l(\lambda)}m_{\lambda_{i}},
\]
where \smash{$F_{N,M;\theta}(x_{1},\dots,x_{N})=\sum_{\mu\colon l(\mu)\le N}c^{\mu}_{F_{N,M;\theta}}\cdot m_{\mu}(\vec{x})$} is an analytic function near 0 satisfying~${\exp(F_{N,M;\theta})=G_{N,M;\theta}}$. We need to show
\[
 c^{\lambda}_{F_{N,M;\theta}}\longrightarrow \begin{cases}
 0,& l(\lambda)>1\ \text{or $\lambda$ is not even},\\
 \frac{\kappa_{2k}}{2k},& \lambda=(2k)
 \end{cases}
\]
in the limit regime $N\theta\rightarrow\gamma, M\theta\rightarrow q\gamma$, where $\{\kappa_{l}\}_{l=1}^{\infty}=\mathrm{T_{m\rightarrow \kappa}^{q,\gamma}}(\{m_{2k}\}_{k=1}^{\infty})$.

Note that we only need to consider the case $|\lambda|$ is even, and \smash{$c^{v}_{F_{N,M;\theta}}=0$} for all $v$ not even, since~the type BC Bessel function of each $\Vec{c}_{N}$ is a symmetric function in $x_{1}^{2},\dots,x_{N}^{2}$, by Definition~\ref{def:bgf} and Proposition~\ref{prop:polyexpectation}. We proceed by induction on $|\lambda|$.

For $|\lambda|=0$, there is nothing to show. Suppose the result holds for all $|\lambda|\le 2k-2$, we now consider the partition that $|\lambda|=2k$. By Theorem~\ref{thm:technicalequivalence}, for each $\theta$, $N$, $M$, we have a (finite) system of linear equations of the unknowns \smash{$\{c^{\mu}_{F_{N,M;\theta}}\}_{\mu\colon |\mu|=2k, l(\mu)>l(\lambda), \mu \text{\ is\ even}}$}.
\begin{gather}
 b^{\lambda}_{\lambda}\cdot c^{\lambda}_{F_{N,M;\theta}}+\sum_{\mu\colon |\mu|=2k,\, l(\mu)>l(\lambda),\, \mu \text{\ is\ even}}b^{\lambda}_{\mu}\cdot c^{\mu}_{F_{N,M;\theta}}\nonumber\\
 \qquad= N^{-l(\lambda)}\left[\prod_{i=1}^{l(\lambda)}P_{\lambda_{i}}\right]\exp(F_{N,M;\theta}(x_{1},\dots,x_{N}))\bigr|_{x_{1}=\dots =x_{N}=0}
 -L\bigl(c_{F_{N,M;\theta}}^{(i)},1\le i\le 2k-1\bigr)\nonumber\\
 \phantom{\qquad=}{}-R_{1}(c^{v}_{F_{N,M;\theta}},|v|<2k)-N^{-1}R_{2}(c^{v}_{F_{N,M;\theta}},|v|\le 2k).\label{eq_linearsystem}
\end{gather}
We observe that if we write it in the matrix form in the lexicographical order of $\mu$ introduced in Section~\ref{sec:sympoly}, the above system is upper triangular, and again by Theorem~\ref{thm:technicalequivalence}, its diagonal entries~$b^{\mu}_{\mu}$ all converge to some nonzero constant in the limit regime, and the off-diagonal entries are uniformly bounded. Hence the matrix is invertible asymptotically, and its inverse has uniformly bounded entries.

\begin{Claim} If $\lambda\ne (2k)$, the right side of~\eqref{eq_linearsystem} converges to 0 in the limit regime.
\end{Claim}

\begin{proof} $R_{1}\rightarrow 0$ by induction hypothesis (recall that each of its term involves some partition $v$ with $l(v)\ge 2$), and $N^{-1}R_{2}\rightarrow 0$ since the coefficients all vanish in the limit.

By the LLN condition,
\[N^{-l(\lambda)}\left[\prod_{i=1}^{l(\lambda)}P_{\lambda_{i}}\right]\exp(F(x_{1},\dots,x_{N}))\bigr|_{x_{1}=\dots =x_{N}=0}\longrightarrow \prod_{i=1}^{l(\lambda)}m_{\lambda_{i}},\]
and by Theorem~\ref{thm:technicalequivalence} and Definition~\ref{def:operators}, when $\lambda\ne (2k)$, each $\lambda_{i}<2k$ and
\[
L\bigl(c_{F_{N,M;\theta}}^{(i)},1\le i\le 2k-1\bigr)=\prod_{i=1}^{l(\lambda)}m^{N,M;\theta}_{\lambda_{i}},
\]
where
\[
\bigl\{m^{N,M;\theta}_{k}\bigr\}_{k=1}^{\infty}=\mathrm{T}_{\kappa\rightarrow m}^{q,\gamma}\bigl(\bigl\{l\cdot c^{(l)}_{F_{N,M;\theta}}\bigr\}_{l=1}^{\infty}\bigr).
\]
By induction hypothesis \smash{$\bigl\{l\cdot c^{(l)}_{F_{N,M;\theta}}\bigr\}_{l=1}^{\infty}\longrightarrow\{\kappa_{l}\}_{l=1}^{\infty}$} pointwisely for $l<2k$, and hence $\smash{m^{N,M;\theta}_{j}}\rightarrow\allowbreak m_{j}$ pointwisely for $j<k$, and
\[L\bigl(c_{F_{N,M;\theta}}^{(i)},1\le i\le 2k-1\bigr)\longrightarrow \prod_{i=1}^{l(\lambda)}m_{\lambda_{i}}\ \]
as well. \end{proof}

Because of this claim, we conclude that when $N\theta\rightarrow\gamma$, $M\theta\rightarrow q\gamma$, the solutions of the linear system where $|\lambda|=2k$, $\lambda\ne (2k)$ converge to the zero vector, in particular,
\begin{equation}\label{eq_llnapproxexceptonerowYD}
 c^{\lambda}_{F_{N,M;\theta}}\longrightarrow 0\qquad
\text{for\ all}\quad |\lambda|=2k,\qquad \lambda\ne (2k).
\end{equation}

It remains to consider $\lambda=(2k)$. This time we write down a single identity
\begin{gather*}
 b^{(2k)}_{(2k)}\cdot c^{(2k)}_{F_{N,M;\theta}}+\sum_{v\colon |v|=2k,\, l(v)>1,\, \mu\ \text{is\ even}}b^{(2k)}_{v}\cdot c^{v}_{F_{N,M;\theta}}\\
 \qquad= N^{-1}P_{2k}\left[\exp(F_{N,M;\theta}(x_{1},\dots,x_{N})\right]\bigr|_{x_{1}=\dots =x_{N}=0}
 -L\bigl(c_{F_{N,M;\theta}}^{(i)},1\le i\le 2k-1\bigr)\\
 \phantom{\qquad=}{}-R_{1}(c^{v}_{F_{N,M;\theta}},|v|<2k)-N^{-1}R_{2}(c^{v}_{F_{N,M;\theta}},|v|\le 2k).
\end{gather*}
We have that
\[
\mathrm{LHS}=b^{(2k)}_{(2k)}\cdot c^{(2k)}_{F_{N,M;\theta}}+o(1),
\]
where \smash{$g_{N,M;\theta}(z)=\sum_{n=1}^{\infty}nc^{(n)}_{F_{N,M;\theta}}z^{n-1}$},
 because of Theorem~\ref{thm:technicalequivalence} and~\eqref{eq_llnapproxexceptonerowYD}.
 And
\begin{align*}
\mathrm{RHS}={}&m_{2k}-\bigl[z^{0}\bigr]\left(\partial+2\gamma d+\left((q-1)\gamma-\frac{1}{2}\right) d'+*_{g_{N,M;\theta}}\right)^{2k-1}g_{N,M;\theta}(z)\\
&+b^{(2k)}_{(2k)}\cdot c^{(2k)}_{F_{N,M;\theta}}+o(1)
\end{align*}
because of Theorem~\ref{thm:technicalequivalence} and the
 LLN assumption. Hence, when $N\theta\rightarrow\gamma$, $M\theta\rightarrow q\gamma$,
 \[
 \bigl[z^{0}\bigr]\Big(\partial+2\gamma d+\left((q-1)\gamma-\frac{1}{2}\right) d'+*_{g_{N,M;\theta}}\Big)^{2k-1}g_{N,M;\theta}(z)\longrightarrow m_{2k}.
 \]
 By Definition~\ref{def:operators}, the invertibility of $\mathrm{T}_{\kappa\rightarrow m}^{q,\gamma}$ and the induction hypothesis, this is equivalent to
 \[
 (2k)\cdot c^{(2k)}_{F_{N,M;\theta}}\longrightarrow \kappa_{2k}
 \]
 in the limit regime, that $\kappa_{2k}$ is in the image of $\mathrm{T}_{m\rightarrow \kappa}^{q,\gamma}(\{m_{2j}\}_{j=1}^{\infty})$. This finishes the induction step and therefore the proof.
 \end{proof}

\subsection{Proof of Theorem~\ref{thm:technicalequivalence}}
We start by reducing $F(x_{1},\dots,x_{N})$ from a (locally) analytic function to its $2k^{\rm th}$ Taylor polynomial.
\begin{Lemma}\label{lem:reducetopoly}
 For $F(x_{1},\dots,x_{N})$ of the form~\eqref{eq_Taylorpolynomial}, denote
 \[
 F{'}(x_{1},\dots,x_{N})=\sum_{\lambda\colon|\lambda|\le 2k,\, l(\lambda)\le N}c^{\lambda}_{F}\cdot m_{\lambda}(\vec{x}),
 \]
 the truncation of $F(x_{1},\dots,x_{N})$. Then for a partition $\lambda$ with $|\lambda|=2k$, we have
 \[
 \Bigg[\prod_{i=1}^{l(\lambda)}P_{\lambda_{i}}\Bigg]\exp(F'(x_{1},\dots,x_{N}))\bigr|_{x_{1}=\dots =x_{N}=0}=\Bigg[\prod_{i=1}^{l(\lambda)}P_{\lambda_{i}}\Bigg]\exp(F(x_{1},\dots,x_{N}))\bigr|_{x_{1}=\dots =x_{N}=0}.
 \]
\end{Lemma}
\begin{proof}
 Since $F$ is analytic near 0, write $\exp(F(x_{1},\dots,x_{N}))$ and $\exp(F'(x_{1},\dots,x_{N}))$ as symmetric power series. Their difference $R(\vec{x})$ is a power series of order $O\bigl(||x||^{2k+1}\bigr)$. Since \smash{$\prod_{i=1}^{l(\lambda)}P_{\lambda_{i}}$} is a homogeneous polynomial of $D_{i}$ and each $D_{i}$ reduces the total power of a monomial by 1,
 \[
 \left[\prod_{i=1}^{l(\lambda)}P_{\lambda_{i}}\right]R(\vec{x})\bigr|_{x_{1}=\dots =x_{N}=0}=0.\tag*{\qed}
 \]\renewcommand{\qed}{}
\end{proof}

By Lemma~\ref{lem:reducetopoly}, in the remainder of this section we take
\[F(x_{1},\dots,x_{N})=\sum_{\lambda\colon |\lambda|\le 2k,\, l(\lambda)\le N,\, \lambda\ \text{even}}c^{\lambda}_{F}\cdot m_{\lambda}(\vec{x}).\]
\smash{$\prod_{i=1}^{l(\lambda)}P_{\lambda_{i}}$} is a sum of products of $D_{i}$, $i=1,2,\dots,N$. For each product of the form $D_{1}^{m_{1}}\cdots D_{N}^{m_{N}}$ acting on $\exp(F(x_{1},\dots,x_{N}))$, by~\eqref{prop:commutativity} the order does not matter.
Recall from Definition~\ref{def:dunkl} that
\[D_{i}=\partial_{i}+\left[\theta(M-N+1)-\frac{1}{2}\right]\frac{1-\sigma_{i}}{x_{i}}+\theta\sum_{j\ne i}\left[\frac{1-\sigma_{ij}}{x_{i}-x_{j}}+\frac{1-\tau_{ij}}{x_{i}+x_{j}}\right].\]
Observe that $D_{i}[\exp(F(x_{1},\dots,x_{N}))]$ is of the form $H(x_{1},\dots,x_{N})\exp(F(x_{1},\dots,x_{N}))$, where $H(x_{1},\dots,x_{N})$ is a polynomial of $x_{1},\dots,x_{N}$, and for any i, $D_{i}[H(x_{1},\dots,x_{N})\exp(F(x_{1},\dots,x_{N}))]$ is still $\exp(F(x_{1},\dots,x_{N}))$ multiplied by a polynomial. More precisely,
\begin{gather}
 \partial_{i}[H(x_{1},\dots,x_{N})\exp(F(x_{1},\dots,x_{N}))]\nonumber\\
 \qquad=(\partial_{i}H(x_{1},\dots,x_{N})+H(x_{1},\dots,x_{N})\partial_{i} F(x_{1},\dots,x_{N}))\cdot \exp(F(x_{1},\dots,x_{N})),\label{eq_hz1}
\\
 \frac{1-\sigma_{i}}{x_{i}}[H(x_{1},\dots,x_{N})\exp(F(x_{1},\dots,x_{N}))]\nonumber\\
 \qquad=\left(\frac{1-\sigma_{i}}{x_{i}}H(x_{1},\dots,x_{N})\right)\cdot \exp(F(x_{1},\dots,x_{N})),\nonumber
\\
 \frac{1-\sigma_{ij}}{x_{i}-x_{j}}[H(x_{1},\dots,x_{N})\exp(F(x_{1},\dots,x_{N}))]\nonumber\\
 \qquad=\left(\frac{1-\sigma_{ij}}{x_{i}-x_{j}}H(x_{1},\dots,x_{N})\right)\cdot \exp(F(x_{1},\dots,x_{N})),\nonumber
\\
 \frac{1-\tau_{ij}}{x_{i}+x_{j}}[H(x_{1},\dots,x_{N})\exp(F(x_{1},\dots,x_{N}))]\nonumber\\
 \qquad=\left(\frac{1-\tau_{ij}}{x_{i}+x_{j}}H(x_{1},\dots,x_{N})\right)\cdot \exp(F(x_{1},\dots,x_{N})).\label{eq_hz4}
\end{gather}
 We see that
 \[
 \prod_{i=1}^{l(\lambda)}\big[D_{1}^{m_{1}}\cdots D_{N}^{m_{N}}\big]\exp(F(x_{1},\dots,x_{N}))|_{x_{1}=\dots =x_{N}=0}
 \]
 is obtained by acting a polynomial of $\partial_{i}$, $\frac{1-\sigma_{i}}{x_{i}}$, $\frac{1-\sigma_{ij}}{x_{i}-x_{j}}$, $\frac{1-\tau_{ij}}{x_{i}+x_{j}}$ on $F(x_{1},\dots,x_{N})$, then take the constant term. Then we have the following basic observation.
\begin{prop}\label{prop:unifombounded}
 For any $N$-tuple of nonnegative integers $m_{1},\dots,m_{N}$,
 \[
 \bigl[D_{1}^{m_{1}}\cdots D_{N}^{m_{N}}\bigr]\exp(F(x_{1},\dots,x_{N}))\big|_{x_{1}=\dots =x_{N}=0}
 \]
 is a homogeneous polynomial in $c^{v}_{F}$ of degree $\sum_{i=1}^{N}m_{i}$, if we take $c^{v}_{F}$ to be of degree $|v|$. Moreover, the coefficients of this polynomial are all uniformly bounded in the limit regime $N\theta\rightarrow\gamma$, ${M\theta\rightarrow q\gamma}$.
\end{prop}
\begin{proof}
 Each of $\partial_{i}$, \smash{$\frac{1-\sigma_{i}}{x_{i}}$}, \smash{$\frac{1-\sigma_{ij}}{x_{i}-x_{j}}$}, \smash{$\frac{1-\tau_{ij}}{x_{i}+x_{j}}$} reduces the degree of a monomial by~1, the constant term of
\smash{$\bigl[D_{1}^{m_{1}}\cdots D_{N}^{m_{N}}\bigr]\exp(F(x_{1},\dots,x_{N}))$} is then obtained from some monomials of $x_{1},\dots,x_{N}$ of degree \smash{$\sum_{i=1}^{N}m_{i}$}. Since $c^{v}_{F}$ is the coefficient of $m_{v}(\vec{x})$ which is of degree $|v|$, by assigning $c^{v}_{F}$ with degree $|v|$ one can pass the degree of the original monomials to their resulting constant terms.

Each $D_i$ is a sum of $2N$ single operators $\partial_{i}$, \smash{$\frac{1-\sigma_{i}}{x_{i}}$}, \smash{$\frac{1-\sigma_{ij}}{x_{i}-x_{j}}$} and \smash{$\frac{1-\tau_{ij}}{x_{i}+x_{j}}$}, in which $2N-2$ terms~involves a factor $\theta$, and hence \smash{$D_{1}^{m_{1}}\cdots D_{N}^{m_{N}}$} is a sum of \smash{$(2N)^{\sum_{i=1}^{N}m_{i}}$} products of single operators. The constant term of each of these products acting on $\exp(F(x_{1},\dots,x_{N}))$ is changing with $\theta$, $N$, $M$ as a muliple of \smash{$\theta^{\sum_{i=1}^{N}m_{i}-\#\partial_{i}\text{\ in\ the\ product}}$}, and the number of such products is of order~\smash{${\rm O}(N^{\sum_{i=1}^{N}m_{i}-\#\partial_{i}\text{\ in\ the\ product}})$}. Hence as $N\theta\rightarrow\gamma$ the coefficient is uniformly bounded.
\end{proof}

\begin{prop}\label{prop:dunklptodunkld}
 For any partition $\lambda$, we have
 \begin{gather*}
 N^{-l(\lambda)}\left[\prod_{i=1}^{l(\lambda)}P_{\lambda_{i}}\right]\exp(F(x_{1},\dots,x_{N}))\bigr|_{x_{1}=\dots =x_{N}=0}\\
 \qquad=\left[\prod_{i=1}^{l(\lambda)}(D_{i})^{\lambda_{i}}\right]\exp(F(x_{1},\dots,x_{N}))\bigr|_{x_{1}=\dots =x_{N}=0}+O\left(\frac{1}{N}\right),
 \end{gather*}
 in the limit regime $N\theta\rightarrow\gamma$, $ M\theta\rightarrow q\gamma$, where $O\bigl(\frac{1}{N}\bigr)$ is a homogeneous polynomial of $c^{v}_{F}$ {\rm(}taking~$c^{v}_{F}$ to be of degree $|v|)$ whose coefficients are of order $O\bigl(\frac{1}{N}\bigr)$.
\end{prop}
\begin{proof}
 Each $P_{\lambda_{i}}$ is a sum of $N$ terms $D_{j}^{\lambda_{i}}\ (j=1,2,\dots,N)$, hence $\prod_{i=1}^{l(\lambda)}P_{\lambda_{i}}$ is a sum of $N^{l(\lambda)}$ such terms, in which \smash{$O\bigl(N^{l(\lambda)-1}\bigr)$} terms have not all distinct indices $j$. By Proposition~\ref{prop:unifombounded}, each of these terms has uniformly bounded coefficient, hence they together contribute $O\bigl(\frac{1}{N}\bigr)$. As for the remaining terms with all distinct indices, by symmetry of $F(x_{1},\dots,x_{N})$, their action are all the same as \smash{$\prod_{i=1}^{l(\lambda)}D_{i}^{\lambda_{i}}$}.
\end{proof}

After all the reductions above, it remains to study \[\left[\prod_{i=1}^{l(\lambda)}(D_{i})^{\lambda_{i}}\right]\exp(F(x_{1},\dots,x_{N}))\bigr|_{x_{1}=\dots =x_{N}=0}\]for an arbitrary partition $\lambda$, whose expression should match the right side of~\eqref{eq_dunklaction}.
The expression on the right side of~\eqref{eq_dunklaction} can be split into three parts: the linear polynomials of $c^{v}_{F}$, the terms involving only $c^{v}_{F}$ where $v$ are length 1 partitions, and all the other remaining terms. In the next two propositions, we deal with the first two cases separately. Before that we present several lemmas that will be used in the proof. Consider the action of \smash{$\prod_{i=1}^{l(\lambda)}D_{i}^{\lambda_{i}}$} on $m_{\mu}(\vec{x})$. Each~$D_{i}$ is a~combination of \smash{$\partial_{i}+\bigl[\theta(M-N+1)-\frac{1}{2}\bigr]\frac{1-\sigma_{i}}{x_{i}}$} and \smash{$\theta\bigl[\frac{1-\sigma_{ij}}{x_{i}-x_{j}}+\frac{1-\tau_{ij}}{x_{i}+x_{j}}\bigr]$} with $N-1$ choices of~${j\ne i}$, hence \smash{$\prod_{i=1}^{l(\lambda)}D_{i}^{\lambda_{i}}$} will lead to a sum, whose summand are products of these two terms.

For a given $i\in \{1,2,\dots,N\}$, we say that the component \smash{$\theta\bigl[\frac{1-\sigma_{ij}}{x_{i}-x_{j}}+ \frac{1-\tau_{ij}}{x_{i}+x_{j}}\bigr]$} in $D_{i}$ has index~$j$ ($j\ne i$).

 \begin{Lemma}\label{lem:generic}
 For arbitrary partitions $\lambda$ and $\mu$, the constant term of \smash{$\prod_{i=1}^{l(\lambda)}D_{i}^{\lambda_{i}}m_{\mu}(\vec{x})$} has~a~ge\-neric part, which is a sum of product of components in $D_{i}$ of \smash{$\prod_{i=1}^{l(\lambda)}D_{i}^{\lambda_{i}}$}, such that all the indices $j$ in the product are distinct, and all bigger than $l(\lambda)$. The remaining part is of order~$O\bigl(\frac{1}{N}\bigr)$ in the limit regime $N\theta\rightarrow\gamma$, $M\theta\rightarrow q\gamma$.
 \end{Lemma}
 \begin{proof}
 For $k=0,1,\dots,|\lambda|$, the number of summands in the remaining part (which means their exists a pair of indices that coincides) with $k$ components of \smash{$\theta\bigl[\frac{1-\sigma_{ij}}{x_{i}-x_{j}}+\frac{1-\tau_{ij}}{x_{i}+x_{j}}\bigr]$} is of order~${\rm O}(N^{k-1})$, and the power of $\theta$ in these summands is $k$. Since $N\theta\rightarrow\gamma>0$, the remaining part is a finite sum of order $O\bigl(N^{k-1}\theta^{k}\bigr)$, which is $O\bigl(\frac{1}{N}\bigr)$.
 \end{proof}

 Because of this, we only consider the limit of the generic part of the expression. For simplicity, we write $l=l(\lambda)$.

 \begin{Lemma}\label{lem:generic2}
 The generic part of constant term of \smash{$\prod_{i=1}^{l}D_{i}^{\lambda_{i}}m_{\mu}(\vec{x})$} is given by
 \begin{equation}\label{eq_symmetricdunklaction}
 \big[D_{l}^{\lambda_{l}-1}\partial_{l}\big]\cdots\big[D_{2}^{\lambda_{2}-1}\partial_{2}\big]\cdot\big[D_{1}^{\lambda_{1}-1}\partial_{1}\big]m_{\mu}(\vec{x}).
 \end{equation}
 \end{Lemma}
 \begin{proof}
 Let $m=1,2,\dots,l$ denote the index of the operators $D_{m}$. For $m=1$, because of the symmetry, $m_{\mu}(\vec{x})$ is invariant under the action of $\sigma_{i}$, $\sigma_{ij}$ and $\tau_{ij}$, and hence $D_{i}m_{\mu}(\vec{x})$ is equal to~$\partial_{i}m_{\mu}(\vec{x})$.

 For $m=2,3,\dots,l$, after acting $\bigl[D_{m-1}^{\lambda_{m-1}-1}\partial_{m-1}\bigr]\cdots\bigl[D_{1}^{\lambda_{1}-1}\partial_{1}\bigr]$ on $m_{\mu}(\vec{x})$, since the generic part has distinct indices $j$ bigger than $l(\lambda)$, we get
 some polynomial $H(x_1,\dots,x_{N})$,
 where the operators act on variables $x_{1},\dots,x_{m-1}$ and $x_{j}$ ($j>l$). So $H(x_1,\dots,x_{N})$ is still symmetric as function of $x_{N}^{2}$ and $x_{j'}^{2}$ for another different $j'$ in the first copy of $D_{m}$, invariant again under the action of $\sigma_{i}$, $\sigma_{ij}$ and $\tau_{ij}$. We conclude that
 \[
 D_{m}\bigl[D_{m-1}^{\lambda_{m-1}-1}\partial_{m-1}\bigr]\cdots\bigl[D_{1}^{\lambda_{1}-1}\partial_{1}\bigr]m_{\mu}(\vec{x})
 =\partial_{m}\bigl[D_{m-1}^{\lambda_{m-1}-1}\partial_{m-1}\bigr]\cdots\bigl[D_{1}^{\lambda_{1}-1}\partial_{1}\bigr]m_{\mu}(\vec{x}).\tag*{\qed}
 \]\renewcommand{\qed}{}
 \end{proof}

 \begin{Remark}
 One can replace $m_{\mu}(\vec{x})$ by $F(x_{1},\dots,x_{N})$ or $\exp(F(x_{1},\dots,x_{N}))$ in last lemma, since these functions satisfy the same symmetry.
 \end{Remark}

 The next lemma considers the concrete action of $D_{i}$ on a polynomial of $x_{1},\dots,x_{l}$.
 \begin{Lemma}\label{lem:dunklonmonomial}
 For an arbitrary $l$-tuple $(n_{1},\dots,n_{l})\in \Z_{\ge 0}^{l}$ and arbitrary $i=1,2,\dots,l$,
 we have that for the generic part of $D_{i}$,
 \begin{align*}
 D_{i}\bigl[x_{1}^{n_{1}}\dots x_{l}^{n_{l}}\bigr]={}&\left(\partial_{i}+\left[\theta(M-N+1)-\frac{1}{2}\right]d_{i}'+2\theta(N-1)d_{i}\right)\bigl[x_{1}^{n_{1}}\cdots x_{l}^{n_{l}}\bigr]\\
 &+\sum_{j\ne i}\bigl(x_{j}p_{1}^{j}+x_{j}p_{2}^{j}\bigr),
 \end{align*}
 where $p_{1}^{j}$, $p_{2}^{j}$ are some polynomials of $x_{1},\dots,x_{l}$ depending on $(n_{1},\dots,n_{l})$, and
 $d_{i}$, $d'_{i}$ are linear operators on polynomials of $x_{1},\dots,x_{N}$ such that
\[
 d_{i}(x_{i}^{n})=
 \begin{cases}
 0, & n=0,\\
 2x_{i}^{n-1},& n>0,
 \end{cases}\qquad
 d'_{i}(x_{i}^{n})=
 \begin{cases}
 0 ,& \text{n is even,}\\
 2x_{i}^{n-1},& \text{n is odd.}
 \end{cases}
 \]
Note that the action depends on whether the power of $x_{i}$ is odd or even.
 \end{Lemma}
 \begin{proof}
 This follows directly from definition. More precisely, for $j>l$,
 \begin{gather*}
 \theta\frac{1-\sigma_{ij}}{x_{i}-x_{j}}\bigl[x_{1}^{n_{1}}\cdots x_{l}^{n_{l}}\bigr]=d_{i}\bigl[x_{1}^{n_{1}}\cdots x_{l}^{n_{l}}\bigr]+x_{j}p_{1}^{j},\\
 \theta\frac{1-\tau_{ij}}{x_{i}+x_{j}}\bigl[x_{1}^{n_{1}}\cdots x_{l}^{n_{l}}\bigr]=d_{i}\bigl[x_{1}^{n_{1}}\cdots x_{l}^{n_{l}}\bigr]+x_{j}p_{2}^{j},
 \end{gather*}
 and
 \[\theta\frac{1-\sigma_{i}}{x_{i}}\bigl[x_{1}^{n_{1}}\cdots x_{l}^{n_{l}}\bigr]=d_{i}'\bigl[x_{1}^{n_{1}}\cdots x_{l}^{n_{l}}\bigr].\tag*{\qed}
 \]\renewcommand{\qed}{}
 \end{proof}

\begin{prop}\label{prop:dunkllinear}
 For any even partition $\lambda$ with $|\lambda|=2k$, we have
 \begin{gather*}
 \left[\prod_{i=1}^{l(\lambda)}(D_{i})^{\lambda_{i}}\right]\exp(F(x_{1},\dots,x_{N}))\bigr|_{x_{1}=\dots =x_{N}=0}\\
 \qquad=b^{\lambda}_{\lambda}\cdot c^{\lambda}_{F}+\sum_{\mu\colon |\mu|=2k, l(\mu)>l(\lambda)}b^{\lambda}_{\mu}\cdot c^{\mu}_{F}+R+O\left(\frac{1}{N}\right).
 \end{gather*}
 In particular,
 \begin{align*}
 \lim_{N\theta\rightarrow\gamma,M\theta\rightarrow q\gamma}b^{\lambda}_{\lambda}={}&\prod_{i=1}^{l(\lambda)}[\lambda_{i}(\lambda_{i}-2+2q\gamma)(\lambda_{i}-2+2\gamma)(\lambda_{i}-4-2q\gamma)\\
 &\times(\lambda_{i}-4+2\gamma)\cdots
 (2+2q\gamma)(2+2\gamma)2q\gamma].
 \end{align*}
 The summand $R$ is a polynomial of $c^{v}_{F}$ that $|v|<2k$. And $O\bigl(\frac{1}{N}\bigr)$ denotes a linear polynomial of~$c^{v}_{F}$ such that $|v|=2k$, and the coefficients are of order $O\bigl(\frac{1}{N}\bigr)$ in the limit regime of this section.
\end{prop}
\begin{proof}
 Since the expression on the right is homogeneous of degree 2k, all the nonlinear terms are collected as $R$, and it suffices to consider the linear terms, which is
 \[\sum_{\mu\colon |\mu|=2k, \, \mu\ \text{is\ even}}b^{\lambda}_{\mu}\cdot c^{\mu}_{F}.\]
 We classify all the even partition $\mu$ with $|\mu|=2k$ in terms of their length. When $l(\mu)>l(\lambda)$, there is nothing to show.
 When $l(\mu)\le l(\lambda)$, we want to show when $\mu\ne \lambda$, $b^{\lambda}_{\mu}$ is of order $O\bigl(\frac{1}{N}\bigr)$.

 Writing $\exp(F(x_{1},\dots,x_{N}))$ as power series of $F(x_{1},\dots,x_{N})$. Since each term in $D_{i}$ reduces the total power of a monomial by 1, and \smash{$F(x_{1},\dots,x_{N})=\sum_{\mu\colon |\mu|\le 2k}c^{\mu}_{F}\cdot m_{\mu}(\vec{x})$} where each~$m_{\mu}(\vec{x})$ is homogeneous of degree $|\mu|$, we see that $b^{\lambda}_{\mu}$ is obtained from the action of \smash{$\prod_{i=1}^{l(\lambda)}D_{i}^{\lambda_{i}}$} on the single symmetric monomial $m_{\mu}(\vec{x})$.

 Again let $l=l(\lambda)$. By Lemmas~\ref{lem:generic} and~\ref{lem:generic2}, we first consider the generic part of $\prod_{i=1}^{l(\lambda)}D_{i}^{\lambda_{i}}\allowbreak\times m_{\mu}(\vec{x})$. When $l(\mu)<l$, each monomial of $m_{\mu}(\vec{x})$ is missing some variable among $x_{1},\dots,x_{l}$, say $x_{m}$. Then when acting the $\partial_{m}$ in~\eqref{eq_symmetricdunklaction}, we get 0 since \smash{$\bigl[D_{m-1}^{\lambda_{m-1}-1}\partial_{m-1}\bigr]\cdots\bigl[D_{1}^{\lambda_{1}-1}\partial_{1}\bigr]$} does not produce any power of $x_{N}$ to $m_{\mu}(\vec{x})$. And the remaining part of the action gives $O\bigl(\frac{1}{N}\bigr)$.

 What remains is to consider the case $l(\mu)=l(\lambda)$, and we calculate the limit of $b^{\lambda}_{\mu}$.
 Again~$b^{\lambda}_{\mu}$ is obtained from the action of \smash{$\prod_{i=1}^{l(\lambda)}D_{i}$} on $m_{\mu}(\vec{x})$. By Lemmas~\ref{lem:generic} and \ref{lem:generic2}, we consider only the generic part of the Dunkl product, and act it on the monomials of $m_{\mu}(\vec{x})$ separately. The monomials with variables other than $x_{1},\dots,x_{l}$ are missing some variables, say $x_{N}\ (1\le m\le l)$. Then~\eqref{eq_symmetricdunklaction} again tells that these monomials only contribute $O\bigl(\frac{1}{N}\bigr)$.

 Now we consider monomials formed by $x_{1},\dots,x_{l}$. For an arbitrary $l$-tuple $(n_{1},\dots,n_{l})$ and arbitrary $i=1,2,\dots,l$, by Lemma~\ref{lem:dunklonmonomial},
 we have
 \begin{align*}
 D_{i}\bigl[x_{1}^{n_{1}}\dots x_{l}^{n_{l}}\bigr]={}&\left(\partial_{i}+\left[\theta(M-N+1)-\frac{1}{2}\right]\bigl[1-(-1)^{n_{i}}\bigr]d_{i}+\theta(N-1)d_{i}+\theta(N-1)d_{i}\right)\\
 &\times\bigl[x_{1}^{n_{1}}\cdots x_{l}^{n_{l}}\bigr]+\sum_{j\ne i}\bigl(x_{j}p_{1}^{j}+x_{j}p_{2}^{j}\bigr).
 \end{align*}

 One can see from the above expression that, after a single action of $D_{i}$, $x_{1}^{n_{1}}\dots ,x_{l}^{n_{l}}$ splits into two parts. The $x_{i}$-power of the first part decreases by 1. The second part has a common factor~$x_{j}$, and its $x_{i}$-powers decreases as well while the powers of other variables are unchanged. For the action of \smash{$\prod_{i=1}^{l}D_{i}^{\lambda_{i}}$}, we repeat the above action by another $|\lambda|-1$ times. Since all indices~$j$ are distinct, the second part has no chance to become a constant. Hence we only apply the first part each time we apply one more single $D_{i}$, and \smash{$\prod_{i=1}^{l}D_{i}^{\lambda_{i}}$} results in reducing power of $x_{i}$ by $\lambda_{i}$. In the monomials of $m_{\mu}(\vec{x})$ where $l(\mu)=l(\lambda)$, only \smash{$x_{1}^{\lambda_{1}}\dots x_{l}^{\lambda_{l}}$} survives as a nonzero constant. More precisely (we use $\approx$ to omit the $O\bigl(\frac{1}{N}\bigr)$ part),
 \begin{align*}
 b^{\lambda}_{\lambda}\approx{}&
\bigl[z^{0}\bigr]\prod_{i=1}^{l}D_{i}^{\lambda_{i}}m_{\lambda}(\vec{x})
\approx\bigl[z^{0}\bigr]\prod_{i=1}^{l}D_{i}^{\lambda_{i}}\bigl[x_{1}^{\lambda_{1}}\cdots x_{l}^{\lambda_{l}}\bigr]\\
 \approx{}&\bigl[z^{0}\bigr] [\partial_{l}+ (2(N-1)\theta+2(M-N+1)\theta-1 )d_{l} ] [\partial_{l}+2(N-1)\theta d_{l} ]\cdots \\
 &\times [\partial_{l}+2(N-1)\theta d_{l} ] [\partial_{l}+ (2(N-1)\theta+2(M-N+1)\theta-1 )d_{l} ]\partial_{l}\cdots\\
 &\times [\partial_{1}+ (2(N-1)\theta+2(M-N+1)\theta-1 )d_{1} ] [\partial_{1}+2(N-1)\theta d_{1} ] \cdots \\
 &\times [\partial_{1}+2(N-1)\theta d_{1} ] [\partial_{1}+ (2(N-1)\theta+2(M-N+1)\theta-1 )d_{1} ]\partial_{1}
 \bigl[x_{1}^{\lambda_{1}}\cdots x_{l}^{\lambda_{l}} \bigr]\\
 ={}& [\partial_{l}+(2M\theta-1)d_{l} ] [\partial_{l}+2(N-1)\theta d_{l} ]\cdots
 [\partial_{l}+2(N-1)\theta d_{l} ] [\partial_{l}+(2M\theta-1)d_{l} ]\partial_{l}\cdots\\
 & \times [\partial_{1}+(2M\theta-1)d_{1} ] [\partial_{1}+2(N-1)\theta d_{1} ]\cdots [\partial_{1}+2(N-1)\theta d_{1} ] [\partial_{1}+(2M\theta-1)d_{1} ]\\
 &\times\partial_{1} \bigl[x_{1}^{\lambda_{1}}\cdots x_{l}^{\lambda_{l}} \bigr].
 \end{align*}
 As $N\theta\rightarrow\gamma$, $M\theta\rightarrow q\gamma$, the above expression converges to
 \begin{gather*}
 [\partial_{l}+(2q\gamma-1)d_{l} ] [\partial_{l}+2\gamma d_{l} ]\cdots [\partial_{l}+2\gamma d_{l} ] [\partial_{l}+(2q\gamma-1)d_{l} ]\partial_{l}\cdots\\
\qquad \quad{} \times [\partial_{1}+(2q\gamma-1)d_{1} ] [\partial_{1}+2\gamma d_{1} ]\cdots [\partial_{1}+2\gamma d_{1} ] [\partial_{1}+(2q\gamma-1)d_{1} ]\partial_{1}\bigl[x_{1}^{\lambda_{1}}\cdots x_{l}^{\lambda_{l}}\bigr]\\
\qquad=\prod_{i=1}^{l}(\lambda_{i}-1+2q\gamma-1)(\lambda_{i}-2+2\gamma)(\lambda_{i}-3+2q\gamma-1)\cdots(2+2\gamma)(1+2q\gamma-1)\\
\qquad =\prod_{i=1}^{l}\lambda_{i}(\lambda_{i}-2+2q\gamma)(\lambda_{i}-2+2\gamma)(\lambda_{i}-4+2q\gamma)(\lambda_{i}-4+2\gamma)\cdots(2+2q\gamma)\\
\phantom{\qquad =}{}\times(2+2\gamma)2q\gamma.
\end{gather*}
The above argument also implies for $l(\mu)=l(\lambda)$, $\mu\ne\lambda$,
\[
\prod_{i=1}^{l(\lambda)}D_{i}^{\lambda_{i}}m_{\mu}(\vec{x})
\]
 is of order $O\bigl(\frac{1}{N}\bigr)$, and so is $b^{\lambda}_{\mu}$.
\end{proof}

The next proposition deals with the terms involving only length 1 partitions, and identify them with $L$ in~\eqref{eq_dunklaction}.
\begin{prop}\label{prop:dunklactionnonlinear}
 For any even partition $\lambda$ with $|\lambda|=2k$, we have
 \begin{gather}
 \left[\prod_{i=1}^{l(\lambda)}(D_{i})^{\lambda_{i}}\right]\exp(F(x_{1},\dots,x_{N}))\bigr|_{x_{1}=\dots =x_{N}=0}\nonumber\\
 \qquad=\prod_{i=1}^{l(\lambda)}\left(\bigl[z^{0}\bigr](\partial+2\gamma d+\left((q-1)\gamma-\frac{1}{2}\right)d'+*_{g})^{\lambda_{i}-1}g(z)\right)+R+O\left(\frac{1}{N}\right),\label{eq_dunklactionnonlinear}
 \end{gather}
 where $g(z)=\sum_{n=1}^{\infty}nc^{(n)}_{F}z^{n-1}$, and $\partial$, $ d$, $ d'$ and $*_{g}$ are defined in Definition~{\rm\ref{def:operators}}. Moreover, $R$ is a homogeneous polynomial of $c^{v}_{F}$ that $|v|\le 2k$, and each monomial contains at least one~$c^{v}_{F}$ that~${l(v)>1}$, and $O\bigl(\frac{1}{N}\bigr)$ is a homogeneous polynomial of $c^{v}_{F}$ whose coefficients are of order~$O\bigl(\frac{1}{N}\bigr)$ in the limit regime $N\theta\rightarrow\gamma, M\theta\rightarrow q\gamma$.

\end{prop}
\begin{proof}
 Again by Lemmas~\ref{lem:generic} and \ref{lem:generic2}, we only take the generic part of the action of Dunkl operators, namely, all indices $j$ involved are distinct and bigger than $l(\lambda)$, and the remaining part becomes $O\bigl(\frac{1}{N}\bigr)$ in~\eqref{eq_dunklactionnonlinear}.

Moreover, we only consider the polynomials involving only \smash{$c^{(n)}_{F}$} ($n=2, 4,\dots,2k)$, which are corresponding to $m_{(n)}(\vec{x})$, and all other terms are collected in $R$ and $O\bigl(\frac{1}{N}\bigr)$.
Hence, we only look at the action
 \begin{gather*}
 \left[\prod_{i=1}^{l(\lambda)}(D_{i})^{\lambda_{i}}\right]\exp\left(\sum_{n=2}^{2k}c^{(n)}_{F}m_{(n)}(\vec{x})\right)\Bigr|_{x_{1}=\dots =x_{N}=0}\\
 \qquad=\left[\prod_{i=1}^{l(\lambda)}(D_{i})^{\lambda_{i}}\right]\prod_{t=1}^{N}\exp\left(\sum_{n=2}^{2k}c^{(n)}_{F}(x_{t}^{n})\right)\Bigr|_{x_{1}=\dots =x_{N}=0}.
 \end{gather*}

\begin{Claim}
\begin{gather}
\left[\prod_{i=1}^{l(\lambda)}(D_{i})^{\lambda_{i}}\right]\prod_{t=1}^{N}\exp\left(\sum_{n=2}^{2k}c^{(n)}_{F}(x_{t}^{n})\right)\Bigr|_{x_{1}=\dots =x_{N}=0}\nonumber\\
\qquad=\prod_{i=1}^{l(\lambda)}\Bigg((D_{i})^{\lambda_{i}-1}\partial_{i}\left[ \prod_{t=1}^{N}\exp\left(\sum_{n=2}^{2k}c^{(n)}_{F}(x_{t}^{n})\right)\right]\Bigr|_{x_{i}=0}\Bigg).\label{eq_claim}
\end{gather}
\end{Claim}
\begin{proof}
Since the indices $j$ of $D_{i}$ are distinct and bigger than $l(\lambda)$, for $i_{1}\ne i_{2}$, \smash{$D_{i_{1}}^{\lambda_{i_{1}}}$} and \smash{$D_{i_{2}}^{\lambda_{i_{2}}}$} are acting on two groups of disjoint variables. Hence
the action of each \smash{$(D_{i})^{\lambda_{i}}$} factors. Moreover, the first $D_i$ acts as $\partial_{i}$ for the same reason as in the proof of Lemma~\ref{lem:generic2}.\end{proof}

Without loss of generality, consider $i=1$.
\[\partial_{1}\left[\prod_{t=1}^{N}\exp\left(\sum_{n=2}^{2k}c^{(n)}_{F}(x_{t}^{n})\right)\right]=g(x_{1})\prod_{t=1}^{N}\exp\left(\sum_{n=2}^{2k}c^{(n)}_{F}(x_{t}^{n})\right).\]
By~\eqref{eq_hz1}--\eqref{eq_hz4}, it suffices to consider the explicit action of $D_{i}$ on
\[
H(x_{1})\prod_{t=1}^{N}\exp\left(\sum_{n=2}^{2k}c^{(n)}_{F}(x_{t}^{n})\right),
\]
 where $H(x_1)$ is a polynomial of $x_{1}$, and
\[
D_{1}\left[H(x_{1})\prod_{t=1}^{N}\exp\left(\sum_{n=2}^{2k}c^{(n)}_{F}(x_{t}^{n})\right)\right]=[D_{1}H(x_{1})+g(x_{1})]\prod_{t=1}^{N}\exp\left(\sum_{n=2}^{2k}c^{(n)}_{F}(x_{t}^{n})\right).\]
Hence, we have for $i=1,2,\dots,l(\lambda)$,
\[D_{i}^{\lambda_{i}-1}\partial_{i}\left[\prod_{t=1}^{N}\exp
\left(\sum_{n=2}^{2k}c^{(n)}_{F}(x_{t}^{n})\right)\right]\Biggr|_{x_{1}=\dots =x_{N}=0}=(D_{i}+*_{g})^{\lambda_{i}-1}g(x_{i})\bigr|_{x_{i}=0}.\]

Again by Lemmas~\ref{lem:generic}, \ref{lem:generic2} and \ref{lem:dunklonmonomial}, up to $O\bigl(\frac{1}{N}\bigr)$ error,
\begin{align*}
(D_{i}+*_{g})^{\lambda_{i}-1}g(x_{i})\bigr|_{x_{i}=0}&\approx\left(\partial_{i}+2(N-1)\theta d_{i}+\left[\theta(M-N+1)-\frac{1}{2}\right]d_{i}'\right)^{\lambda_{i}-1}g(x_{i})\bigr|_{x_{i}=0}\\
&\xrightarrow[N\theta\rightarrow\gamma]{M\theta\rightarrow q\gamma}\left(\partial_{i}+2\gamma d_{i}+\left[(q-1)\gamma-\frac{1}{2}\right]d_{i}'+*_{g}\right)^{\lambda_{i}-1}g(x_{i})\bigr|_{x_{i}=0},
\end{align*}
and plugging this back to~\eqref{eq_claim} gives
\begin{gather*}
 \left[\prod_{i=1}^{l(\lambda)}(D_{i})^{\lambda_{i}}\right]\prod_{t=1}^{N}\exp\left(\sum_{n=1}^{2k}c^{(n)}_{F}(x_{j}^{n})\right)\bigr|_{x_{1}=\dots =x_{N}=0}\\
\qquad=\prod_{i=1}^{l(\lambda)}\left(\bigl[z^{0}\bigr]\left(\partial+2\gamma d+\left[(q-1)\gamma-\frac{1}{2}\right]d'+*_{g}\right)^{\lambda_{i}-1}g(z)\right)+O\left(\frac{1}{N}\right).
\end{gather*}
Proposition~\ref{prop:dunklactionnonlinear} then follows.
\end{proof}

Combining all the results above in this section, we arrive at the expansion~\eqref{eq_dunklaction} representing action of Dunkl operators on $\exp(F(x_{1},\dots,x_{N}))$.

\begin{proof}[Proof of Theorem \ref{thm:technicalequivalence}]
By Propositions~\ref{prop:unifombounded} and~\ref{prop:dunklptodunkld}, the left side of~\eqref{eq_dunklaction} is a homogeneous polynomial of $c^{v}_{F}$ of degree $2k$ with uniformly bounded coefficients in the limit regime. The right side of~\eqref{eq_dunklaction} is a combination of Proposition~\ref{prop:dunkllinear} which gives
\[
b^{\lambda}_{\lambda}\cdot c^{\lambda}_{F}+\sum_{\mu\colon|\mu|=2k,\, l(\mu)>l(\lambda)}b^{\lambda}_{\mu}\cdot c^{v}_{F}
\]
 and~\ref{prop:dunklactionnonlinear} (which gives polynomial L), and note that the only possible overlap of linear terms and terms involving only length 1 partitions is when $\lambda$ itself is $(2k)$, which gives the term subtracted in~\eqref{eq_L}.
\end{proof}

\subsection[q-gamma convolution]{$\boldsymbol{q}$-$\boldsymbol{\gamma}$ convolution}
After stating the equivalence in Theorem~\ref{thm:hightemperaturemainthm}, Theorem~\ref{thm:hightemplln} follows as a direct consequence.

\begin{proof}[Proof of Theorem~\ref{thm:hightemplln}]
For each $M\ge N$, $\theta>0$, let \smash{$G_{\theta,N,M}^{a}$}, \smash{$G_{\theta,N,M}^{b}$}, \smash{$G_{\theta,N,M}^{c}$}, denote the~type BC Bessel generating function of $\Vec{a}(N)$, \smash{$\Vec{b}(N)$} and \smash{$\Vec{c}_{N}=\Vec{a}(N)\boxplus_{N,M}^{\theta}\Vec{b}(N)$}. Then
\[G_{\theta,N,M}^{c}(x_{1},\dots,x_{N})=G_{\theta,N,M}^{a}(x_{1},\dots,x_{N})\cdot G_{\theta,N,M}^{b}(x_{1},\dots,x_{N}),\]
and hence partial derivatives of \smash{$\ln(G_{\theta,N,M}^{c})$} are equal to the sum of the ones of \smash{$\ln(G_{\theta,N,M}^{b})$} and~\smash{$\ln(G_{\theta,N,M}^{a})$}. By assumption of the theorem, $\{\Vec{a}(N)\}$ and $\bigl\{\Vec{b}(N)\bigr\}$ satisfy LLN condition, then by Theorem~\ref{thm:hightemperaturemainthm} they are $q$-$\gamma$-LLN appropriate. Hence by Definition~\ref{def:llnappropriateness} $\{\Vec{c}_{N}\}$ is also $q$-$\gamma$-LLN appropriate. By Theorem~\ref{thm:hightemperaturemainthm}, again $\{\Vec{c}_{N}\}$ satisfies LLN. \end{proof}

\section[q-gamma cumulants and moments]{$\boldsymbol{q}$-$\boldsymbol{\gamma}$ cumulants and moments}\label{sec:momentcumulant}
Fix $q\ge 1$, $\gamma>0$, in this section we continue with the limit regime $N,M\rightarrow\infty$, $\theta\rightarrow 0$, $N\theta\rightarrow\gamma$, $M\theta\rightarrow q\gamma$. Definition~\ref{def:operators} introduces a map $\mathrm{T}_{\kappa\rightarrow m}^{q,\gamma}$ in terms of operators, that sends the real sequence $\{\kappa_{l}\}_{l=1}^{\infty}$ to another real sequence $\{m_{k}\}_{k=1}^{\infty}$. We keep the interpretation from Theorem~\ref{thm:hightemperaturemainthm}, that is, we call $\{\kappa_{l}\}_{l=1}^{\infty}$ the $q$-$\gamma$ cumulants and take $\kappa_{l}=0$ for all odd $l$. In this section, we give a more combinatorial description of $T_{\kappa\rightarrow m}^{q,\gamma}$. After that, we also provide an explicit relation of $T_{m\rightarrow \kappa}^{q,\gamma}$ in terms of generating functions, and by taking $q$, $\gamma$ to some extreme values, we set the connections of our $q$-$\gamma$-cumulants to the usual cumulants and (rectangular) free cumulants in free probability theory, and also to the $\gamma$-cumulants defined in~\cite{BCG}, that arises in the high temperature regime of self-adjoint matrix additions.

\subsection[From q-gamma cumulants to moments]{From $\boldsymbol{q}$-$\boldsymbol{\gamma}$ cumulants to moments}\label{sec:cumulanttomoment}
We start by introducing some basic notions of set partitions, which are necessary for the statement of the main theorem.

For $k\in \Z_{\ge 1}$, a \emph{set partition} $\pi$ of $[k]$ is a way to write $[k]:=\{1,2,\dots,k\}$ as disjoint union of sets $B_{1},\dots,B_{s}$ for some $s$. We write $\pi=B_{1}\sqcup B_{2}\sqcup\dots \sqcup B_{s}$, and denote the space of all set partitions of $[k]$ by $P(k)$. Given a set partition $\pi$, for each $B_{i}$ let $\min(B_{i})$ and $\max(B_{i})$ denote the minimal and maximal number in the subset $B_{i}$ of $[k]$, and for simplicity, we label $B_{1},\dots,B_{s}$ by $\min(B_{i})$ in increasing order.

In this paper, we are in particular interested in the non-crossing partitions.
\begin{Definition}\label{def:noncrossingpartition}
 Fix $k\in \Z_{\ge 1}$, a set partition $\pi=B_{1}\sqcup\dots \sqcup B_{s}$ of $[k]$ is non-crossing if for any $l=2,\dots,s$, and any $j=1,2,\dots,l-1$, the elements in $B_{j}$ are either bigger than $\max(B_{l})$ or smaller than $\min(B_{l})$. See Figure~\ref{Figure-partition}. Denote the set of all non-crossing partitions of $[k]$ by~${\rm NC}(k)$.
\end{Definition}

Each set partition can be realized visually as a collection of blocks $B_{1},\dots,B_{s}$ with $k$ legs in total. We say that $|B_{i}|$, the number of elements in $B_{i}$, is the \emph{size} of $B_{i}$, which also counts the number of legs in the corresponding block, see Figure~\ref{Figure-partition}. From this point of view, $\pi$ is non-crossing if and only if the legs of one block does not cross any other blocks.
\begin{figure}[htpb]
\centering

 \includegraphics[width=0.68\linewidth]{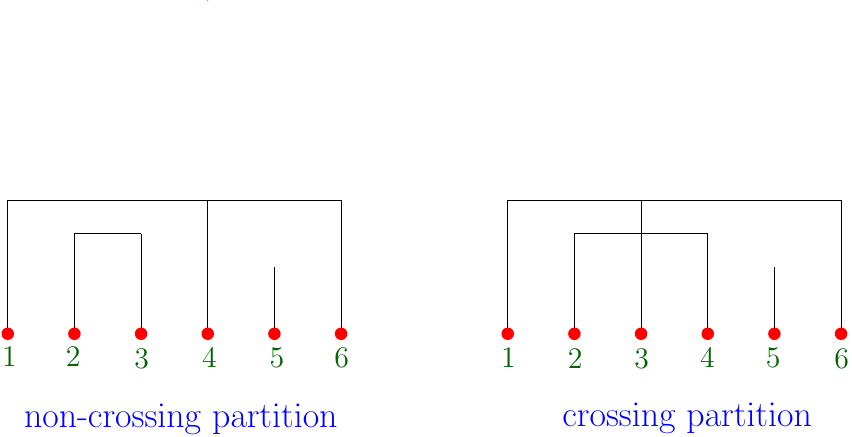}
 \caption{\small{The graph on the left represents a noncrossing partition $\pi$ of $[6]$, where $B_{1}=\{1,4,6\}$, $B_{2}=\{2,3\}$, $B_{3}=\{5\}$, and the graph on the right represents a crossing partition $\pi'$ of $[6]$, where $B_{1}=\{1,3,6\}$, $B_{2}=\{2,4\}$, $B_{3}=\{5\}$.}}\label{Figure-partition}
\end{figure}

Next we define a quantity associated with the non-crossing set partition $\pi$.
\begin{Definition}\label{def:weight}
 Given $\pi=B_{1}\sqcup\dots \sqcup B_{s}\in {\rm NC}(k)$, for $i=1,2,\dots,s$, let $P_{i}=\#$ of elements in $B_{1},\dots,B_{i}$ bigger than $\min(B_{i})$, and $Q_{i}=\#$ of elements in $B_1,\dots,B_{i}$ bigger than $\max(B_{i})$ (note that $P_{i}-Q_{i}=|B_{i}|-1$ by definition of $\pi$ being non-crossing). Let $C_{1},C_{2},\dots $ be the countable sequence of constants \[2q\gamma, \ 2\gamma+2, \ 2q\gamma+2, \ 2\gamma+4, \ 2q\gamma+4,\ 2\gamma+6,\ 2q\gamma+6,\ \dots, \] respectively. Then we define
 \smash{$
 W(\pi)=\prod_{i=1}^{s}\bigl[C_{Q_{i}+1}C_{Q_{i}+2}\cdots C_{P_{i}}\bigr]$}.
\end{Definition}
\begin{ex}\label{ex:partition}
In Figure~\ref{Figure-partition_example}, $\pi$ is a non-crossing partition of $[14]$, such that $B_{1}=\{1,7,8\}$, $B_{2}=\{2,3,6\}$, $B_{3}=\{4,5\}$, $B_{4}=\{9,10,13,14\}$, and $B_{5}=\{11,12\}$. Moreover, $P_{1}=2$, $Q_{1}=0$, $P_{2}=4$, $Q_{2}=2$, $P_{3}=4$, $Q_{3}=3$, $P_{4}=3$, $Q_{4}=0$, $P_{5}=3$, $Q_{5}=2$, so $W(\pi)=C_{1}C_{2}\cdot C_{3}C_{4}\cdot C_{4}\cdot C_{1}C_{2}C_{3}\cdot C_{3}=C_{1}^{2}C_{2}^{2}C_{3}^{3}C_{4}^{2}=(2q\gamma)^{2}(2\gamma+2)^{2}(2q\gamma+2)^{3}(2\gamma+4)^{2}$.
\end{ex}
\begin{figure}[htpb]
\centering

 \includegraphics{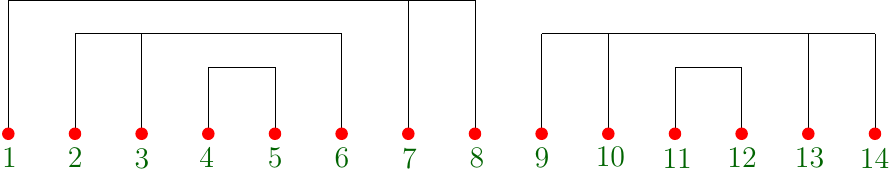}
 \caption{\small{The graphical representation of the non-crossing partition in Example~\ref{ex:partition}.}\label{Figure-partition_example}}
\end{figure}
We also introduce a notion of \emph{even partition} that will be used later.
\begin{Definition}\label{def:evenpartition}
 We say $\pi$ is even if $|B_{1}|,\dots,|B_{s}|$ are all even, and denote the collection of all non-crossing even set partitions of $[2k]$ by $\mathfrak{NC}(2k)$, for some $k\in \Z_{\ge 1}$.
\end{Definition}
The following main theorem of this section gives the combinatorial expression of moments as polynomials of $q$-$\gamma$ cumulants, whose coefficients are given by $W(\pi)$.
\begin{thm}[$q$-$\gamma$ cumulants to moments formula]\label{thm:cumulanttomomentcomb}
 Let $\{\kappa_{l}\}_{l=1}^{\infty}$, $\{m_{k}\}_{k=1}^{\infty}$ be two real sequences such that $\kappa_{l}=0$ for all odd $l$, and $\{m_{2k}\}_{k=1}^{\infty}=\mathrm{T}_{\kappa\rightarrow m}^{q,\gamma}(\{\kappa_{l}\}_{l=1}^{\infty})$. Then for any $k=1,2,\dots $,
 \begin{equation}\label{eq_cumulanttomomentcomb}
 m_{2k-1}=0,\qquad m_{2k}=\sum_{\pi\in \mathfrak{NC}(2k)}W(\pi)\prod_{B_{i}\in \pi}\kappa_{|B_{i}|}.
 \end{equation}

\end{thm}
\begin{Remark}
 It is well known (see, e.g., \cite[Proposition 9.8]{NS}) that non-crossing partitions are in bijection with the so-called Lukasiewicz paths, i.e., lattice paths whose steps take value in $\Z_{\ge -1}$ with starting and ending point at height 0. Therefore, the cumulant expression in~\eqref{eq_cumulanttomomentcomb} for $m_{2k}$ can be rewritten as a weighted sum over Lukasiewicz paths of length 2k. Such interpretation is used in the subsequent work~\cite{KX} and~\cite{CDM}. In particular, the moment expression in \cite[Theorem~3.7]{CDM} is similar to~\eqref{eq_cumulanttomomentcomb} and \cite[Theorem~3.10]{BCG}. But notice that in~\cite{CDM} Lukasiewicz paths are not allowed to have horizontal steps at height 0.
\end{Remark}
\begin{ex}
 By manipulating~\eqref{eq_cumulanttomomentcomb}, we have the explicit expression of the first two nontrivial moments in terms of $q$-$\gamma$ cumulants
 \begin{gather}\label{eq_cumulanttomomentex}
 m_{2}=2q\gamma \kappa_{2},\qquad
 m_{4}=2q\gamma(2\gamma+2)(2q\gamma+2)\kappa_{4}+\bigl[(2q\gamma)^{2}+2q\gamma(2\gamma+2)\bigr]\kappa_{2}^{2}.
 \end{gather}
\end{ex}
\begin{proof}[Proof of Theorem~\ref{thm:cumulanttomomentcomb}]
 Recall from Definition~\ref{def:operators} that
 \[m_{2k}=\bigl[z^{0}\bigr]\left(\partial+2\gamma d+\left[(q-1)\gamma-\frac{1}{2}\right]d'+*_{g}\right)^{2k-1}g(z),
 \]
 where $g(z)=\sum_{l=1}^{\infty}\kappa_{l}z^{l-1}$, i.e., we act the operator $D:=\partial+2\gamma d+\left((q-1)\gamma-\frac{1}{2}\right)d'+*_{g}$ on~$g(z)$ by $2k-1$ times, then take the constant term of the resulting expression.

 An explicit calculation shows that \[ \partial+2\gamma d+\left[(q-1)\gamma-\frac{1}{2}\right]d'+*_{g}=\sum_{l=1}^{\infty}U_{l}+
 \tilde{d}, \]
 where $U_{l}$, $l=1,2,\dots$, and $\tilde{d}$ are linear operators acting on polynomials of $z$, such that for $m\in \Z_{\ge 0}$,
 \begin{align*}
 U_{l}(z^{m})=\kappa_{l}z^{m+l-1},\qquad
 \tilde{d}(z^{m})=\begin{cases}
 (2q\gamma+m-1)z^{m-1}, & m\ \text{is odd};\\
 (2\gamma+m)z^{m-1}, &m\ \text{is even}, \ m\ge 1,\\
 0, &m=0.
 \end{cases}
 \end{align*}
Therefore,
 \[
 m_{2k}=\sum_{T(n)\colon 1\le n\le 2k}\bigl[z^{0}\bigr]\prod_{n=1}^{2k}T(n)(1),
 \]
 where the summation in the last line is over the finitely many choices of the operators $T(n)$. More precisely, each $T(n)$ takes value among $U_{1},\dots,U_{2k}$ and $\tilde{d}$. For each product $T(2k)\cdots T(1)$ with nonzero constant term, one can check easily that there is a one-to-one correspondence with a non-crossing partition $\pi=B_{1}\sqcup\cdots \sqcup B_{s}$ of $[2k]$ as follows:
 \begin{itemize}\itemsep=0pt
 \item if $T(n)=U_{l}$ for some $l\ge 1$, then the $n^{\rm th}$ leg of $\pi$ is the first leg of a block of size $l$,
 \item otherwise $T(n)=\tilde{d}$, and the $n^{\rm th}$ leg of $\pi$ is not the first leg of the block it belongs to,
 \item the non-crossing condition of $\pi$ then gives a unique way to group the legs.
 \end{itemize}
 Moreover, one can check by definition that under such correspondence, \[\bigl[z^{0}\bigr]T(2k)\cdots T(1)(1)=W(\pi)\prod_{i=1}^{s}\kappa_{|B_{i}|}.\]
 This finishes the proof.
\end{proof}

\begin{ex}
 To help clarifying the proof idea of Theorem~\ref{thm:cumulanttomomentcomb}, the partition in Example~\ref{ex:partition} is corresponding to the product $T(14)\cdots T(1)$, where $T(1)=T(2)=U_{3}$, $T(4)=T(11)=U_{2}$, $T(9)=U_{4}$, and all the rest of the operators are equal to $\tilde{d}$.
\end{ex}

\subsection[From moments to q-gamma-cumulants]{From moments to $\boldsymbol{ q}$-$\boldsymbol{\gamma}$-cumulants}\label{sec:momenttocumulant}
Recall that $T_{\kappa\rightarrow m}^{q,\gamma}$ is invertible, and for each $l=1,2,\dots $, $\kappa_{2l}$ is a polynomial of $m_{2}, m_{4},\dots,m_{2l}$ with leading term as a multiple of $m_{2l}$. For example, by reversing~\eqref{eq_cumulanttomomentex}, we have
\[
 \kappa_{2}=\frac{1}{2q\gamma} m_{2},\qquad
 \kappa_{4}=\frac{1}{2q\gamma(2\gamma+2)(2q\gamma+2)}\left[m_{4}-\left(1+\frac{\gamma+1}{q\gamma}\right)m_{2}^{2}\right].
 \]

For the more general cases, we express the generating function of $q$-$\gamma$ cumulants by the generating function of moments.

\begin{thm}\label{thm:momenttocumulant}
 Let $\{m_{2k}\}_{k=1}^{\infty}$, $\{\kappa_{l}\}_{l=1}^{\infty}$ be two real sequences such that
 \[\{\kappa_{l}\}_{l=1}^{\infty}=\mathrm{T}_{m\rightarrow \kappa}^{q,\gamma}(\{m_{2k}\}_{k=1}^{\infty}).\] Then $\kappa_{l}=0$ for all odd $l$, and
 \begin{gather}
 \exp\left[\gamma\sum_{k=1}^{\infty}\frac{m_{2k}}{k}y^{2k}\right]=\sum_{n=0}^{\infty}
c_{n}\cdot y^{2n},\qquad
\exp\left[\sum_{l=1}^{\infty}\frac{\kappa_{2l}}{2l}y^{2l}\right]=\sum_{n=0}^{\infty}\frac{c_{n}}{(\gamma)_{n}(q\gamma)_{n}}2^{-2n}y^{2n}\label{eq_momenttocumulant}
 \end{gather}
 for some auxiliary sequence $\{c_{n}\}_{n=0}^{\infty}$. Here we use the Pochhammer symbol notation
 \[(x)_{n}:=\begin{cases}
 x(x+1)\cdots(x+n-1) &\text{if}\ n\in \Z_{\ge 1},\\
 1 &\text{if}\ n=0.
 \end{cases}\]

 Alternatively, one has the more compact expression
 \begin{equation}\label{eq_momenttocumulantcmpt}
 \exp\left[\sum_{l=1}^{\infty}\frac{\kappa_{2l}}{2l}y^{2l}\right]=\bigl[z^{0}\bigr]\Bigg\{\sum_{n=0}^{\infty}\frac{(yz)^{2n}}{(\gamma)_{n}(q\gamma)_{n}}2^{-2n}\cdot \exp\left[\gamma\sum_{k=1}^{\infty}\frac{m_{2k}}{k}z^{-2k}\right]\Bigg\}.
 \end{equation}
\end{thm}

Before giving the proof, we first present two technical results that will be used.

\begin{Lemma}\label{lem:forthmmomenttocumulant}\quad
\begin{itemize}\itemsep=0pt
 \item[$(a)$] The following Taylor series expansion holds:
 \begin{equation}\label{eq_onerowjackgeneratingfunction}
 \sum_{k=0}^{\infty}Q_{(k)}\bigl(a_{1}^{2},\dots,a_{N}^{2};\theta\bigr)y^{2k}=\prod_{i=1}^{N}\bigl(1-a_{i}^{2}y^{2}\bigr)^{-\theta}.
 \end{equation}

 \item[$(b)$] For $\theta>0$, $y\in \C$ and $\Vec{a}=(a_{1}\ge \dots \ge a_{N}\ge 0)$,
 \[
 B\bigl(\Vec{a},y,0^{N-1};\theta\bigr)=\sum_{k=0}^{\infty}\frac{1}{(N\theta)_{k}(M\theta)_{k}}2^{-2k}Q_{(k)}\bigl(a_{1}^{2},\dots,a_{N}^{2};\theta\bigr)y^{2k},
 \]
 where $Q_{(k)}\bigl(a_{N}^{2};\theta\bigr)$ is defined in~\eqref{eq_jack2}, and $(k)$ denotes the partition $(k,0,\dots,0)\in \Lambda_{N}$. Moreover, the power series converges uniformly in a domain near 0.
\end{itemize}

\end{Lemma}
\begin{proof}
 (a) is a well known result that can be found in \cite[p.~378 and p.~380]{M}. (b) follows from Definition~\ref{def:bessel} and~\eqref{eq_jack2}, after taking $(x_{1},\dots,x_{N})=\bigl(y,0^{N-1}\bigr)$.
\end{proof}

\begin{proof}[Proof of Theorem~\ref{thm:momenttocumulant}]
 First we note that~\eqref{eq_momenttocumulant} and~\eqref{eq_momenttocumulantcmpt} are equivalent by comparing the coefficients for $y^{2n}$ for each $n=0,1,2,\dots $, and we will prove~\eqref{eq_momenttocumulant}.

 For now, we assume that there exists a probability measure $\mu$ supported on $[0,\eta]$ for some $\eta>0$, such that for $k=1,2,\dots $,
$m_{2k}=\int x^{k} {\rm d}\mu$.

 We take a sequence of deterministic $N$-tuples $\{\vec{a}(N)\}_{N=1}^{\infty}$ such that $\Vec{a}(N)=(a_{N,1},\dots,a_{N,N})\in [0,\sqrt{\eta}]^{N}$, and define \smash{$\mu_{N}=\frac{1}{N}\sum_{i=1}^{N}\delta_{a_{N,i}^{2}}$}. We choose $\{\Vec{a}(N)\}$ in a way that $\mu_{N}\rightarrow \mu$ weakly as $N\rightarrow \infty$. This implies that the moments of $\mu_{N}$ also converge pointwisely to the corresponding moments of $\mu$, i.e.,
 \[\frac{1}{N}\sum_{i=1}^{N}a_{N,i}^{2k}\longrightarrow m_{2k}.\]
 In other words, $\{\vec{a}(N)\}_{N=1}^{\infty}$ satisfies the LLN condition, and by Theorem~\ref{thm:hightemperaturemainthm} $\{\vec{a}(N)\}_{N=1}^{\infty}$ is $q$-$\gamma$-LLN-appropriate.

 By Lemma~\ref{lem:forthmmomenttocumulant} (a),
 \begin{align}
 \sum_{k=0}^{\infty}Q_{(k)}\bigl(a_{N,1}^{2},\dots,a_{N,N}^{2};\theta\bigr)y^{2k}&=\prod_{i=1}^{N}\bigl(1-a_{N,i}^{2}y^{2}\bigr)^{-\theta}=\exp\Bigg[-\theta\sum_{i=1}^{N}\ln\bigl(1-a_{N,i}^{2}y^{2}\bigr)\Bigg]\nonumber\\
 &=\exp\Bigg[\theta N\sum_{k=1}^{\infty}\frac{y^{2k}}{k}\frac{1}{N}\sum_{i=1}^{N}(a_{N,i})^{2k}\Bigg]\label{eq_ck1}
 \end{align}
 as a formal power series. Taking $N\rightarrow\infty$, $\theta\rightarrow 0$, $N\theta\rightarrow\gamma$, the coefficients of the exponent on the right side converge by the LLN assumption, which implies the convergence of $Q_{(k)}\bigl(a_{N,1}^{2},\dots,a_{N,N}^{2};\theta\bigr)$ for each $k\in \Z_{\ge 0}$. Then the above equality becomes
 \begin{equation}\label{eq_ck2}
 \sum_{k=0}^{\infty}c_{k}\cdot y^{2k}=\exp\left[\gamma\sum_{k=1}^{\infty}\frac{m_{k}}{k}y^{2k}\right],
 \end{equation}
 where $c_{k}$ is the point-wise limit of $Q_{(k)}\bigl(a_{N,1}^{2},\dots,a_{N,N}^{2};\theta\bigr)$. This defines $\{c_{k}\}_{k=1}^{\infty}$ in terms of $\{m_{2k}\}_{k=1}^{\infty}$.

 On the other hand, since $\mu_{N}$ is deterministic, its type BC Bessel generating function is equal to its Bessel function. And we have for $l=1,2,\dots $
 \[
 \left(\frac{\partial}{\partial y}\right)^{2l} \ln\bigl[ B\bigl(\Vec{a}(N),y,0^{N-1};\theta\bigr)\bigr]\bigr|_{y=0}\xrightarrow[N\theta\rightarrow\gamma]{M\theta\rightarrow q\gamma}(2l-1)!\cdot \kappa_{2l}.
 \]

By Lemma~\ref{lem:forthmmomenttocumulant}\,(b), the above equation is equivalent to
\begin{gather}
 \left(\frac{\partial}{\partial y}\right)^{2l} \ln \left[\sum_{k=0}^{\infty}\frac{1}{(N\theta)_{k}(M\theta)_{k}}2^{-2k}Q_{(k)} \bigl(a_{N,1}^{2},\dots,a_{N,N}^{2};\theta\bigr)y^{2k}\right]\biggr|_{y=0}\nonumber\\
 \qquad\xrightarrow[N\theta\rightarrow\gamma]{M\theta\rightarrow q\gamma}(2l-1)!\cdot \kappa_{2l}.\label{eq_derivativeofqgeneratingfunction}
\end{gather}
 Also, since type BC Bessel function is analytic over $x_{1},\dots,x_{N}$, so is its logarithm near 0, and by Taylor expanding \smash{$\ln \bigl[\sum_{k=0}^{\infty}\frac{1}{(N\theta)_{k}(M\theta)_{k}}2^{-2k}Q_{(k)}\bigl(a_{N,1}^{2},\dots,a_{N,N}^{2};\theta\bigr)y^{2k}\bigr]$} we see that each~$\kappa_{2l}$ is a polynomial of finitely many terms \smash{$\frac{1}{(N\theta)_{k}(M\theta)_{k}}2^{-2k}Q_{(k)}\bigl(a_{N,1}^{2},\dots,a_{N,N}^{2};\theta\bigr)$}, each of which converges to \smash{$\frac{1}{(\gamma)_{k}(q\gamma)_{k}}2^{-2k}c_{k}$}.

 We claim that as $N\theta\rightarrow\gamma$, $M\theta\rightarrow q\gamma$, \smash{$\sum_{k=0}^{\infty}\frac{1}{(N\theta)_{j}(M\theta)_{j}}2^{-2k}Q_{(k)}\bigl(a_{N,1}^{2},\dots,a_{N,N}^{2};\theta\bigr)y^{2k}$} converges uniformly on a domain near 0. Indeed, the pointwise convergence of coefficient of $y$ is already given above, and to obtain a tail bound of the power series, first note that we have assumed that $a_{N,i}$ are uniformly bounded, then by writing each $Q_{(k)}\bigl(a_{N,1}^{2},\dots,a_{N,N}^{2};\theta\bigr)$ as a~contour integral of the right side of~\eqref{eq_onerowjackgeneratingfunction} on the circle $\{z\colon |z|=r\}$ for some $r$ small enough, we see that $Q_{(k)}\bigl(a_{N,1}^{2},\dots,a_{N,N}^{2};\theta\bigr)$ are uniformly bounded by $C\cdot r^{-2k}$ for some constant $C$ and $r$.
 By~\eqref{eq_ck1}, \eqref{eq_ck2}, the limit is
 \[
 \sum_{k=0}^{\infty}\frac{c_{k}}{(\gamma)_{k}(q\gamma)_{k}}y^{2k},
 \]
 but since the uniform convergence of analytic functions implies convergence of derivatives, it is also equal to $\exp\left[\sum_{l=1}^{\infty}\frac{\kappa_{2l}}{2l}y^{2l}\right]$ by~\eqref{eq_derivativeofqgeneratingfunction}. Hence these two functions are equal.

 It remains to generalize to the case where $m_{2}, m_{4},\dots $ are arbitrary real sequence. For each $l=1,2,\dots $, $\kappa_{2l}$ is a polynomial $h_{l}(m_{2},m_{4},\dots,m_{2l})$ of degree at most $l$, where the expression of $h_{l}$ is given by~\eqref{eq_momenttocumulantcmpt}, while on the other hand, for each $\kappa_{2l}$, \eqref{eq_cumulanttomomentcomb} gives another polynomial of degree at most $l$, such that $\kappa_{2l}=h'_{l}(m_{2},m_{4},\dots,m_{2l})$. What we need to show is that $h_{l}=h_{l}'$ for all $l$.

 Fix $l\ge 1$. We have already shown that $h_{l}(m_{2},m_{4},\dots,m_{2l})=h_{l}'(m_{2},m_{4},\dots,m_{2l})$, when $m_{2},m_{4},\dots,m_{2l}$ are the first $l$ moments of some compactly supported probability measure $\mu$ on $\R_{\ge 0}$. Clearly, there exists more than $l$ such choices of $m_{2},m_{4},\dots,m_{2l}$, and therefore by fundamental theorem of algebra these two polynomials coincide. The same argument holds for arbitrary $l\ge 1$.
 \end{proof}

\subsection{Connections to self-adjoint additions}
Let $A$, $B$ be two independent $N\times N$ matrices, uniformly chosen from the sets of self-adjoint matrices with deterministic eigenvalues $a_{1}\ge\dots \ge a_{N}$ and $b_{1}\ge\dots \ge b_{N}$ respectively. The study of eigenvalues of $C=A+B$ dates back to~\cite{Vo}, which considers the empirical measure of~$C$ in fixed temperature regime. in the high temperature regime, it was first considered by~\cite{MeP} which introduced the notion of ``high temperature convolution'', and studied later in several texts including~\cite{BCG,Me} and~\cite{Me2}. In particular, it was proved rigorously that when~${N\rightarrow\infty}$, $\theta\rightarrow0$ and $N\theta\rightarrow\gamma$, assuming the empirical measure of $A$, $B$ converge to some deterministic probability measure $\mu_{A}$, $\mu_{B}$ on $\R$, then the empirical measure of $C$ converges to some deterministic probability measure $\mu_{C}$, which is named as the $\gamma$-convolution of $\mu_{A}$ and $\mu_{B}$.

There is a collection of quantities $\{\kappa_{l}^{\gamma}\}_{l=1}^{\infty}$ introduced in~\cite{MeP} and~\cite{BCG} independently by different approaches, such that
$\kappa^{\gamma}_{l}(\mu_{C})=\kappa^{\gamma}_{l}(\mu_{A})+\kappa^{\gamma}_{l}(\mu_{B})$
for each $l\ge 1$. In the following, we refer to~\cite{BCG}. We write $\{m_{k}'\}_{k=1}^{\infty}=\mathrm{T}_{\kappa\rightarrow m}^{\gamma}\bigl(\bigl\{\kappa^{\gamma}_{l}\bigr\}_{l=1}^{\infty}\bigr)$, where~${m'_{k}\in \R}$ denotes the \smash{$k^{\rm th}$} moment of the limiting empirical measure, and $\mathrm{T}_{\kappa\rightarrow m}^{\gamma}$ is a map that gives moment-cumulant relation of $\gamma$-convolution. While in this paper we are considering addition of a different type of matrices, we find a limit transition in the high temperature regime from rectangular addition to self-adjoint addition, which is stated in terms of cumulants.

\begin{thm}\label{thm:gammacumulant}
 Given a real sequence $\{\kappa_{l}\}_{l=1}^{\infty}$ such that $\kappa_{l}=0$ for all odd $l$, let $\{m_{2k}\}_{k=1}^{\infty}=\mathrm{T}_{\kappa\rightarrow m}^{q,\gamma}(\{\kappa_{l}\}_{l=1}^{\infty})$, $m_{k}'=\frac{m_{2k}}{(q\gamma)^{k}}$, $\kappa_{l}'=2^{2l-1}\kappa_{2l}$ for $l=1,2,\dots $. Then
 \[
 \lim_{q\rightarrow\infty}\{m_{k}'\}_{k=1}^{\infty}=\mathrm{T}_{\kappa\rightarrow m}^{\gamma}(\{\kappa_{l}'\}_{l=1}^{\infty}).
 \]
\end{thm}
\begin{proof}
 This follows from a straightforward limit transition of~\eqref{eq_momenttocumulant} under the assigned rescaling, and the moment-cumulant relation of $\gamma$-convolution in \cite[Theorem 3.11]{BCG}.
\end{proof}

Moreover, we point out that the combinatorial moment-cumulant formula of $\gamma$-convolution, given in \cite[Theorem 3.10]{BCG} can be expressed in an alternate way similar to our Theorem~\ref{thm:cumulanttomomentcomb}.
\begin{prop}\label{prop:gammacumulantcomb}
 Let $\{m_{k}'\}_{k=1}^{\infty}=\mathrm{T}_{\kappa\rightarrow m}^{\gamma}\bigl(\bigl\{\kappa^{\gamma}_{l}\bigr\}_{l=1}^{\infty}\bigr)$. Then for each $k=1,2,\dots $,
 \[
 m'_{k}=\sum_{\pi\in {\rm NC}(k)}W(\pi)\prod_{B_{i}\in \pi}\kappa^{\gamma}_{|B_{i}|},
 \]
 where $W(\pi)$ is defined in the same way as in Definition~{\rm\ref{def:weight}}, after replacing the values of $C_{1}, C_{2},\dots$ to be $\gamma+1, \gamma+2,\dots$.
\end{prop}
\begin{proof}
 By \cite[Definition 3.7 and Theorem 3.8]{BCG},
$m_{k}=\bigl[z^{0}\bigr](\partial+\gamma d+*_{g})^{k-1}g(z)$, $ k=1,2,\dots$.
 The statement then follows from the similar argument as in the proof of Theorem~\ref{thm:cumulanttomomentcomb}.
\end{proof}

\subsection{Connections to the classical convolutions}\label{sec:connection}
Recall from previous sections that, limit of $\boxplus_{N,M}^{\theta}$ gives the $q$-$\gamma$ convolution $\boxplus_{q,\gamma}$ of two (virtual) probability measures on $\R$, which is linearized by $q$-$\gamma$ cumulants. We show in this section that, under certain limit transition of the parameters $q$, $\gamma$, $\boxplus_{q,\gamma}$ converge to the usual convolution, free convolution, and rectangular free convolution respectively.

We first provide the connection of $q$-$\gamma$ cumulants to usual cumulants. For this, we recall the combinatorial classical moment-cumulant formula: for $k=1,2,\dots $,
 \[m_{k}=\sum_{\pi'\in P(k)}\prod_{i=1}^{s}\kappa'_{|B_{i}|},\]
 where $\{\kappa'_{l}\}_{l=1}^{\infty}$ stands for the usual cumulants, and $P(k)$ is the set of all set partitions of $[k]$. We denote the map that sends $\{m_{k}\}_{k=1}^{\infty}$ to $\{\kappa'_{l}\}_{l=1}^{\infty}$ by $\mathrm{T}_{m\rightarrow \kappa}^{0}$.

 \begin{thm}\label{thm:usualcumulant}
 Given a real sequence $\{m_{2k}\}_{k=1}^{\infty}$, let
 \[
 \{\kappa_{l}\}_{l=1}^{\infty}=\mathrm{T}_{m\rightarrow \kappa}^{q,\gamma}(\{m_{2k}\}_{k=1}^{\infty}),\qquad
 \kappa'_{l}=(q\gamma)^{l}2^{2l-1}(l-1)!\kappa_{2l},
 \]
 then
 \begin{equation}\label{eq_usualcumulantlimit}
 \lim_{\gamma\rightarrow 0, q\gamma\rightarrow\infty}\{\kappa'_{l}\}_{l=1}^{\infty}=\mathrm{T}_{m\rightarrow \kappa}^{0}(\{m_{2k}\}_{k=1}^{\infty}).
 \end{equation}
 \end{thm}
 \begin{proof}
 By Theorem~\ref{thm:cumulanttomomentcomb}, after rescaled by $(q\gamma)^{k}$, the coefficient $W(\pi)$ does not vanish asymptotically only if for each $i=1,2,\dots,m$, $2l_{i}-1:=P_{i}-Q_{i}$,
 there are $l_{i}$ terms in $C_{Q_{i}+1}\cdots C_{P_{i}}$ that contain $q\gamma$. Hence each $Q_{i}$ must be even.

 Recall that ${\rm NC}(k)$ denote the space of all (not necessary even) non-crossing partitions. We say a non-crossing even partition $\pi$ of $[2k]$ is \emph{equivalent to $\pi'\in {\rm NC}(k)$}, if there exists some $\pi'\in {\rm NC}(k)=B_{1}'\sqcup \dots \sqcup B_{s}'$, such that by replacing all element $j\in B_{i}$ by $\{2j-1,2j\}$, we get the set $B_{i}$, for any $i=1,2,\dots,m$.

\begin{Claim} For $\pi=B_{1}\sqcup \dots \sqcup B_{s}\in \mathfrak{NC}(2k)$, each $Q_{i}$ is even if and only if $\pi$ is equivalent to some $\Tilde{\pi}\in {\rm NC}(k)$.
\end{Claim}
 \begin{proof}
 The ``if'' part is clear. For the ``only if'' part, just notice that when $\pi$ is even, non-crossing and each $Q_{i}$ is even, $max(B_{i})$ turn out to be all even. The statement then follows by going over all the legs in the graphical representation of $\pi$ from right to left.
 \end{proof}

 Set $\Tilde{C}_{i}=i$, then after taking the limit,
 \[\frac{C_{Q_{i}+1}\cdots C_{P_{i}}}{(q\gamma)^{l_{i}}}\xrightarrow[q\gamma\rightarrow\infty]{\gamma\rightarrow0} 2^{2l_{i}-1}\Tilde{C}_{Q_{i}+1}\cdots \Tilde{C}_{P_{i}}.\]
In other words,
 \begin{align*}
 m_{2k}&=\sum_{\pi=B_{1}\sqcup \dots \sqcup B_{s}\in \mathfrak{NC}(2k)}W(\pi)\prod_{B\in\pi}\kappa_{|B_{i}|}\\
& \xrightarrow[q\gamma\rightarrow \infty]{\gamma\rightarrow 0}
\sum_{\Tilde{\pi}=\Tilde{B}_{1}\sqcup \dots \sqcup \Tilde{B}_{s}\in {\rm NC}(k)}\prod_{i=1}^{s}[ (Q_{i}+1)\cdots (P_{i})\cdot \kappa'_{|\Tilde{B}_{i}|}]\\
&=\lim_{\gamma\rightarrow0}\sum_{\Tilde{\pi}=\Tilde{B}_{1}\sqcup \dots \sqcup \Tilde{B}_{s}\in {\rm NC}(k)}\prod_{i=1}^{s}[ (\gamma+Q_{i}+1)\cdots (\gamma+P_{i})\cdot \kappa'_{|\Tilde{B}_{i}|}]\\
&=\lim_{\gamma\rightarrow0}\mathrm{T}_{\kappa\rightarrow m}^{\gamma}(\{\kappa_{l}'\}_{l=1}^{\infty})=\mathrm{T}_{\kappa\rightarrow m}^{0}(\{\kappa_{l}'\}_{l=1}^{\infty})=\sum_{\pi'=B_{1}'\sqcup \dots \sqcup B_{s}'\in P(k)}\prod_{i=1}^{s}\kappa'_{|B'_{i}|}.
 \end{align*}
The two equalities in the second to last row hold by Proposition~\ref{prop:gammacumulantcomb} and \cite[Theorem~8.2]{BCG}, respectively. Then~\eqref{eq_usualcumulantlimit} follows from acting $\mathrm{T}_{m\rightarrow \kappa}^{0}$ on both sides.
\end{proof}

\begin{cor} For two real sequences $\{m_{2k}^{a}\}_{k=1}^{\infty}$, $\bigl\{m_{2k}^{b}\bigr\}_{k=1}^{\infty}$, set $m_{2k-1}^{a}=m_{2k-1}^{b}=0$ for $k=1,2,\dots $, and define
 \[
 \{m^{c}_{k}\}_{k=1}^{\infty}:=\lim_{\gamma\rightarrow 0, q\gamma\rightarrow \infty}\bigl[\{m^{a}_{k}\}_{k=1}^{\infty}\boxplus_{q,\gamma}\bigl\{m^{b}_{k}\bigr\}_{k=1}^{\infty}\bigr].
 \]
 Then \smash{$m^{c}_{2k-1}=0$} for $k=1,2,\dots $, and the usual cumulants of \smash{$\{m^{c}_{2k}\}_{k=1}^{\infty}$} are given by the sum of the corresponding usual cumulants of $\{m^{a}_{2k}\}_{k=1}^{\infty}$ and \smash{$\bigl\{m_{2k}^{b}\bigr\}_{k=1}^{\infty}$}, i.e.,
 \[
 \mathrm{T}_{m\rightarrow \kappa}^{0}\bigl(\{m^{c}_{2k}\}_{k=1}^{\infty}\bigr)=\mathrm{T}_{m\rightarrow \kappa}^{0}\bigl(\{m^{a}_{2k}\}_{k=1}^{\infty}\bigr)+\mathrm{T}_{m\rightarrow \kappa}^{0}\bigl(\bigl\{m^{b}_{2k}\bigr\}_{k=1}^{\infty}\bigr).
 \]
 \end{cor}
 \begin{Remark}
 Suppose $\mu_{a}$, $\mu_{b}$ are two probability measures on $\R_{\ge 0}$ that for $k=1,2,\dots $
 \[m_{2k}^{a}=\int x^{k} {\rm d}\mu_{a},\qquad m_{2k}^{b}=\int x^{k} {\rm d}\mu_{b},\]
 then $\{m^{c}_{2k}\}_{k=1}^{\infty}$ are the moments of the usual convolution of $\mu_{a}$ and $\mu_{b}$.
 \end{Remark}

 Next, we consider $\boxplus_{q,\gamma}$ and match its asymptotic behavior with the free convolution and rectangular free convolution. Before that we recall the definitions of their corresponding cumulants. For $k\in \Z_{\ge 1}$, let $\mathrm{T}_{r\rightarrow m}'$ denote the map sending the real sequence $\{r_{l}\}_{l=1}^{\infty}$ of free cumulants to the sequence $\{m_{k}\}_{k=1}^{\infty}$ of moments. Then there is a moment-cumulant formula
 \[
 m_{k}=\sum_{\pi=B_{1}\sqcup\dots \sqcup B_{s}\in {\rm NC}(k)}\prod_{B\in\pi}r_{|B_{i}|}
\]
for $\{m_{k}\}_{k=1}^{\infty}=\mathrm{T}_{r\rightarrow m}'
(\{r_{l}\}_{l=1}^{\infty})$. See, e.g., \cite{No} for a reference.

Similarly, as defined in \cite[Section 3.1]{B1}, rectangular free cumulants are a real sequence $\{c_{l}^{q}\}_{l=1}^{\infty}$ parametrized by $q\ge 1$, such that for $l=1,2,\dots $, $c^{q}_{2l-1}=0$, and $c^{q}_{2l}$ are related with moments~${\{m_{k}\}_{k=1}^{\infty}}$ by the following identities:
\begin{equation}\label{eq_recfreecumulant}
 m_{2k}=\sum_{\pi\in \mathfrak{NC}(2k)}q^{-e(\pi)}\prod_{B\in \pi}c_{|B_{i}|},
\end{equation}
where $e(\pi)=\#$ of blocks $B_{i}$ with even $\min(B_{i})$,
and $m_{2k-1}=0$ for $k=1,2,\dots $. Denote the map sending even moments to rectangular free cumulants by $\mathrm{T}^{\infty}_{m\rightarrow \kappa}$, i.e., $\mathrm{T}^{\infty}_{m\rightarrow \kappa}(\{m_{2k}\}_{k=1}^{\infty})=\{c_{l}\}_{l=1}^{\infty}$.
 \begin{thm}\label{thm:connectiontofreecumulant}
 Given a real sequence $\{m_{2k}\}_{k=1}^{\infty}$, $q\ge 1$, let
 \[\{\kappa_{l}\}_{l=1}^{\infty}=\mathrm{T}_{m\rightarrow \kappa}^{q,\gamma}(\{m_{2k}\}_{k=1}^{\infty}),\qquad r_{l}^{\gamma}=(2\gamma)^{l-1}q^{\frac{l}{2}}\cdot \kappa_{l}.\]
 Then we have the following:
\begin{itemize}\itemsep=0pt
 \item[$(a)$]
$\lim_{\gamma\rightarrow\infty}\{r_{l}^{\gamma}\}_{l=1}^{\infty}=\mathrm{T}_{m\rightarrow \kappa}^{\infty}(\{m_{2k}\}_{k=1}^{\infty})$.

 \item[$(b)$]
$\lim_{\gamma\rightarrow\infty}\{r_{l}^{\gamma}\}_{l=1}^{\infty}=\mathrm{T}_{m\rightarrow r}'(\{m_{2k}\}_{k=1}^{\infty})$,
 when $q=1$.
 \end{itemize}
 \end{thm}
 \begin{Remark}
 (b) is a special case of (a) when $q=1$. Such connection of rectangular free convolution and free convolution was first pointed out in \cite[Remark 2.2]{B1}.
 \end{Remark}
 \begin{proof}
 It suffices to prove (a).
 By Theorem~\ref{thm:cumulanttomomentcomb},
 \[m_{2k}=\sum_{\pi\in \mathfrak{NC}(2k)}W(\pi)\prod_{B\in \pi}\kappa_{|B_{i}|},\]
 where $C_{i}(i=1,2,\dots )$ are $2q\gamma, 2\gamma+2, 2q\gamma+2, 2\gamma+4, 2q\gamma+4,\dots$.
 Hence by taking $\gamma\rightarrow \infty$, the right side above becomes
 \begin{equation}\label{eq_cpi}
 \sum_{\pi\in \mathfrak{NC}(2k)}q^{-n(\pi)}\prod_{B\in \pi}C_{|B_{i}|},\end{equation}
 where $n(\pi):=\#$ of blocks $B_{i}$ such that $Q_{i}$ is odd.
 Since
 \begin{align*}
Q_{i}\ \text{is odd}&\iff \text{\ there are odd elements of } B_{1},\dots,B_{i-1} \text{ bigger than }\max(B_{i})\\
&\iff \text{ there are odd elements of } B_{1},\dots,B_{i-1} \text{ smaller than }\min(B_{i})\\
&\iff \min(B_{i}) \text{ is even},
 \end{align*}
 $n(\pi)=e(\pi)$, and~\eqref{eq_cpi} is equal to the right side of~\eqref{eq_recfreecumulant}.
 \end{proof}

 Recall also that similar to $q$-$\gamma$ convolution, free convolution and rectangular free convolution are both binary operation of two probability measures linearized by free cumulants. Therefore, Theorem~\ref{thm:connectiontofreecumulant} implies the following.

 \begin{cor} Given $q\ge 1$, for two real sequences $\{m_{2k}^{a}\}_{k=1}^{\infty}$, $\bigl\{m_{2k}^{b}\bigr\}_{k=1}^{\infty}$, set $m_{2k-1}^{a}=m_{2k-1}^{b}=0$ for $k=1,2,\dots $, and define
 \[
 \{m^{c}_{k}\}_{k=1}^{\infty}:=\lim_{\gamma\rightarrow \infty}\bigl[\{m^{a}_{k}\}_{k=1}^{\infty}\boxplus_{q,\gamma}\bigl\{m^{b}_{k}\bigr\}_{k=1}^{\infty}\bigr].
 \]
 Then the free cumulants of $\{m^{c}_{k}\}_{k=1}^{\infty}$ are given by the sum of the corresponding rectangular free cumulants of $\{m^{a}_{k}\}_{k=1}^{\infty}$ and $\bigl\{m_{k}^{b}\bigr\}_{k=1}^{\infty}$, i.e.,
 \[
 \mathrm{T}_{m\rightarrow \kappa}^{\infty}\bigl(\{m^{c}_{k}\}_{k=1}^{\infty}\bigr)=\mathrm{T}_{m\rightarrow \kappa}^{\infty}\bigl(\{m^{a}_{k}\}_{k=1}^{\infty}\bigr)+\mathrm{T}_{m\rightarrow \kappa}^{\infty}\bigl(\bigl\{m^{b}_{k}\bigr\}_{k=1}^{\infty}\bigr).
 \]
 \end{cor}
 \begin{Remark}
 Suppose $\mu_{a}$, $\mu_{b}$ are two symmetric probability measures on $\R$ that for $k=1,2,\dots $
 \[
 m_{k}^{a}=\int x^{k} {\rm d}\mu_{a},\qquad m_{k}^{b}=\int x^{k} {\rm d}\mu_{b},
 \]
 then $\{m^{c}_{k}\}_{k=1}^{\infty}$ are the moments of the rectangular free convolution of $\mu_{a}$ and $\mu_{b}$.
 \end{Remark}
 \begin{Remark}
 Similar results hold for free convolution when $q=1$.
 \end{Remark}

\subsection[Law of large numbers of Laguerre beta ensembles]{Law of large numbers of Laguerre $\boldsymbol{\beta}$ ensembles}\label{sec:laguerre}
For $M\ge N$ and $\theta=\frac{1}{2},1,2$, let $X$ be an $N\times M$ rectangular random matrix, whose entries are real/complex/real quaternionic i.i.d.\ Gaussian random variables $\mathcal{N}(0,1)$/$\mathcal{N}(0,1)+i \mathcal{N}(0,1)$/$\mathcal{N}(0,1)+i \mathcal{N}(0,1)+j \mathcal{N}(0,1)+k \mathcal{N}(0,1)$. One can check directly that $X$ satisfies the same invariant property given in Section~\ref{sec:addition}, with $N$ random singular values $\Vec{x}_{N}=(x_{N,1}\ge \dots \ge x_{N,N}\ge 0)$.

 The density of $\Vec{x}_{N}$ is (see, e.g., \cite[Chapter 3]{Forrester})
\[
f(\Vec{x}_{N};\theta,N,M)=\frac{1}{Z_{\theta,N,M}}\prod_{i=1}^{N}\left[x_{N,i}^{2\theta (M-N+1)-1}\exp\left(-\frac{1}{2} x_{N,i}^{2}\right)\right]\prod_{1\le j< k\le N}\bigl(x_{N,j}^{2}-x_{N,k}^{2}\bigr)^{2\theta},
\]
where $Z_{\theta,N,M}$ is the normalizing constant. While for general $\theta>0$
there is again no skew field of real dimension $2\theta$,
$f(\Vec{x}_{N};\theta,N,M)$ continues to make sense, and is defined as the so-called Laguerre $\beta$ ensemble.

\begin{Remark}
It is easy to check that $f(\Vec{x}_{N};\theta,N,M)$ is an exponential decaying measure defined in Definition~\ref{def:expdecaying}, and therefore by Theorem~\ref{thm:dunklonbgf}, its type BC Bessel generating function is defined and well-behaved under the action of type BC Dunkl operators.
\end{Remark}

\begin{prop}
 Let $G_{\theta,N,M}^{L}(x_{1},\dots,x_{N})$ denote the type BC Bessel generating function
 \[\int_{x_{N,1}\ge \dots \ge x_{N,N}\ge 0}B(\Vec{x}_{N},x_{1},\dots,x_{N};\theta,N,M)f(\Vec{x}_{N};\theta,N,M){\rm d}x_{M,1}\cdots {\rm d}x_{N,N},\]
 then
 \[G_{\theta,N,M}^{L}(x_{1},\dots,x_{N})=\exp\left[\frac{1}{2}\bigl(x_{1}^{2}+\dots+x_{N}^{2}\bigr)\right].\]

\end{prop}
\begin{proof}
 For $\theta=\frac{1}{2},1,2$, one can use Definition~\ref{def:matrixintegral} and check this by hand. For general $\theta>0$, this is a special case of \cite[Proposition 2.37\,(2)]{Ro}, such that in that identity $y$ is set to be 0, and our $B(\Vec{a},x_{1},\dots,x_{N};\theta,N,M)$ is a symmetric version of $E_{k}(x,z)$, see \cite[Definition 2.35]{Ro}.
\end{proof}

For each $\theta$, $N$, $M$, denote the random empirical measure of $f(\Vec{x}_{N};\theta,N,M)$ by $\mu_{\theta,N,M}:=\smash{\frac{1}{N}\sum_{i=1}^{N}\delta_{x_{N,i}^{2}}}$.
\begin{thm}
 As $N\rightarrow \infty, M\rightarrow \infty, \theta\rightarrow 0, N\theta\rightarrow\gamma, M\theta\rightarrow q\gamma$,
$\mu_{\theta,N,M}\longrightarrow \mu_{q,\gamma}$
 weakly in moments, where $\mu_{q,\gamma}$ is a probability measure on $\R_{\ge 0}$, which is uniquely determined by its moments
 \begin{equation}\label{eq_cumulanttomomentlaguerre}
 m'_{k}=\int_{\R \ge 0} x^{k} {\rm d}\mu_{q,\gamma}=\sum_{\pi}\prod_{i=1}^{k}C_{P_{i}},\qquad \text{for}\quad k=1,2,\dots,
 \end{equation}
 where $C_{l}$, $P_{i}$ are defined in the same way as in Section~{\rm\ref{sec:cumulanttomoment}}, and $\pi$ goes over all non-crossing perfect matchings of $[2k]$.
\end{thm}
\begin{Remark}
 The weak convergence of $\mu_{\theta,N,M}$ to $\mu_{q,\gamma}$ in probability was already proved in~\cite{ABMV}, in which the authors give an explicit density of $\mu_{q,\gamma}$ in terms of the Whittaker function.
\end{Remark}
\begin{proof}
 By taking logarithm and partial derivatives of $G_{\theta,N,M}^{L}$, we have that $\{\Vec{x}_{N}\}$ is $q$-$\gamma$-LLN appropriate with $q$-$\gamma$ cumulant $\kappa_{2}=1$, $\kappa_{l}=0$ for $l\ne 2$. By Theorem~\ref{thm:cumulanttomomentcomb}, only the set partitions that are formed by blocks of size two survive, and in this case $P_{i}=Q_{i}+1$. \eqref{eq_cumulanttomomentlaguerre} is then specified from~\eqref{eq_cumulanttomomentcomb}.

 Since the existence of probability measure $\mu_{q,\gamma}$ is known by~\cite{ABMV}, it remains to show that the moments in~\eqref{eq_cumulanttomomentlaguerre} does correspond to a unique probability measure. This is the so-called Stieltjes moment problem, since the (potential) corresponding measure lies on $[0,\infty)$, see, e.g.,~\cite{Ak}.
 We need to check $
\sum_{k=1}^{\infty}(m'_{k})^{-\frac{1}{2k}}=\infty$.
 Again by~\eqref{eq_cumulanttomomentlaguerre}, $m'_{k}$ is a sum of $\prod_{i=1}^{k}C_{P_{i}}$. Among these summands the biggest term corresponds to $P_{i}=i$ for $i=1,2,\dots,k$, and
 \[\prod_{i=1}^{k}C_{P_{i}}\le W(W+2)(W+4)\cdots (W+2k-2)=2^{k}\frac{\Gamma\bigl(\frac{W}{2}+k\bigr)}{\Gamma\bigl(\frac{W}{2}\bigr)},\]
 where $W:=\operatorname{ max}\{2q\gamma-2,2\gamma\}$.
 The number of non-crossing perfect matching is $\operatorname{Cat}(k)=\frac{1}{k+1}\binom{2k}{k}$, the $k^{\rm th}$ Catalan number. Multiplying these two gives an upper bound of $m'_{k}$. By Stirling approximation, it turns out that
\smash{$
(m'_{k})^{-\frac{1}{2k}}\le W_{1}\cdot \sqrt{k}
$}
 for some positive constant $W_{1}$. Hence the series diverges.
\end{proof}

The limiting measure $\mu_{q,\gamma}$ is an $q$-$\gamma$ analog of the Gaussian and semicircle law, in the sense that their only nonvanishing ($q$-$\gamma$/classical/free) cumulant is $\kappa_{2}=1$.

Moreover, the connections to the usual and free convolution in Section~\ref{sec:connection} continues to hold in this special case. Indeed, one can show from~\eqref{eq_cumulanttomomentlaguerre} that
\[\frac{m'_{k}}{(q\gamma)^{k}}\xrightarrow[q\gamma\rightarrow\infty]{\gamma\rightarrow 0}\sum_{\pi}\prod_{i=1}^{k}(2),\]
where $\pi$ goes over all set partitions of $[k]$ into $k$ blocks
(which is indeed a single one), since any other non-crossing perfect matching has coefficient with $C_{2}=2\gamma+2$, and therefore after rescaled by $(q\gamma)^{k}$ this term vanishes in the limit. The sum on the right is equal to $2^{k}$, which means $m_{k}'=2^{k}$, and
$
\mu_{q,\gamma}\longrightarrow \delta_{2}
$
weakly when $q\gamma\rightarrow\infty, \gamma\rightarrow 0$.

On the other hand, by taking $q=1$, $\gamma\rightarrow \infty$, \eqref{eq_cumulanttomomentlaguerre} becomes
\[
 \frac{m'_{k}}{(2\gamma)^{k}}\longrightarrow \#\ \text{of non-crossing perfect matchings of } [2k]=\operatorname{Cat}(k).
\]
$\operatorname{Cat}(k)$ is exactly the $2k^{\rm th}$ moment of the semicircle law.

\section{A match of convolutions in high and low temperature}\label{sec:duality}
After studying the behavior of rectangular matrix additions in both low and high temperatures, we present a quantitative connection between these two regimes.

Recall that in the low temperature regime, given two deterministic $N$-tuples $\Vec{a}$, $\Vec{b}$, the limit of \smash{$\vec{c}=\Vec{a}\boxplus_{N,M}^{\theta}\Vec{b}$} is a deterministic $N$-tuple $\Vec{\lambda}$, where $\Vec{\lambda}$ is the $(N,M)$-rectangular finite convolution of $\Vec{a}$ and $\Vec{b}$. For an $N$-tuple $\Vec{a}=(a_{1},\dots,a_{N})$, let $r_{i}=a_{i}^{2}$ for $i=1,2,\dots,N$. Let $m_{k}'=\frac{1}{N}\bigl(r_{1}^{k}+\dots+r_{N}^{k}\bigr)$ be the finite version of moments, for $k=1,2,\dots,N$.
Then the $(N,M)$-rectangular finite convolution of $\Vec{a}$ and \smash{$\Vec{b}$} can be thought of as a deterministic binary operation of \smash{$\{m'_{k}(\Vec{a})\}_{k=1}^{N}$} and \smash{$\{m'_{k}(\Vec{b})\}_{k=1}^{N}$}. Similarly, we view the $q$-$\gamma$ convolution of \smash{$\{m_{k}^{a}\}_{k=1}^{\infty}$} and \smash{$\bigl\{m_{k}^{b}\bigr\}_{k=1}^{\infty}$} as inducing a deterministic binary operation on the first $N$ moments \smash{$\{m_{k}^{a}\}_{k=1}^{N}$} and \smash{$\bigl\{m_{k}^{b}\bigr\}_{k=1}^{N}$}.
\begin{thm}\label{thm:duality}
 By identifying $N$ with $-\gamma$, $\frac{M}{N}$ with $q$, $m_{2k}$ with $m'_{k}(-M)^{k}$ for $k=1,2,\dots,N$, the $(N,M)$-rectangular convolution matches the $q$-$\gamma$ convolution as binary operation of first $N$ nontrivial moments.
\end{thm}

The theorem is claiming that under the above identification, the moment-cumulant formula of these two convolutions are the same, and therefore we need to introduce a version of cumulants for the rectangular finite convolution. For this, we refer to
\cite{Gri}, which considers sum of two invariant $N\times M$ ($N=M\lambda$, $ \lambda\in [0,1]$) rectangular matrices as in Section~\ref{sec:lowtemp}, and defines the rectangular finite R-transform as the analog of the R-transform in (classical) free probability theory, in the sense that it linearizes the finite rectangular addition.
\begin{Definition}[{\cite[Definition 3.7]{Gri}}]
\smash{$R^{N,M}_{S_{p_{A}}}(z)$} is the unique polynomial of degree $M$ verifying
\begin{equation}\label{eq_finiterecrtransform}
 R^{N,M}_{S_{p_{A}}}(z)\equiv \frac{-1}{N}z\frac{\rm d}{{\rm d}z}\ln\bigl(\mathbb{E}\bigl[{\rm e}^{-T^{(N,M)}_{S_{p_{A}}}zNM}\bigr]\bigr)\quad \text{mod}\ \bigl[z^{N+1}\bigr],
\end{equation}
where \smash{$T^{(N,M)}_{S_{p_{A}}}$} is a random variable. By \cite[p.~13]{Gri}, for $n=1,2,\dots,N$,
\begin{equation}\label{eq_finiterecadditionrv}
 \mathbb{E}\bigl[\bigl(T^{(N,M)}_{S_{p_{A}}}\bigr)^{n}\bigr]=\frac{n!(M-n)!}{M!}\frac{(N-n)!}{N!}a_{n},
\end{equation}
where $a_{n}=e_{n}(\Vec{r})$.
\end{Definition}
Inspired by the fact that (classical) R-transform is the generating function of free cumulants, we define the rectangular finite cumulants, such that
\smash{$
 R^{N,M}_{S_{p_{A}}}(z)=\sum_{l=1}^{N}
 \kappa^{N,M}_{l}z^{l}$}.
\smash{$\kappa^{N,M}_{1},\dots,\kappa^{N,M}_{N}$} uniquely determine $r_{1},\dots,r_{N}$.

\begin{proof}[Proof of Theorem~\ref{thm:duality}]
 We prove that under the following identification of parameters:
 \begin{gather}
\kappa^{N,M}_{l} \longleftrightarrow \frac{\kappa_{2l}}{2}\gamma^{l-1}\qquad \text{for}\ l=1,2,\dots,N,\qquad
 N \longleftrightarrow -\gamma,\nonumber\\
 \frac{M}{N}\longleftrightarrow q,\qquad
 M^{n}a_{n}\longleftrightarrow c_{n},\qquad
 m_{2k}\longleftrightarrow m_{k}'\cdot (-M)^{k} ,\label{eq_parameterid}
 \end{gather}
 the moment-cumulant relation in rectangular finite convolution and~\eqref{eq_momenttocumulant} match exactly.

 We match the second formula of~\eqref{eq_momenttocumulant} and~\eqref{eq_finiterecrtransform}, which play the role of cumulant generating function in their own setting. Let $y^{2}=(-N)z$, then the second formula of~\eqref{eq_momenttocumulant} becomes
 \begin{gather}
 \exp\left(\sum_{l=1}^{\infty}\frac{\kappa_{2l}}{2l}\gamma^{l}z^{l}\right)=\sum_{n=0}^{\infty}\frac{c_{n}}{(\gamma)_{n}(q\gamma)_{n}}(-N)^{n}z^{n}\nonumber\\
\qquad \Longrightarrow \sum_{l=1}^{\infty}\frac{\kappa_{2l}}{2}\gamma^{l-1}z^{l}=-\frac{1}{N}z\frac{\rm d}{{\rm d}z}\ln\left(\sum_{n=0}^{\infty}\frac{c_{n}}{(\gamma)_{n}(q\gamma)_{n}}(-N)^{n}z^{n}\right).\label{eq_cumulanttomomentmatchform}
 \end{gather}

 It remains to match the right side of~\eqref{eq_cumulanttomomentmatchform} and~\eqref{eq_finiterecrtransform}, i.e., matching
 \[\mathbb{E}\bigl[{\rm e}^{-T^{(N,M)}_{S_{p_{A}}}zNM}\bigr]\qquad \text{with}\quad \sum_{k=0}^{\infty}\frac{c_{k}}{(q\gamma)_{k}(\gamma)_{k}}(-N)^{k}z^{k}\]
 for $k=1,2,\dots,N$.
 This follows by Taylor expanding \smash{${\rm e}^{-T^{(N,M)}_{S_{p_{A}}}zNM}$}, \eqref{eq_finiterecadditionrv} and~\eqref{eq_parameterid}.

 Then we identify the first formula of~\eqref{eq_momenttocumulant} with the moment generating function in rectangular finite convolution. In the latter setting, recall that $a_{n}=e_{n}(\Vec{r})$ for $n=1,2,\dots,N$, and $m_{k}'=\frac{1}{N}p_{k}(\Vec{r})$ for $k=1,2,\dots$.. Moreover, take $r_{i}=0$ for all $i>N$, and identify $a_{n}$ with $e_{n}(\Vec{r})$ formally for $n>N$ (both have value 0) as well. Then on the rectangular finite addition side,
 \begin{align*}
 \sum_{n=0}^{\infty}M^{n}a_{n}y^{2n}&=\sum_{n=0}^{\infty}e_{n}(\Vec{r})\bigl(My^{2}\bigr)^{n}=\prod_{i=1}^{\infty}\bigl(1+r_{i}My^{2}\bigr)\\
 &=\exp\left(-\sum_{k=1}^{\infty}\frac{p_{k}(\Vec{r})(-M)^{k}y^{2k}}{k}\right)=\exp\left(-N\sum_{k=1}^{\infty}\frac{m'_{k}(-M)^{k}y^{2k}}{k}\right).
 \end{align*}
 This matches the first formula of~\eqref{eq_momenttocumulant} under the identification of parameters.
\end{proof}

\begin{Remark}
 After identifying the \smash{$\kappa_{1}^{N,M},\dots,\kappa_{N}^{N,M}$} with the first $N$ even $q$-$\gamma$ cumulants, one can define \smash{$\kappa_{l}^{N,M}$} for $l\ge N+1$ for rectangular finite convolution, by the moment-cumulant relation of $q$-$\gamma$ convolution under the same parameter identification in~\eqref{eq_parameterid}.
\end{Remark}
\begin{Remark}
 Note that in~\eqref{eq_parameterid}, both $N$ and $\gamma$ are positive, hence there is no choice of parameters that the finite rectangular cumulants coincide with $q$-$\gamma$ cumulants. Instead, one can combine the domain of these two groups of parameters, and treat the result as an extension of the moment-cumulant relation to, say, $\gamma\in \R_{>0}\cup \Z_{\le -1}$.
\end{Remark}

\appendix

\section{Limit transition of type BC Bessel functions}
In this appendix, we provide a limit transition of $B(\Vec{a},x_{1},\dots,x_{N};\theta,N,M)$ to a simple symmetric combination of exponents. This transition implies that in the $\theta=0$ regime, the rectangular addition \smash{$\Vec{a}\boxplus_{N,M}^{\theta}\Vec{b}$} becomes the usual convolution of the empirical measures \smash{$\frac{1}{N}\sum_{i=1}^{N}\delta_{a_{i}^{2}}$} and~\smash{$\frac{1}{N}\sum_{i=1}^{N}\delta_{b_{i}^{2}}$}.

\begin{prop}\label{prop:transitiontoexponent}
 Given $\Vec{a}=(a_{1}\ge a_{2}\ge \dots \ge a_{N})$, take $N$ to be fixed, $M\rightarrow \infty$, $\theta\rightarrow 0$, $M\theta \rightarrow \infty$, then
 \[
 B(\Vec{a},M\theta x_{1},\dots,M\theta x_{N};\theta, N,M)\longrightarrow \frac{1}{N!}\sum_{\sigma\in S_{N}}\prod_{i=1}^{N}{\rm e}^{a_{i}^{2}x_{\sigma(i)}^{2}}.
 \]
\end{prop}

\begin{proof}
 This follows from a straightforward calculation. Indeed, by Definition~\ref{def:bessel},
 \begin{gather*}
 B(\Vec{a},M\theta x_{1},\dots,M\theta x_{N};\theta, N,M)\\
 \qquad=\sum_{\mu}\prod_{j=1}^{l(\mu)}\frac{(M\theta)^{\mu_{j}}}{[\theta(M-j+1)]\cdots [\theta(M-j+1)+\mu_{j}-1]}\cdot \frac{\prod_{s\in\mu}[\mu_{i}-j+\theta (\mu_{j}'-i)+\theta]}{\prod_{s\in \mu}[N\theta+(j-1)-\theta(i-1)]}\\
 \phantom{\qquad=}{}\cdot\frac{1}{\prod_{s\in \mu}[\mu_{i}-j+1+\theta (\mu_{j}'-i)]}P_{\mu}\bigl(a_{1}^{2},\dots,a_{N}^{2};\theta\bigr)P_{\mu}\bigl(x_{1}^{2},\dots,x_{N}^{2};\theta\bigr).
 \end{gather*}
 When taking the limit in the above way,
 \[\prod_{j=1}^{l(\mu)}\frac{(M\theta)^{\mu_{j}}}{[\theta(M-j+1)]\cdots [\theta(M-j+1)+\mu_{j}-1]}\longrightarrow 1,\]
and
 \[
 \frac{1}{\prod_{s\in \mu}[\mu_{i}-j+1+\theta (\mu_{j}'-i)]}\longrightarrow \prod_{j=1}^{l(\mu)}\frac{1}{\mu_{j}!}.
 \]
 Also note that $\mu_{i}-j+\theta (\mu_{j}'-i)+\theta$ goes to 0 only if $\mu_{i}-j=0$, and $N\theta +(j-1)-\theta(i-1)$ goes to 0 only if $j=1$. These terms give
 \[
 \prod_{i\ge 1}k_{i}!\cdot \frac{(N-l(\mu))!}{N!}=\prod_{i \ge 0}k_{i}!\cdot \frac{1}{N!},
 \]
 where $k_{i}$ denotes the number of rows in $\mu$ of length $i$. And the remaining part of
 \[
 \frac{\prod_{s\in\mu}[\mu_{i}-j+\theta (\mu_{j}'-i)+\theta]}{\prod_{s\in \mu}[N\theta+(j-1)-\theta(i-1)]}
 \]
 converges to 1. Together with~\eqref{eq_lln4}, we have the limit is equal to
 \[
 \sum_{\mu}\frac{\prod_{i\ge 0}k_{i}!}{N!}\frac{1}{\prod_{j=1}^{l(\mu)}\mu_{j}!}m_{\mu}\bigl(a_{1}^{2},\dots,a_{N}^{2}\bigr)m_{\mu}\bigl(x_{1}^{2},\dots,x_{N}^{2}\bigr)
 \]
 which is the Taylor expansion of
 \[
 \frac{1}{N!}\sum_{\sigma\in S_{N}}\prod_{i=1}^{N}{\rm e}^{a_{i}^{2}x_{\sigma(i)}^{2}}.\tag*{\qed}
 \] \renewcommand{\qed}{}
\end{proof}

\subsection*{Acknowledgements} The author is grateful to Vadim Gorin for a lot of stimulating discussions, and all his useful suggestions on the presentation of this paper. We thank Grigori Olshanski for pointing out a~useful reference, Simon Marshall for explaining some basics of symmetric spaces, and Margit Roesler for clarification of a technical issue in her lecture notes. We are also indebted to Florent Benaych-Georges, Pierre Mergny, Michael Voit and Cesar Cuenca for their valuable comments. We are grateful to the anonymous referees for their valuable comments. The project was partially supported by NSF grants DMS-1949820 and DMS-2152588.

\pdfbookmark[1]{References}{ref}
\LastPageEnding

\end{document}